\documentclass[11pt]{amsart}
\usepackage[utf8]{inputenc}
\usepackage{amsmath}
\usepackage{amssymb}
\usepackage{amsthm}
\usepackage{tikz}
\usepackage{setspace}
\usepackage{verbatim}
\usetikzlibrary{patterns,decorations.pathreplacing}
\usepackage[pdftex,all]{xy}
\usepackage[justification=centering]{caption}
\allowdisplaybreaks

\newtheorem{theorem}{Theorem}[section]
\newtheorem{definitio}[theorem]{Definition}

\newtheorem{lemma}[theorem]{Lemma}
\newtheorem{proposition}[theorem]{Proposition}
\newtheorem{corollary}[theorem]{Corollary}
\theoremstyle{remark}
\newtheorem*{remark}{Remark}
\newtheorem*{notation}{Notation}
\newtheorem*{notations}{Notations}
\newtheorem*{convnot}{Conventions and notations}
\newtheorem*{plan}{Plan of the paper}
\newtheorem*{acknowledgement}{Acknowledgement}

\definecolor{vert}{RGB}{0,205,0}

\def\fl#1{\smash{\mathop{\hbox to 11mm{ \rightarrowfill\ }}\limits^{\scriptstyle{#1}}}}

\newcommand{\N}{\mathbb{N}}
\newcommand{\Z}{\mathbb{Z}}
\newcommand{\Q}{\mathbb{Q}}
\newcommand{\K}{\mathbb{K}}
\newcommand{\R}{\mathbb{R}}

\newcommand{\Ztt}{\mathbb{Z}[t^{\pm1}]}
\newcommand{\Qtt}{\mathbb{Q}[t^{\pm1}]}
\newcommand{\Al}{\mathfrak{A}}
\newcommand{\bl}{\mathfrak{b}}
\newcommand{\un}{\mathcal{O}}

\newcommand{\iso}{\fl{\scriptstyle{\cong}}}

\newcommand{\tang}{\mathcal{T}_q}

\newcommand{\A}{\mathcal{A}}
\newcommand{\F}{\mathcal{F}}
\newcommand{\G}{\mathcal{G}}
\newcommand{\tA}{\widetilde{\mathcal{A}}}
\newcommand{\tAw}{\widetilde{\mathcal{A}}_\textrm{\textnormal{w}}}
\newcommand{\holw}{Hol$_\textrm{w}$}
\newcommand{\eqw}{\sim_\textnormal{w}}
\newcommand{\eql}{\sim_\ell}
\newcommand{\eqlgl}{\sim_{g\ell}}
\newcommand{\xd}{\,\raisebox{-0.5ex}{\begin{tikzpicture} \draw[->] (0,0) -- (0,0.4);\end{tikzpicture}}\,}
\newcommand{\xdop}{\,\raisebox{-0.5ex}{\begin{tikzpicture} \draw[<-] (0,0) -- (0,0.4);\end{tikzpicture}}\,}
\newcommand{\xc}{\raisebox{-0.2ex}{\begin{tikzpicture} \draw (0,0) circle (0.16); \draw[->] (0.16,0.03) -- (0.16,0.04);\end{tikzpicture}}\,}

\newcommand{\xxlw}{\ooalign{$\bigcirc$\cr\hidewidth $\vcenter{\hbox{\smaller\smaller\smaller$\bigcirc$}}$\hidewidth\cr\hidewidth$\ast$\hidewidth}^{}}
\newcommand{\Zp}{Z^\bullet}
\newcommand{\Zc}{Z^\circ}
\newcommand{\tZ}{\tilde{Z}}
\newcommand{\expd}{\textrm{\textnormal{exp}}_\sqcup}
\newcommand{\lk}{\textrm{\textnormal{lk}}}

\newcommand{\hW}{\widehat{W}}
\newcommand{\LMO}{{\scalebox{0.6}{LMO}}}

\newcommand{\Kri}{{\scalebox{0.6}{Kri}}}
\newcommand{\Les}{{\scalebox{0.6}{Les}}}
\newcommand{\aug}{{\scalebox{0.6}{aug}}}
\newcommand{\z}{{\scalebox{0.6}{$\Z$}}}
\newcommand{\DD}{\mathcal D}
\newcommand{\SK}{\textnormal{SK}}
\newcommand{\Int}{\mathrm{Int}}

\newcommand{\fract}[2]{\hbox{\leavevmode \kern.1em \raise .25ex \hbox{\the\scriptfont0 $#1$}\kern-.1em }\big/  {\hbox{\kern-.15em \lower .5ex \hbox{\the\scriptfont0 $#2$}} }}

\newcommand{\draww}[1]{\draw[white,line width=5pt] #1 \draw #1}

\newcommand{\struts}[3]{
\raisebox{-0.5cm}{
\begin{tikzpicture} 
\draw[->] (0,0) node[right] {$\scriptstyle #1$} -- (0,0.4) node[left] {$\scriptstyle #3$};
\draw (0,0.4) -- (0,0.8) node[right] {$\scriptstyle #2$};
\end{tikzpicture}}}

\title{A universal finite type invariant of knots in homology $3$--spheres}

\author{Benjamin Audoux}
\author{Delphine Moussard}

\begin{document}

\begin{abstract}
 An essential goal in the study of finite type invariants of some objects (knots, manifolds) is the construction of a universal finite type invariant, universal in the sense that it contains all finite type invariants of the given objects. Such a universal finite type invariant is known for knots in the $3$--sphere ---the Kontsevich integral--- and for homology $3$--spheres ---the Le--Murakami--Ohtsuki invariant. For knots in homology $3$--spheres, an invariant constructed by Garoufalidis and Kricker as a lift of the Kontsevich integral has been considered for the last two decades as the best candidate to be a universal finite type invariant. Although this invariant is eventually universal in restriction to knots whose Alexander polynomial is trivial, we prove here that it is not powerful enough in general. For that we provide a refinement of its construction which produces a strictly stronger invariant, and we prove that this new invariant is a universal finite type invariant of knots in homology $3$--spheres. This provides a full diagrammatic description of the graded space of finite type invariants of knots in homology $3$--spheres.
\end{abstract}

\maketitle

\tableofcontents

\section{Introduction}

Given a set of topological objects, in general knots or manifolds, and an operation on them, like crossing change for knots or some surgery on manifolds, the finite type invariants for these objects are defined by their behaviour with respect to this operation.
The notion of finite type invariants was introduced independently by Goussarov and Vassiliev in the early 90's for the study of invariants of knots in $S^3$. It was extended by Ohtsuki to invariants of integral homology $3$--spheres \cite{Oht4}. Then different theories were developed, including in particular a theory due to Goussarov and Habiro independently which applies to invariants of all $3$--manifolds and their knots \cite{Gou,Hab}. Garoufalidis, Kricker and Rozansky built upon the latter to study invariants of knots in homology $3$--spheres \cite{GK,GR}. 

To a theory of finite type invariants, {\em ie} a set of objects and an operation, one associates the graded space defined by a %the topological
filtration, provided by the operation, on the vector space generated by the objects; its dual is the space of finite type invariants, graded by their degree. In the study of such a theory, the grail is to obtain a combinatorial description of this graded space, by identifying it with a graded space of Feynman diagrams; this requires the construction of a universal finite type invariant. This has been achieved by Bar-Natan and Kontsevich for knots in $S^3$ using the Kontsevich integral \cite{Kon,BN}. For integral homology $3$--spheres in the Goussarov--Habiro theory, the description of the graded space follows from works of Garoufalidis--Goussarov--Polyak \cite{GGP}, Habiro \cite{Hab} and Le \cite{Le}, with the LMO invariant of Le--Murakami--Ohtsuki as universal finite type invariant. This was generalized to rational homology $3$--spheres by works of Massuyeau \cite{Mas} and the second author \cite{M2}. In their study of finite type invariants of knots in homology $3$--spheres, Garoufalidis--Kricker--Rozansky described the graded space in the case of knots whose Alexander polynomial is trivial with, as universal invariant, the Kricker invariant defined in \cite{GK} following a first construction of \cite{Kri}. A graded space of diagrams suitable for knots with any Alexander polynomial was proposed in \cite{M7}, but the universal invariant that would make the description complete was still missing. Here, we construct a refinement of the Kricker invariant and we prove that it is a universal finite type invariant of knots in homology $3$--spheres. Since this refinement is strict, it shows in particular that the Kricker invariant is not universal.

The universal invariants that appear in these theories all derive from the Kontsevich integral of knots and links in $S^3$ in some sense. The constructions of the LMO invariant and the Kricker invariant apply the Kontsevich integral to surgery presentations of the manifold or the pair (manifold, knot). Another approach is to extend directly the idea of the Kontsevich integral: in this case, the invariants of $3$--manifolds and their knots are obtained by computing integrals in associated configuration spaces. This approach has led to the KKT invariant of Kontsevich--Kuperberg--Thurston for homology $3$--spheres \cite{Kon2,KT,Les5} and to the Lescop invariant for knots in homology $3$--spheres \cite{Les2}. These two methods are expected to produce equivalent invariants, but no direct comparison is known for the LMO and KKT invariants, nor for the Lescop and Kricker invariants (meaning that we are not able to express one invariant in terms of the other). However, these invariants may be compared through finite type invariants theories. 
The LMO and KKT invariants are both universal finite type invariants of integral homology $3$--spheres, so that they are equivalent (in the sense that they distinguish the same pairs of manifolds) \cite{Le,Les}; this also holds for rational homology $3$--spheres with homology groups of a given cardinality \cite{Mas,M2}. Similarly, the equivalence of the Kricker and Lescop invariants for knots in homology $3$--spheres with trivial Alexander polynomial stems from \cite{GK,M7}. In \cite{Les3}, Lescop conjectured that this equivalence holds for all null-homologous knots in rational homology $3$--spheres. The description of the graded space given here is a great step toward this conjecture.

\begin{notations}
 For $\K=\Z,\Q$, a {\em $\K$--sphere} is a closed oriented $3$--manifold which has the same homology with coefficients in $\K$ as the standard $3$--sphere. A {\em $\K\SK$--pair} is a pair $(M,K)$ where $M$ is a $\K$--sphere and $K$ is an oriented knot in $M$ with trivial homology class in $H_1(M;\Z)$. We consider $\K$--spheres and $\K$SK--pairs up to orientation-preserving homeomorphism (of pairs when appropriate).
\end{notations}

We consider finite type invariants of $\Q\SK$--pairs with respect to a surgery move called null LP--surgery. It turns out that the classes of $\Q\SK$--pairs up to null LP--surgeries are classified by their Blanchfield modules, namely their Alexander modules equipped with the Blanchfield form \cite{M3}. Hence we fix a Blanchfield module $(\Al,\bl)$ and we work with $\Q\SK$--pairs whose Blanchfield module is isomorphic to $(\Al,\bl)$. We denote by $\G(\Al,\bl)$ the graded space associated to the finite type theory for null LP--surgery on those $\Q\SK$--pairs (see the end of Section \ref{subsecsurgeries} for a precise definition). In \cite{M7}, a graded space of diagrams $\A^\aug(\Al,\bl)$ was constructed, together with a canonical onto map $\varphi:\A^\aug(\Al,\bl)\twoheadrightarrow\G(\Al,\bl)$. The superscript ``aug'' refers to diagrams that are ``augmented'' with isolated vertices, as detailed in the next section.

\begin{figure} [htb]
\begin{center}
\begin{tikzpicture} [xscale=1.5,yscale=0.5]
 \draw (2.2,2) node {$\A^\aug(\Al,\bl)$};
 \draw (2,-2) node {$\A^\aug(\delta)$};
 \draw (0.2,0) node {$\G(\Al,\bl)$};
 \draw[->>] (1.5,1.6) -- (0.7,0.45); \draw (1.2,0.6) node {$\varphi$};
 \draw[->] (0.7,-0.45) -- (1.5,-1.6); \draw (1,-1.7) node {$Z^\aug$};
 \draw[->] (2,1.3) -- (2,-1.3); \draw (2.4,0) node {$\psi$};
 \draw[->] (0.7,0.75) -- (1.5,1.9); \draw (1,1.8) node {$\tilde Z^\aug$};
\end{tikzpicture}
\end{center}
\caption{Invariants and maps on the graded space\\ {\footnotesize Here, $Z^{\rm \scriptscriptstyle aug}$ is the augmented version of either $Z^{\rm \scriptscriptstyle Kri}$ or $Z^{\rm \scriptscriptstyle Les}$.}}
\label{figcommutativediagram}
\end{figure}

The Kricker invariant $Z^\Kri$ and the Lescop invariant $Z^\Les$ are valued in a graded space of diagram $\A(\delta)$ which depends on the annihilator $\delta$ of the Alexander module $\Al$. They can be written as series of finite type invariants and they induce maps on $\G(\Al,\bl)$ \cite{Les3,M6}. These invariants can be completed into an invariant with values in an ``augmented'' diagram space $\A^\aug(\delta)$. The composition with the map $\varphi$ is in both cases the same map $\psi:\A^\aug(\Al,\bl)\to\A^\aug(\delta)$. However, we proved in \cite{AM} that the map $\psi$ is not injective in general. The main goal of this article is to refine the Kricker invariant into an invariant $\tZ$ with values in $\A(\Al,\bl)$, so that its augmented version provides the inverse of the map $\varphi$. This will show that the graded space $\G(\Al,\bl)$ is naturally isomorphic to $\A^\aug(\Al,\bl)$. Hence, as it realizes the isomorphism, our invariant is a universal finite type invariant for $\Q\SK$--pairs with respect to null LP--surgeries, meaning that all finite type invariants can be recovered from it.

\begin{theorem}[Theorem~\ref{thtZauguniv}]
 There is an invariant $\tZ^\aug$ of $\Q\SK$--pairs which induces the inverse of the map $\varphi:\A^\aug(\Al,\bl)\to\G(\Al,\bl)$.
\end{theorem}

The invariant $\tZ$ is a lift of the Kricker invariant. The fact that the map $\psi$ is not always injective, while $\varphi$ is an isomorphism, proves that $\tZ$ strictly contains $Z^\Kri$.

Moreover, the description of the graded space $\G(\Al,\bl)$ for all Blanchfield modules shows that the Kricker invariant and the Lescop invariant induce the same map on $\G(\Al,\bl)$. A refinement of the Lescop invariant similar to the one we construct for the Kricker invariant would allow to prove the Lescop conjecture.

One can focus on knots in $\Z$--spheres, or $\Z\SK$--pairs.
There is a natural $\Z$--version of the null LP--surgery move, which provides a similar notion of finite type invariants. All the results of this paper can be transposed to this setting. In particular, the analogous result to the above theorem is given in Theorem~\ref{thZZuniv}.

\begin{plan}
 In Section~\ref{secbackground}, we give definitions and notations and we recall some background. In Section~\ref{secwinding}, we introduce the surgery presentations that we will use to define the invariant and we give a combinatorial computation of the associated equivariant linking matrices. Section~\ref{secdiagrams} is devoted to diagram spaces and diagrammatic operations that underly the construction of the invariant. Section~\ref{secinvariant} gives the construction of the invariant $\tZ$. In Section~\ref{secuniversality}, we describe the behaviour of our invariant with respect to null LP--surgeries and we establish universality. Finally, Section~\ref{secZpairs} adapts the results to the setting of knots in $\Z$--spheres.
\end{plan}

\begin{convnot}\ \\
 The boundary of an oriented manifold is oriented with the outward normal first convention. \\
 For submanifolds $X$ and $Y$ of a manifold $M$ such that $\dim(X)+\dim(Y)=\dim(M)$, $\langle X,Y\rangle$ denotes the algebraic intersection number of $X$ and $Y$.\\
 For $\K=\Z,\Q$, a {\em genus--$g$ $\K$--handlebody} is a compact $3$--manifold which has the same homology with coefficients in $\K$ as the standard genus--$g$ handlebody.\\
 Given a matrix $W$, $^tW$ denotes the transpose of $W$.\\
 For a matrix $A(t)$ with polynomial coefficients, we set $\bar A(t)=A(t^{-1})$. \\
 We denote by $\expd$ the exponential of diagrams with respect to the disjoint union. 
\end{convnot}

\begin{acknowledgement}
  The authors thank deeply the anonymous referee for its numerous comments which were very helpful to improve the clarity of the paper.  
\end{acknowledgement}

\section{Finite type invariants and diagram spaces}
\label{secbackground}

\subsection{Filtration defined by null LP--surgeries} \label{subsecsurgeries}

We first recall the definition of the Alexander module and the Blanchfield form. 

Let $(M,K)$ be a $\Q\SK$--pair. The \emph{exterior} $X$ of $K$ is the complement in $M$ of an open tubular neighborhood of $K$. Let $\tilde{X}$ be the \emph{infinite cyclic covering} of $X$, {\em ie} the covering associated with the kernel of the map $\pi_1(X)\to\Z=\langle t\rangle$ which sends the positive meridian of $K$ to $t$. The automorphism group of the covering is isomorphic to $\Z$; let $\tau$ be the generator associated with the action of the positive meridian. 
Denoting the action of $\tau$ as the multiplication by $t$, 
we get a structure of $\Qtt$--module on $\Al(M,K)=H_1(\tilde{X};\Q)$. This $\Qtt$--module is called the \emph{Alexander module} of $(M,K)$. 
It is a finitely generated torsion $\Qtt$--module. 

Let $\delta\in\Qtt$ be the annihilator of~$\Al$ (it is defined up to an invertible element of $\Qtt$, which has no importance here). 
Given two disjoint knots $J_1$ and $J_2$ in $\tilde{X}$, define the {\em equivariant linking number} of $J_1$ and~$J_2$ as: 
$$\lk_e(J_1,J_2)=\frac{1}{\delta(t)}\sum_{k\in\Z}\big\langle S,\tau^k(J_2)\big\rangle\, t^k,$$
where $S$ is a rational 2--chain in $\tilde X$ such that $\partial S=\delta(\tau)J_1$ and $\langle\cdot,\cdot\rangle$ is the intersection pairing in~$\tilde X$.
It is a well-defined element of $\frac{1}{\delta(t)}\Qtt$ which satisfies $\lk_e(J_2,J_1)(t)=\lk_e(J_1,J_2)(t^{-1})$ and 
$\lk_e(P(\tau)J_1,J_2)(t)=P(t)\lk_e(J_1,J_2)(t)$.
Now define the \emph{Blanchfield form} $\bl : \Al\times\Al \to \frac{\Q(t)}{\Qtt}$ as follows: if $\gamma$ ({\em resp} $\eta$) is the homology class of $J_1$ ({\em resp} $J_2$) in $\Al$, set
$$\bl(\gamma,\eta)=\lk_e(J_1,J_2)\ mod\ \Qtt.$$
The Blanchfield form is {\em hermitian}: $\bl(\gamma,\eta)(t)=\bl(\eta,\gamma)(t^{-1})$ and $\bl(P(t)\gamma,\eta)(t)=P(t)\,\bl(\gamma,\eta)(t)$ 
for all $\gamma,\eta\in\Al$ and all $P\in\Qtt$. 
Moreover, it is {\em non degenerate} (see Blanchfield in \cite{Bla}): $\bl(\gamma,\eta)=0$ for all $\eta\in\Al$ implies $\gamma=0$. 

The Alexander module of a $\Q\SK$--pair $(M,K)$ endowed with its Blanchfield form is its {\em Blanchfield module} denoted by $(\Al,\bl)(M,K)$. 
In the sequel, by {\em a Blanchfield module $(\Al,\bl)$}, we mean a pair $(\Al,\bl)$ that can be realized as the Blanchfield module of a $\Q\SK$--pair. 
An isomorphism between Blanchfield modules is an isomorphism between the underlying Alexander modules which preserves the Blanchfield form.

\medskip

We now define LP--surgeries. 
Note that the boundary of a genus--$g$ $\Q$--handlebody is homeomorphic to the standard genus--$g$ surface. 
The \emph{Lagrangian} $\mathcal{L}_A$ of a $\Q$--handlebody $A$ is the kernel of the map $i_*: H_1(\partial A;\Q)\to H_1(A;\Q)$ induced by the inclusion; it is indeed a Lagrangian subspace of $H_1(\partial A;\Q)$ with respect to the intersection form. 
Two $\Q$--handlebodies $A$ and $B$ have \emph{LP--identified} boundaries if $(A,B)$ is equipped with a homeomorphism $h:\partial A\to\partial B$ such that $h_*(\mathcal{L}_A)=\mathcal{L}_B$.

Let $M$ be a $\Q$--sphere, let $A\subset M$ be a $\Q$--handlebody and let $B$ be a $\Q$--handlebody whose boundary is LP--identified with $\partial A$. Set $M\left(\frac{B}{A}\right)=(M\setminus Int(A))\cup_{\partial A=_h\partial B}B$. We say that the $\Q$--sphere $M\left(\frac{B}{A}\right)$ is obtained from $M$ by the \emph{Lagrangian-preserving surgery}, or \emph{LP--surgery}, $\left(\frac{B}{A}\right)$.
 
Given a $\Q\SK$--pair $(M,K)$, a subset $A\subset M\setminus K$ is {\em null in $M\setminus K$} if the map $i_* : H_1(A;\Q)\to H_1(M\setminus K;\Q)$ induced by the inclusion has a trivial image.
A \emph{null LP--surgery} on $(M,K)$ is an LP--surgery $\left(\frac{B}{A}\right)$ such that $A$ is null in $M\setminus K$. 
The $\Q\SK$--pair obtained by surgery is denoted by $(M,K)\left(\frac{B}{A}\right)$. 

\medskip

We can now define a filtration on the rational vector space $\F_0$ generated by all $\Q\SK$--pairs up to orientation-preser\-ving homeomorphism. 
For $n\geq1$, we define $\F_n$ as the subspace of $\F_0$ generated by the {\em brackets}
$$\left[(M,K);\left(\frac{B_i}{A_i}\right)_{1\leq i \leq n}\right]=\sum_{I\subset \{ 1,...,n\}} (-1)^{|I|} (M,K)\left(\left(\frac{B_i}{A_i}\right)_{i\in I}\right)$$ 
for all $\Q\SK$--pairs $(M,K)$ and all families of $\Q$--handlebodies $(A_i,B_i)_{1\leq i \leq n}$, where the $A_i$ are null in $M\setminus K$ and disjoint, and each pair $(A_i,B_i)$ has LP--identified boundaries. Note that $\F_{n+1}\subset \F_n$. 
An invariant $\lambda$ of $\Q\SK$--pairs valued in some rational vector space is a \emph{finite type invariant of degree at most $n$ with respect to null LP--surgeries} if its $\Q$--linear extension to $\F_0$ satisfies $\lambda(\F_{n+1})=0$. 

As proven in \cite[Theorem 1.14]{M3}, two $\Q\SK$--pairs are related by null LP--surgeries if and only if their Blanchfield modules are isomorphic. This allows to work with $\Q\SK$--pairs with a given Blanchfield module. 
Hence, we fix an abstract Blanchfield module $(\Al,\bl)$, namely a pair $(\Al,\bl)$ which is isomorphic to the Blanchfield module of some $\Q\SK$--pair (these pairs have been characterized in \cite[Proposition~1.2 $\&$ Theorem~1.4]{M1}). Then we consider the subspace  $\F_0(\Al,\bl)$ of $\F_0$ generated by the $\Q\SK$--pairs whose Blanchfield module is isomorphic to $(\Al,\bl)$. Let $\big(\F_n(\Al,\bl)\big)_{n\in\N}$ be the filtration defined on $\F_0(\Al,\bl)$ by null LP--surgeries. Then, for $n\in\N$, $\F_n$ is the direct sum, over all isomorphism classes $(\Al,\bl)$ of Blanchfield modules, of the $\F_n(\Al,\bl)$. Set $\G_n(\Al,\bl)=\F_n(\Al,\bl) / \F_{n+1}(\Al,\bl)$ and $\G(\Al,\bl)=\prod_{n\in\N}\G_n(\Al,\bl)$. Our goal is to describe the graded space $\G(\Al,\bl)$; note that $\G_0(\Al,\bl)\cong\Q$. 

\subsection{Borromean surgeries}

We now introduce a specific type of LP--surgeries.
The {\em standard Y--graph} is the graph $\Gamma_0\subset \mathbb{R}^2$ represented in Figure \ref{figY4}. 
The looped edges of $\Gamma_0$ are called \emph{leaves} and the vertex incident to three different edges is called the {\em internal vertex}. 
With $\Gamma_0$ is associated a regular neighborhood $\Sigma(\Gamma_0)$ of $\Gamma_0$ in the plane. 
The surface $\Sigma(\Gamma_0)$ is oriented with the usual convention. This induces an orientation of the leaves and an {\em orientation} of the internal vertex, {\em ie} a cyclic order of the three edges, shown in Figure \ref{figY4}. 
Consider a 3--manifold $M$ and an embedding $h:\Sigma(\Gamma_0)\to \Int(M)$. The image $\Gamma$ of $\Gamma_0$ is a {\em Y--graph}, endowed with its {\em associated surface} $\Sigma(\Gamma)=h(\Sigma(\Gamma_0))$. The Y--graph $\Gamma$ is equipped with the framing induced by~$\Sigma(\Gamma)$. 
A \emph{Y--link} in $M$ is a collection of Y--graphs in $M$ whose associated surfaces are disjoint. 

\begin{figure}[htb] 
\begin{center}
\begin{tikzpicture} [scale=0.15]
\newcommand{\feuille}[1]{
\draw[rotate=#1,thick,color=gray] (0,-11) circle (5);
\draw[rotate=#1,thick,color=gray] (0,-11) circle (1);
\draw[rotate=#1,line width=8pt,color=white] (-2,-6.42) -- (2,-6.42);
\draw[rotate=#1,thick,color=gray] (2,-1.15) -- (2,-6.42);
\draw[rotate=#1,thick,color=gray] (-2,-1.15) -- (-2,-6.42);
\draw[white,rotate=#1,line width=5pt] (0,0) -- (0,-8);
\draw[rotate=#1] (0,0) -- (0,-8);
\draw[rotate=#1] (0,-11) circle (3);
\draw[->,rotate=#1] (-3,-10.9) -- (-3,-11.1);}
\draw (0,0) circle (1.5);
\draw[->] (0.1,1.5) -- (-0.1,1.5);
\feuille{0}
\feuille{120}
\feuille{-120}
\draw (-4,10) node{$\scriptstyle{\textrm{leaf}}$};
\draw[->] (-5,9) -- (-6.3,7.5);
\draw (11.5,-1) node{$\scriptstyle{\textrm{internal vertex}}$};
\draw[<-] (0.5,-0.1) -- (4,-1);
\draw (4.7,-9.2) node{$\Gamma_0$};
\draw[color=gray] (6.5,-16.8) node{$\Sigma(\Gamma_0)$};
\end{tikzpicture}
\end{center}
\caption{The standard Y--graph and its associated surface}\label{figY4}
\end{figure}

\begin{figure}[htb] 
\begin{center}
\begin{tikzpicture} [scale=0.15]
\begin{scope}
\newcommand{\feuille}[1]{
\draw[rotate=#1] (0,0) -- (0,-8);
\draw[rotate=#1] (0,-11) circle (3);}
\feuille{0}
\feuille{120}
\feuille{-120}
\draw (3,-4) node{$\Gamma$};
\end{scope}
\draw[very thick,->] (21.5,-3) -- (23.5,-3);
\begin{scope}[xshift=1200]
\newcommand{\bras}[1]{
\draw[rotate=#1] (0,-1.5) circle (2.5);
\draw [rotate=#1,white,line width=8pt] (-0.95,-4) -- (0.95,-4);
\draw[rotate=#1] {(0,-11) circle (3) (1,-3.9) -- (1,-7.6)};
\draw[rotate=#1,white,line width=6pt] (-1,-5) -- (-1,-8.7);
\draw[rotate=#1] {(-1,-3.9) -- (-1,-8.7) (-1,-8.7) arc (-180:0:1)};}
\bras{0}
\draw [white,line width=6pt,rotate=120] (0,-1.5) circle (2.5);
\bras{120}
\draw [rotate=-120,white,line width=6pt] (-1.77,0.27) arc (135:190:2.5);
\draw [rotate=-120,white,line width=6pt] (1.77,0.27) arc (45:90:2.5);
\bras{-120}
\draw [white,line width=6pt] (-1.77,0.27) arc (135:190:2.5);
\draw [white,line width=6pt] (1.77,0.27) arc (45:90:2.5);
\draw (-1.77,0.27) arc (135:190:2.5);
\draw (1.77,0.27) arc (45:90:2.5);
\draw (3.5,-4.5) node{$L$};
\end{scope}
\end{tikzpicture}
\end{center}
\caption{Y--graph and associated surgery link}\label{figborro4}
\end{figure}

Let $\Gamma$ be a Y--graph in a 3--manifold $M$ and let $\Sigma$ be its associated surface. Embed $\Sigma\times[-1,1]$ in $M$ so that $\Sigma\times\{0\}$ coincides with $\Sigma$, and in this $\Sigma\times[-1,1]$, associate with $\Gamma$ the six components link $L$ represented in Figure \ref{figborro4}. 
The \emph{borromean surgery on~$\Gamma$} is the surgery along the framed link~$L$; the surgered manifold is denoted $M(\Gamma)$. 
Borromean surgeries were introduced and studied by Matveev in \cite{Mat}. He proved in particular that two $3$--manifolds are related by a sequence of borromean surgeries if and only if they have isomorphic first homology group and linking pairing.
Moreover, he showed that a borromean surgery can be realized by cutting a genus--$3$ handlebody (a regular neighborhood of the Y--graph) and regluing it in another way, which preserves the Lagrangian. It follows that borromean surgeries are specific LP--surgeries.

\subsection{Colored Jacobi diagrams}

Fix a Blanchfield module $(\Al,\bl)$ and let $\delta\in\Qtt$ be the annihilator of $\Al$. 
An {\em $(\Al,\bl)$--colored diagram} $D$ is a unitrivalent graph without {\em strut} (isolated edge), with the following data:
\begin{itemize}
\item trivalent vertices are oriented (an {\em orientation of a trivalent vertex} is a cyclic order of the three half-edges that meet at this vertex; by convention, we fix it as \raisebox{-1.5ex}{
 \begin{tikzpicture} [scale=0.2]
  \newcommand{\tiers}[1]{
  \draw[rotate=#1,color=white,line width=4pt] (0,0) -- (0,-2);
  \draw[rotate=#1] (0,0) -- (0,-2);}
  \draw (0,0) circle (1);
  \draw[<-] (-0.05,1) -- (0.05,1);
  \tiers{0}
  \tiers{120}
  \tiers{-120}
 \end{tikzpicture}}
 in the pictures),
 \item edges are oriented and labeled by elements of $\Qtt$;
 \item univalent vertices are labeled by elements of $\Al$;
 \item for all pair $(v,v')$ in the set $V$ of univalent vertices of $D$, with $v\neq v'$, a rational fraction $f_{vv'}^D(t)\in\frac{1}{\delta(t)}{\Qtt}$ is fixed such that $f_{vv'}^D(t)\ mod\ \Qtt=\bl(\gamma,\gamma')$, where $\gamma$ ({\em resp} $\gamma'$) is the coloring of $v$ ({\em resp} $v'$); we require that $f_{v'v}^D(t)=f_{vv'}^D(t^{-1})$.
\end{itemize}
When it does not seem to cause confusion, we write $f_{vv'}$ for $f_{vv'}^D$. 
The {\em degree} of a colored diagram is the number of trivalent vertices of its underlying graph. The unique degree 0 diagram is the empty diagram. 
For $n\geq0$, set: 
$$\hat\A_n(\Al,\bl)=\frac{\Q \langle(\Al,\bl)\mbox{--colored diagrams of degree $n$}\rangle}{\Q \langle\mbox{AS, IHX, LE, OR, Hol, LV, EV, LD}\rangle},$$ 
where the relations are described in Figures~\ref{figrelAStoHol} and \ref{figrelLVEVLD}. Throughout the article, the notation $\Q\langle-\rangle$ means the space of all finite linear combinations of the elements in the brackets.

\begin{figure}[htb] 
\begin{center}
\begin{tikzpicture} [scale=0.3]
%AS
\begin{scope} [xshift=0cm]
 \draw (0,4) -- (2,2);
 \draw (2,2) -- (4,4);
 \draw (2,2) -- (2,0);
 \draw (5,2) node{$+$};
 \draw (8,2) .. controls +(2,0) and +(2.5,-1) .. (6,4);
 \draw[white,line width=6pt] (8,2) .. controls +(-2,0) and +(-2.5,-1) .. (10,4);
 \draw (8,2) .. controls +(-2,0) and +(-2.5,-1) .. (10,4);
 \draw (8,0) -- (8,2);
 \draw (11,2) node{$=$};
 \draw (12.5,2) node{0};
 \draw (6,-1.5) node{AS};
\end{scope}
%IHX
\begin{scope} [xshift=-1cm]
 \draw (18,4) -- (20,3) -- (20,1) -- (18,0);
 \draw (20,1) -- (22,0);
 \draw (20,3) -- (22,4);
 \draw (20,2) node[left] {$\scriptstyle{1}$};
 \draw (23,2) node{$-$};
 \draw (24,4) -- (25,2) -- (27,2) -- (28,4);
 \draw (24,0) -- (25,2);
 \draw (27,2) -- (28,0);
 \draw (26,2) node[above] {$\scriptstyle{1}$};
 \draw (29,2) node{$+$};
 \draw (30,4) -- (33,2) -- (34,0);
 \draw[white,line width=6pt] (31,2) -- (34,4);
 \draw (30,0) -- (31,2) -- (34,4);
 \draw (31,2) -- (33,2);
 \draw (32,2) node[below] {$\scriptstyle{1}$};
 \draw (35,2) node{$=$};
 \draw (36.5,2) node{0};
 \draw (27,-1.5) node{IHX};
\end{scope}
%OR
\begin{scope} [xshift=40cm]
 \draw (0,0) -- (0,4);
 \draw[->] (0,2.9) -- (0,3) node[right] {$P(t)$};
 \draw (4,2) node{$=$};
 \draw (5.5,0) -- (5.5,4);
 \draw[->] (5.5,3.1) -- (5.5,3) node[right] {$P(t^{-1})$};
 \draw (4,-1.5) node{OR};
\end{scope}
%LE
\begin{scope} [yshift=-10cm]
 \draw (0,2) node{$x$};
 \draw (1,0) -- (1,4);
 \draw[->] (1,0) -- (1,3);
 \draw (1.8,2.8) node{$P$};
 \draw (3.2,2) node{$+$};
 \draw (4.7,2) node{$y$};
 \draw (5.7,0) -- (5.7,4);
 \draw[->] (5.7,0) -- (5.7,3);
 \draw (6.5,2.8) node{$Q$};
 \draw (7.9,2) node{$=$};
 \draw (9.4,0) -- (9.4,4);
 \draw[->] (9.4,0) -- (9.4,3);
 \draw (12.4,2.8) node{$xP+yQ$};
 \draw (6.5,-1.5) node{LE};
\end{scope}
%Hol
\begin{scope} [xshift=22cm,yshift=-9cm,scale=0.9]
 \newcommand{\edge}[1]{
 \draw[rotate=#1] (0,0) -- (0,3);
 \draw[rotate=#1,->] (0,0) -- (0,1.5);}
 \edge{0} \draw (0,1.5) node[right] {$P$};
 \edge{120} \draw (-1.5,-0.75) node[above] {$Q$};
 \edge{240} \draw (1,-0.9) node[above right] {$R$};
 \draw (4,0) node{$=$};
\begin{scope} [xshift=8cm]
 \edge{0} \draw (0,1.5) node[right] {$tP$};
 \edge{120} \draw (-1.5,-0.75) node[above] {$tQ$};
 \edge{240} \draw (1,-0.9) node[above right] {$tR$};
\end{scope}
 \draw (14,-3.5) node{Hol};
\end{scope}
%Hol'
\begin{scope} [xshift=40cm,yshift=-9cm,scale=0.9]
 \foreach \x in {0,7} {
 \draw (\x,0) -- (\x,3);
 \draw (\x,-1) circle (1);
 \draw[->] (\x-1,-1) -- (\x-1,-1.1) node[left] {$Q$};
 \draw[->] (\x,0) -- (\x,1.5);}
 \draw (0,1.3) node[right] {$P$};
 \draw (3,0) node{$=$};
 \draw (7,1.3) node[right] {$tP$};
 %\draw (4,-3.5) node{Hol'};
\end{scope}
\end{tikzpicture}
\end{center} \caption{Relations, where $x,y\in\Q$ and $P,Q,R\in\Qtt$} \label{figrelAStoHol}
\end{figure}

\begin{figure}[htb] 
\begin{center}
\begin{tikzpicture} [scale=0.3]
%LV
\begin{scope} [xshift=-4cm,yshift=-17cm]
\draw (-1,2) node{$x$};
\draw (1,0) -- (1,4);
\draw (1,1) node[right] {$D_1$};
\draw (1,4) node{$\bullet$};
\draw (1,4) node[right] {$\gamma_1$};
\draw (1,4) node[below left] {$\scriptstyle{v}$};
\draw (4,2) node{$+$};
\draw (5.5,2) node{$y$};
\draw (7.5,0) -- (7.5,4);
\draw (7.5,1) node[right] {$D_2$};
\draw (7.5,4) node{$\bullet$};
\draw (7.5,4) node[right] {$\gamma_2$};
\draw (7.5,4) node[below left] {$\scriptstyle{v}$};
\draw (10.5,2) node{$=$};
\draw (13,0) -- (13,4);
\draw (13,1) node[right] {$D$};
\draw (13,4) node{$\bullet$};
\draw (13,4) node[right] {$x\gamma_1+y\gamma_2$};
\draw (13,4) node[below left] {$\scriptstyle{v}$};
\draw (8.5,-2) node{$xf^{D_1}_{vv'}(t)+yf^{D_2}_{vv'}(t)=f^{D}_{vv'}(t)\quad \forall\, v'\neq v$};
\draw (7,-4.5) node{LV};
\end{scope}
%EV
\begin{scope} [xshift=22cm,yshift=-17cm]
\draw (-0.8,0) -- (-0.8,4);
\draw (-0.8,4) node{$\bullet$};
\draw (-0.8,4) node[below left] {$\scriptstyle{v}$};
\draw[->] (-0.8,0) -- (-0.8,2);
\draw (-0.8,4) node[right] {$\gamma$};
\draw (0.6,1.8) node{$PQ$};
\draw (-0.8,0) node[right] {$D$};
\draw (3,2) node{$=$};
\draw (5,0) -- (5,4);
\draw (5,4) node{$\bullet$};
\draw (5,4) node[below left] {$\scriptstyle{v}$};
\draw[->] (5,0) -- (5,2);
\draw (5,4) node[right] {$Q(t)\gamma$};
\draw (6,1.8) node{$P$};
\draw (5,0) node[right] {$D'$};
\draw (10,3) node[right] {$f_{vv'}^{D'}(t)=Q(t)f_{vv'}^D(t)$};
\draw (12,1) node[right] {$\forall\, v'\neq v$};
\draw (7,-2) node{EV};
\end{scope}
%LD
\begin{scope} [xshift=10cm,yshift=-29cm]
\draw (-1,0) -- (-1,4);
\draw (-1,2) node[right] {$1$};
\draw (-1,4) node{$\bullet$};
\draw (-1,4) node[below left] {$\scriptstyle{v_1}$};
\draw (-1,4) node[right] {$\gamma_1$};
\draw (2.5,0) -- (2.5,4);
\draw (2.5,2) node[right] {$1$};
\draw (2.5,4) node{$\bullet$};
\draw (2.5,4) node[below left] {$\scriptstyle{v_2}$};
\draw (2.5,4) node[right] {$\gamma_2$};
\draw (2.5,0) node[right] {$D$};
\draw (5.5,2) node{$=$};
\draw (8,0) -- (8,4);
\draw (8,2) node[right] {$1$};
\draw (8,4) node{$\bullet$};
\draw (8,4) node[below left] {$\scriptstyle{v_1}$};
\draw (8,4) node[right] {$\gamma_1$};
\draw (11.5,0) -- (11.5,4);
\draw (11.5,2) node[right] {$1$};
\draw (11.5,4) node{$\bullet$};
\draw (11.5,4) node[below left] {$\scriptstyle{v_2}$};
\draw (11.5,4) node[right] {$\gamma_2$};
\draw (11.5,0) node[right] {$D'$};
\draw (14.5,2) node{$+$};
\draw (16,0) -- (16,2.75) (18.5,0) -- (18.5,2.75) arc (0:180:1.25);
\draw[->] (17.2,4) -- (17.3,4);
\draw (17.25,4) node[above] {$P$};
\draw (18.5,0) node[right] {$D''$};
\draw (7.5,-2) node {$f^D_{v_1v_2}=f^{D'}_{v_1v_2}+P$};
\draw (8,-4.5) node {LD};
\end{scope}
\end{tikzpicture}
\end{center}
\caption{Relations, where $x,y\in\Q$, $P,Q,R\in\Qtt$ and $\gamma,\gamma_1,\gamma_2\in\Al$} \label{figrelLVEVLD}
\end{figure}

The automorphism group $Aut(\Al,\bl)$ of the Blanchfield module $(\Al,\bl)$ acts on $\hat\A_n(\Al,\bl)$ 
by acting on the colorings of all the univalent vertices of a diagram 
simultaneously. Denote by Aut the relation which identifies two diagrams obtained from one another by the action of an element 
of $Aut(\Al,\bl)$. Set:
$$\A_n(\Al,\bl)=\hat\A_n(\Al,\bl)/\langle \textrm{Aut} \rangle \qquad \textrm{and} \qquad \A(\Al,\bl)=\prod_{n\in\N}\A_n(\Al,\bl).$$ 

An {\em $(\Al,\bl)$--augmented diagram} is the union of an $(\Al,\bl)$--colored diagram (its {\em Jacobi part}) and of finitely many isolated vertices colored by prime integers. The {\em degree} of an $(\Al,\bl)$--augmented diagram is the number of its vertices of valence $0$ or $3$. Set:
$$\A_n^\aug(\Al,\bl)=\frac{\Q \langle\mbox{$(\Al,\bl)$--augmented diagrams of degree $n$}\rangle}{\Q \langle\mbox{AS, IHX, LE, OR, Hol, LV, EV, LD, Aut}\rangle}\quad\textrm{for }n\geq 0,$$
and $\A^\aug(\Al,\bl)=\prod_{n\in\N}\A_n^\aug(\Al,\bl).$
Note that a diagram with a univalent vertex $v$ labeled by $0$ may be non trivial, despite the relation LV, if the $f_{vv'}$ are not all trivial.

It will be useful later on to notice that the relations LD and Hol give the generalized relation LD described in Figure ~\ref{FigLDgen}.
\begin{figure}[htb] 
\begin{center}
\begin{tikzpicture} [scale=0.3]
\foreach \x in {0,9} {
\begin{scope} [xshift=\x cm]
\draw (-1,0) -- (-1,4) node{$\bullet$} node[below left] {$\scriptstyle{v_1}$} node[right] {$\gamma_1$};
\draw[->] (-1,2) node[right] {$Q$};
\draw (2.5,0) -- (2.5,4) node{$\bullet$} node[below left] {$\scriptstyle{v_2}$} node[right] {$\gamma_2$};
\draw[->] (2.5,2) node[right] {$R$};
\end{scope}}
\draw (2.5,0) node[right] {$D$};
\draw (5.5,2) node{$=$};
\draw (11.5,0) node[right] {$D'$};
\draw (14.5,2) node{$+$};
\draw (16,0) -- (16,2.75) (18.5,0) -- (18.5,2.75) arc (0:180:1.25);
\draw[->] (17.2,4) -- (17.3,4);
\draw (17.25,4) node[above] {$QP\bar R$};
\draw (18.5,0) node[right] {$D''$};
\draw (28,2) node {$f^D_{v_1v_2}=f^{D'}_{v_1v_2}+P$};
\end{tikzpicture}
\end{center}
\caption{Generalized relation LD} \label{FigLDgen}
\end{figure}

\subsection{The map $\varphi$}
\label{subsecvarphi}

We shall now recall the construction of the map $\varphi:\A^\aug(\Al,\bl)\to\G(\Al,\bl)$. For this, we fix a $\Q\SK$--pair $(M,K)$ with Blanchfield module isomorphic to $(\Al,\bl)$ and we associate surgery data to $(\Al,\bl)$--colored diagrams. For the labels of univalent vertices to make sense in the Blanchfield module of $(M,K)$, we need to choose an isomorphism $\xi:(\Al,\bl)\to(\Al,\bl)(M,K)$. However, different choices produce the same map $\varphi$ (which amounts to saying that the map $\varphi$ respects the relation Aut); hence we will keep implicit the isomorphism $\xi$ in the sequel.

\begin{figure}[htb] 
\begin{center}
\begin{tikzpicture} [scale=0.2]
\draw (-36,0) -- (-18,0);
\draw (-36,0) node{$\scriptscriptstyle{\bullet}$};
\draw (-18,0) node{$\scriptscriptstyle{\bullet}$};
\draw[->,line width=1.5pt,>=latex] (-10.5,0) -- (-7.5,0);
\draw (0,0) node{$\scriptscriptstyle{\bullet}$};
\draw (0,0) -- (6,0);
\draw (8,0) circle (2);
\draw[color=white,line width=6pt] (8,0) arc (-180:-90:2);
\draw (10,0) circle (2);
\draw[color=white,line width=6pt] (10,0) arc (0:90:2);
\draw (10,0) arc (0:90:2);
\draw[->] (11.5,1.3) -- (11.3,1.5);
\draw[->] (6.5,-1.3) -- (6.7,-1.5);
\draw (12,0) -- (18,0);
\draw (18,0) node{$\scriptscriptstyle{\bullet}$};
\end{tikzpicture}
\end{center}
\caption{Replacement of an edge} \label{figrempla4}
\end{figure}

An $(\Al,\bl)$--colored diagram is {\em elementary} if its edges that connect two trivalent vertices are colored by powers of $t$ and if its edges adjacent to univalent vertices are colored by $1$. Let~$D$ be an elementary diagram. Embed $D$ in $M\setminus K$ in such a way that the vertices of $D$ are embedded in some ball $B\subset M\setminus K$ and, for each edge colored by $t^k$, the closed curve obtained by connecting the extremities of this edge by a path in $B$ has linking number $k$ with $K$. Equip~$D$ with the framing induced by an immersion in the plane which induces the fixed orientation of the trivalent vertices. If an edge connects two trivalent vertices, then insert a positive Hopf link in this edge, as shown in Figure \ref{figrempla4}. 
At each univalent vertex $v$, glue a leaf $\ell_v$, trivial in $H_1(M\setminus K;\Q)$, in order to obtain a null Y--link $\Gamma$. Let $V$ be the set of all univalent 
vertices of $D$. Let $\tilde{B}$ be a lift of the ball $B$ in the infinite cyclic covering $\tilde{X}$ of the exterior of $K$ in $M$. 
For $v\in V$, let $\hat{\ell}_v$ be the extension of $\ell_v$ in $\Gamma$ (see Figure~\ref{figextension}) and let $\tilde{\ell}_v$ be the lift of $\hat{\ell}_v$ in $\tilde{X}$ defined by lifting the basepoint (the point $*$ on Figure~\ref{figextension}) in $\tilde{B}$. The null Y--link $\Gamma$ is a {\em realization} of $D$ if:
\begin{itemize}
 \item for all $v\in V$, $\tilde{\ell}_v$ is homologous to the label $\gamma_v$ of $v$, 
 \item for all $(v,v')\in V^2$, $\lk_e(\tilde{\ell}_v,\tilde{\ell}_{v'})=f_{vv'}$. 
\end{itemize}
If such a realization exists, the elementary diagram $D$ is {\em realizable}. 

\begin{figure}[htb]
\begin{center}
\begin{tikzpicture} [scale=0.4]
\begin{scope} [xshift=-5cm]
\draw (0,0) -- (0,3);
\draw (0,4) circle (1);
\draw[->] (1,3.9) -- (1,4.1);
\draw (0,0) node[right] {$*$};
\draw (0,0) node {$\scriptstyle{\bullet}$};
\draw (1,4) node[right] {$\ell_v$};
\end{scope}
\draw[->,very thick] (-0.5,2.5) -- (0.5,2.5);
\begin{scope} [xshift=5cm]
\draw (0,0) .. controls +(0.2,0) and +(0,-3) .. (0.2,3);
\draw (0,0) .. controls +(-0.2,0) and +(0,-3) .. (-0.2,3);
\draw (0,4) circle (1);
\draw[white,very thick] (-0.18,3) -- (0.18,3);
\draw[->] (1,3.9) -- (1,4.1);
\draw (0,0) node[right] {$*$};
\draw (0,0) node {$\scriptstyle{\bullet}$};
\draw (1,4) node[right] {$\hat{\ell}_v$};
\end{scope}
\end{tikzpicture} 
\end{center}
\caption{Extension of a leaf in a Y--graph\\{\footnotesize The left picture shows a part of a Y--graph made of a trivalent vertex $*$ and a univalent vertex~$v$ joined by an edge, and a leaf~$\ell_v$ glued to $v$. In the right picture, the edge and the leaf~$\ell_v$ have been replaced by the extension $\hat\ell_v$ of $\ell_v$.}} \label{figextension}
\end{figure}

Realizable elementary diagrams turn out to generate the graded space $\A(\Al,\bl)$. Any realization of such a diagram, of degree $n$, provides a family of $n$ disjoint null borromean surgeries in $M\setminus K$, defining a bracket in $\F_n(\Al,\bl)$. By \cite[Section 4]{M7}, this gives a well-defined graded $\Q$--linear map $\varphi:\A(\Al,\bl)\to\G(\Al,\bl)$. Further, this map can be
extended to $\A^\aug(\Al,\bl)$: to an isolated vertex labeled by $p$, we associate a surgery $\left(\frac{B_p}{B^3}\right)$, where $B_p$ is a $\Q$--ball satisfying $|H_1(B_p;\Z)|=p$. 

\begin{theorem}[\textup{\cite[Theorem 2.7]{M7}}]
 The graded $\Q$--linear map $\varphi:\A^\aug(\Al,\bl)\to\G(\Al,\bl)$ is well-defined, canonical and surjective.
\end{theorem}

That our map is canonical means that it does not depend of any of the choices we made, including the choice of the $\Q\SK$--pair $(M,K)$.

\subsection{Kricker and Lescop invariants}
\label{subsecKriLes}

We first introduce the diagram space in which the Kricker invariant $Z^\Kri$ and the Lescop invariant $Z^\Les$ take values.

Let $\delta\in\Qtt$. A {\em $\delta$--colored diagram} is a trivalent graph whose vertices are oriented and whose edges are oriented and colored by $\frac{1}{\delta(t)}\Qtt$. The degree of a $\delta$--colored diagram is the number of its vertices. 
Set:
$$\A_n(\delta)=\frac{\Q\langle\mbox{$\delta$--colored diagrams of degree $n$}\rangle}{\Q\langle\mbox{AS, IHX, LE, OR, Hol}\rangle},$$
where AS, IHX, LE, OR, Hol are the relations represented in Figure \ref{figrelAStoHol} with $P,Q,R\in\frac{1}{\delta(t)}\Qtt$. We have a graded algebra $\A(\delta)=\prod_{n\in\N} \A_n(\delta)$, where the product is defined by the disjoint union.
Since any trivalent graph has an even number of vertices, we have $\A_{2n+1}(\delta)=0$ for all~$n\geq0$.

There is a natural ``closing'' map from $\A(\Al,\bl)$ to $\A(\delta)$. 
With an $(\Al,\bl)$--colored diagram $D$ of degree $n$, we associate a $\delta$--colored diagram defined as the sum of all ways of pairing all vertices as indicated in Figure~\ref{figgroup}. 
This provides a well-defined $\Q$--linear map:
$\psi_n:\A_n(\Al,\bl)\to\A_n(\delta).$

\begin{figure}[thb] 
\begin{center}
\begin{tikzpicture} [scale=0.5]
\draw (0,0) -- (0,4) (2.5,0) -- (2.5,4);
\draw (0,4) node {$\scriptstyle{\bullet}$} (2.5,4) node {$\scriptstyle{\bullet}$};
\draw (0,4) node[above right] {$v$} (2.5,4) node[above right] {$v'$};
\draw[->] (0,0) -- (0,2);
\draw[->] (2.5,0) -- (2.5,2);
\draw (0,2) node[right] {$P$} (2.5,2) node[right] {$Q$};
\draw (-0.43,-0.25) -- (0,0) -- (0.43,-0.25) (2.07,-0.25) -- (2.5,0) -- (2.93,-0.25);
\draw[dashed] (-0.86,-0.5) -- (-0.43,-0.25) (0.43,-0.25) -- (0.86,-0.5) (1.64,-0.5) -- (2.07,-0.25) (2.93,-0.25) -- (3.36,-0.5);
\draw[->,line width=1.5pt] (4.8,2) -- (5.7,2);
\begin{scope} [xshift=8.5cm]
\draw (-0.43,-0.25) -- (0,0) -- (0.43,-0.25) (2.07,-0.25) -- (2.5,0) -- (2.93,-0.25);
\draw[dashed] (-0.86,-0.5) -- (-0.43,-0.25) (0.43,-0.25) -- (0.86,-0.5) (1.64,-0.5) -- (2.07,-0.25) (2.93,-0.25) -- (3.36,-0.5);
\draw (0,0) -- (0,2.75) (2.5,0) -- (2.5,2.75) arc (0:180:1.25);
\draw[->] (1.2,4) -- (1.3,4);
\draw (1.25,4) node[above] {$P(t)Q(t^{-1})f_{vv'}(t)$};
\end{scope}
\end{tikzpicture}
\end{center}
\caption{Pairing of vertices} \label{figgroup}
\end{figure}

The following result asserts the existence and the properties of an invariant $Z$ which may be either the Lescop invariant or the Kricker invariant. Althought it is not known 
whether they are equal or not, they both satisfy the properties of the theorem. Until the end of this section, we shall refer to any of the two by ``the invariant $Z$''.

\begin{theorem}[\cite{Les2,Les3,Kri,GK,M6}] \label{thinvariantZ}
 There is an invariant $Z=(Z_n)_{n\in\N}$ of $\Q\SK$--pairs with the following properties. 
 \begin{itemize}
  \item If $(M,K)$ is a $\Q\SK$--pair with Blanchfield module $(\Al,\bl)$, then $Z_n(M,K)\in\A_n(\delta)$, where $\delta$ is the annihilator of $\Al$. 
  \item Fix a Blanchfield module $(\Al,\bl)$ and let $\delta$ be the annihilator of $\Al$. The $\Q$--linear extension $Z_n:\F_0(\Al,\bl)\to\A_n(\delta)$ vanishes on $\F_{n+1}(\Al,\bl)$ and $Z_n\circ\varphi_n=\psi_n$.
  \item The invariant $Z$ is multiplicative under connected sum.
 \end{itemize}
\end{theorem}

In order to take into account the whole quotient $\G_n(\Al,\bl)$, we extend the invariant~$Z$. 
Define a {\em $\delta$--augmented diagram} as the disjoint union of a $\delta$--colored diagram with finitely many isolated vertices colored by prime integers. The {\em degree} of such a diagram is the number of its vertices. Set:
$$\A_n^\aug(\delta)=\frac{\Q\langle\mbox{$\delta$--augmented diagrams of degree $n$}\rangle}{\Q\langle\mbox{AS, IHX, LE, OR, Hol}\rangle}.$$
Once again, the disjoint union makes $\A^\aug(\delta)=\prod_{n\in\N}\A_n^\aug(\delta)$ a graded algebra. 
The map $\psi_n$ naturally extends to a map $\psi_n:\A_n^\aug(\Al,\bl)\to\A_n^\aug(\delta)$ preserving the isolated vertices. 

We shall complete the invariant $Z$ with degree--$1$ invariants. For each prime integer $p$, we define as follows an invariant $\rho_p$ of $\Q$--spheres: for a $\Q$--sphere $M$, $\rho_p(M)=-v_p(|H_1(M;\Z)|)\,\bullet_p$ where $v_p$ is the $p$--adic valuation. These invariants turn out to be degree--$1$ invariants of $\Q$--spheres with respect to LP--surgeries \cite[Proposition~1.9]{M2}. In turn, they are also degree--$1$ invariants of $\Q\SK$--pairs.
Set:
$$Z^\aug=Z\sqcup\expd\left(\sum_{p\textrm{ prime}}\rho_p\right)$$
and denote $Z_n^\aug$ the degree--$n$ part of $Z^\aug$.

\begin{theorem}[\cite{M7} Theorem 2.10] \label{thunivinv}
 Fix a Blanchfield module $(\Al,\bl)$. Let $\delta$ be the annihilator of $\Al$. The $\Q$--linear extension $Z_n^\aug:F_0(\Al,\bl)\to\A_n^\aug(\delta)$ vanishes on $\F_{n+1}(\Al,\bl)$ and satisfies $Z_n^\aug\circ\varphi_n=\psi_n$.
\end{theorem}

To summarize, both the Kricker invariant and the Lescop invariant give rise to an invariant $Z^\aug$ that fits into the following commutative diagram, where all space and maps are graded.

\begin{center}
\begin{tikzpicture} [xscale=1.5,yscale=0.5]
 \draw (2,2) node {$\G(\Al,\bl)$};
 \draw (2,-2) node {$\A^\aug(\delta)$};
 \draw (0,0) node {$\A^\aug(\Al,\bl)$};
 \draw[->>] (0.7,0.45) -- (1.5,1.6); \draw (1,1.4) node {$\varphi$};
 \draw[->] (0.7,-0.45) -- (1.5,-1.6); \draw (1,-1.5) node {$\psi$};
 \draw[->] (2,1.3) -- (2,-1.3); \draw (2.4,0) node {$Z^\aug$};
\end{tikzpicture}
\end{center}

To reach a full description of the graded space $\G(\Al,\bl)$, we shall construct a lift of the augmented Kricker invariant with values in $\A^\aug(\Al,\bl)$.

\section{Surgery presentations and winding matrices} \label{secwinding}

\subsection{Surgery presentation and Blanchfield module} \label{subsecBlanchfield}
Throughout the article, we fix a trivial knot $\un\subset S^3$. By a {\em surgery link}, we mean a framed link $L=\sqcup_{i=1}^s L_i\subset (S^3\setminus\un)$ whose connected components $L_i$ satisfy $\lk(L_i,\un)=0$. This condition ensures that the pair $(M,K)$ obtained from $(S^3,\un)$ by surgery on $L$ is a $\Q\SK$--pair. Moreover, any $\Q\SK$--pair admits such a surgery presentation (see \cite[Section 2.1]{M1}). 
In this section, we define the equivariant linking matrix of a surgery link and we give a diagrammatic computation of it. 

Let $L\subset S^3\setminus\un$ be a surgery link. Let $\DD$ be a disk in $S^3$ bounded by $\un$, transverse to $L$. An {\em admissible diagram} of $L$ is a projection of $L\cup\DD$ onto a square $[-1,1]^2$ where:
\begin{itemize}
 \item the image of $\DD$ is the segment line $[(0,0),(1,0)]$,
 \item the multiple points of the projection restricted to $L$ are transverse double points disjoint from $\DD$,
 \item the points of $L$ that project onto $[(0,0),(1,0)]$ are the points of $L\cap\DD$.
\end{itemize}

\begin{figure}[htb] 
\begin{center}
\begin{tikzpicture} 
 \draw (-1,-1) -- (1,-1) -- (1,1) -- (-1,1) -- (-1,-1);
 \draw (0,0) node {$\scriptstyle{\bullet}$} -- (1,0); 
 \draw (0.1,-0.3) arc (0:-90:0.3) arc (270:75:0.6);
 \draw (-0.1,0.3) arc (180:90:0.3) arc (90:-105:0.6);
 \draw (0.1,0.4) arc (0:-90:0.2) arc (90:270:0.2);
 \draw (-0.1,-0.4) arc (180:90:0.2) arc (-90:90:0.2);
 \draw (0.1,0.4) arc (0:20:0.2);
 \draw (-0.1,-0.4) arc (180:200:0.2);
\end{tikzpicture}
\end{center} \caption{An admissible diagram of a surgery presentation} \label{figadmproj}
\end{figure}

Let $E$ be the exterior of $\un$ in $S^3$ and let $\tilde{E}$ be the infinite cyclic covering of $E$. Note that~$\tilde{E}$ is homeomorphic to $\DD\times\R$. In particular, the equivariant linking number of knots in $\tilde E$, as defined in Section~\ref{subsecsurgeries}, takes values in $\Ztt$.

Fix an admissible diagram of $L$ and base points $\star_i$ of its components, away from the crossings and the disk $\DD$. Let $\tilde{E}_0\subset\tilde{E}$ be a copy of $E$ cut along $\DD$ and define the lift $\tilde{L_i}$ of $L_i$ in~$\tilde{E}$ by lifting~$\star_i$ in~$\tilde{E}_0$. Note that, since $L_i$ is null-homologous in $S^3\setminus\un$, $\tilde L_i$ is a knot in $\tilde E$. Consider the matrix of equivariant linkings $W_L=\left(\lk_e(\tilde{L}_i,\tilde{L}_j)\right)_{1\leq i,j\leq n}$. 
If the link $L$ is a surgery presentation for a $\Q\SK$--pair $(M,K)$, then the matrix $ ^tW_L$ is a presentation matrix of the Alexander module of $(M,K)$ with generators the classes of meridians $m_i$ of the components~$\tilde L_i$ \cite[Proposition~2.5]{M1}. Moreover, the Blanchfield form is given on these generators by the matrix $-W_L^{-1}$ \cite[Corollary~3.2]{M1}.

\subsection{Winding matrix}
We now give a diagrammatic computation of the matrix $W_L$.

Given an admissible diagram of $L$, define the {\em winding number} $w(L_i,L_j)\in\Ztt$ of $L_i$ and $L_j$ in the following way. For a crossing~$c$ between $L_i$ and $L_j$, denote $\varepsilon_{ij}(c)$ the algebraic intersection number of the disk $\DD$ with the path that goes from $\star_i$ to $c$ along $L_i$ and then from $c$ to $\star_j$ along~$L_j$. If $i=j$, change component at the first occurence of $c$. Set
$$w(L_i,L_j)=\left\lbrace\begin{array}{l l}
                         \displaystyle\frac{1}{2}\sum_c \textrm{sg}(c)\,t^{\varepsilon_{ij}(c)} & \textrm{ if } i\neq j \\ & \\
                         \displaystyle\frac{1}{2}\sum_c \textrm{sg}(c)\,(t^{\varepsilon_{ii}(c)}+t^{-\varepsilon_{ii}(c)}) & \textrm{ if } i=j
                        \end{array}\right.$$
where the sums are over all crossings between $L_i$ and $L_j$, and $\textrm{sg}(c)\in\{\pm1\}$ is the sign of the crossing $c$. Note that $w(L_j,L_i)(t)=w(L_i,L_j)(t^{-1})$. 
\begin{lemma} \label{lemmawdmatrix1}
 The winding numbers are invariant by isotopies that do not allow the base points to pass through the disk $\DD$. 
\end{lemma}
\begin{proof}
 First note that the winding numbers are preserved when a crossing passes through the disk~$\DD$. It is also preserved when the base point of a component passes through a crossing since the algebraic intersection number of this component with $\DD$ is trivial. Hence it only remains to check invariance with respect to framed Reidemeister moves performed far from the base points and the disks, which is direct. 
\end{proof}

\begin{lemma}
 $W_L=\big(w(L_i,L_j)\big)_{1\leq i,j\leq n}$
\end{lemma}
\begin{proof}
First note that, since $\un$ is a trivial knot in $S^3$, its Alexander module is trivial, so that $\lk_e(\tilde L_i,\tilde L_j)=\sum_{k\in\Z}\lk\big(\tilde L_i,\tau^k(\tilde L_j)\big)\,t^k$. 
From the diagram of $L$, we can get a diagram of $\tilde L$ and its translates: cut the diagram along the image of $\DD$ and glue together $\Z$ copies of it, see Figure~\ref{figliftsurgerylink}. A crossing $c$ between $L_i$ and $L_j$ such that $\varepsilon_{ij}(c)=k$ lifts as a crossing between $\tilde L_i$ and $\tau^k(\tilde L_j)$, so that it contributes equally to $w(L_i,L_j)$ and $\lk_e(\tilde L_i,\tilde L_j)$. When $i=j$ and $k\neq0$, the crossing~$c$ lifts as two crossings of $\tilde L_i$, one with $\tau^k(\tilde L_i)$ and one with $\tau^{-k}(\tilde L_i)$.
\end{proof}

\begin{figure}[htb] 
\begin{center}
\begin{tikzpicture} 
\begin{scope} [scale=1.5]
 \draw (-1,-1) -- (1,-1) -- (1,1) -- (-1,1) -- (-1,-1);
 \draw (0,0) node {$\scriptstyle{\bullet}$} -- (1,0); 
 \draw (0.1,-0.3) arc (0:-90:0.3) node {$\star$} arc (270:75:0.6);
 \draw[->] (-0.5,-0.52) -- (-0.55,-0.49) node[below] {$L_2$};
 \draw (-0.1,0.3) arc (180:90:0.3) arc (90:-105:0.6);
 \draw (0.2,-0.6) node {$\star$};
 \draw[->] (0.5,-0.52) -- (0.55,-0.49) node[below] {$L_1$};
 \draw (0.1,0.4) arc (0:-90:0.2) arc (90:270:0.2);
 \draw (-0.1,-0.4) arc (180:90:0.2) arc (-90:90:0.2);
 \draw (0.1,0.4) arc (0:20:0.2);
 \draw (-0.1,-0.4) arc (180:200:0.2);
\end{scope}
\begin{scope} [xshift=3cm,yshift=-3cm,scale=0.6]
 \foreach \x in {0,4} {
 \draw (\x,-0.2) -- (\x,10.2);
 \draw[dashed] (\x,-0.7) -- (\x,-0.2) (\x,10.7) -- (\x,10.2);
 }
 \foreach \y in {0,2.5,5,7.5} {
 \draw (1,\y) .. controls +(0,1) and +(-0.5,0) .. (2,1.5+\y);
 \draw (3,2.5+\y) .. controls +(0,-1) and +(0.5,0) .. (2,1+\y);
 \draww{(3,\y) .. controls +(0,1) and +(0.5,0) .. (2,1.5+\y);}
 \draww{(1,2.5+\y) .. controls +(0,-1) and +(-0.5,0) .. (2,1+\y);}
 }
 \foreach \x in {1,3} {
 \draw (\x,0) -- (\x,-0.2) (\x,10) -- (\x,10.2);
 \draw[dashed] (\x,-0.7) -- (\x,-0.2) (\x,10.7) -- (\x,10.2);
 }
 \foreach \y in {0,5,10} \draw (0,\y) -- (4,\y);
 \draw (2.9,4.3) node {$\star$};
 \draw (2.9,3.2) node {$\star$};
 \draw[->] (2.9,5.7) -- (2.88,5.75) node[right] {$\tilde L_1$};
 \draw[->] (2.9,1.8) -- (2.88,1.75) node[right] {$\tilde L_2$};
 \draw (5,2.5) node {$\tilde E_0$};
 \draw (5,7.5) node {$\tau(\tilde E_0)$};
\end{scope}
 \draw (9,0) node {$w(L_1,L_2)=t+1$};
\end{tikzpicture}
\end{center} \caption{An admissible diagram and its lift} \label{figliftsurgerylink}
\end{figure}

In the sequel, we call $W_L$ the {\em winding matrix} of $L$.
To fully understand the effect of an isotopy on this matrix, we shall describe its modification when a base point passes through the disk $\DD$. Fix a component~$L_i$. Fix an admissible diagram of $L$ with the base point of $L_i$ located ``just before'' the disk, as shown in the first part of Figure \ref{figbasepoint}.
Consider another admissible diagram of $L$ which differs from the previous one only by the position of the base point~$\star_i$, which is as shown 
on the second part of Figure \ref{figbasepoint}. Let $\varepsilon=\pm1$ give the sign of the intersection point of $\DD$ and $L_i$ that the base point passes through. It is easily seen that the winding matrix of the latter diagram is obtained from the winding matrix of the previous one by multiplying the coefficients of the $i$--th column ({\em resp} row) by $t^\varepsilon$ ({\em resp} $t^{-\varepsilon}$).

\begin{figure}[htb] 
\begin{center}
\begin{tikzpicture} [scale=0.5]
\begin{scope}
 \draw (0.5,0) -- (4,0) node[right] {$\DD$};
 \draw (3,2) .. controls +(0,-1) and +(1,1) .. (1,-2);
 \draw[->] (1.44,-1.5) -- (1.36,-1.6) node[right] {$L_i$};
 \draw (2.87,0.8) node {$\star_i$};
\end{scope}
 \draw (7,0) node {$\rightsquigarrow$};
\begin{scope} [xshift=9cm]
 \draw (0.5,0) -- (4,0) node[right] {$\DD$};
 \draw (3,2) .. controls +(0,-1) and +(1,1) .. (1,-2);
 \draw[->] (1.44,-1.5) -- (1.36,-1.6) node[right] {$L_i$};
 \draw (2.09,-0.8) node {$\star_i$};
\end{scope}
\end{tikzpicture}
\caption{Base point passing through the disk} \label{figbasepoint}
\end{center}
\end{figure}

\section{Diagram spaces} \label{secdiagrams}

In this section, we introduce the diagram spaces that are needed in the construction of the invariant $\tZ$, and we define an operation that will play, in our construction, the role of the formal Gaussian integral in the construction of the Kricker invariant.

\subsection{Beaded Jacobi diagrams} \label{subsecdiagramspaces}

For a compact oriented 1--manifold $X$ ({\em resp} a finite set $C$), a {\em Jacobi diagram on $X$} ({\em resp} a {\em Jacobi diagram on $C$}) is a unitrivalent graph whose trivalent vertices are oriented and whose univalent vertices are embedded in the interior of $X$ ({\em resp} labeled by $C$). When relevant, the manifold $X$ is called the {\em skeleton} of the diagram. 
A {\em beaded Jacobi diagram on $X$ or $C$} is a Jacobi diagram on $X$ or $C$ whose graph edges are oriented and labeled by $\Qtt$ (following the terminology of ``diagrams with beads'' from \cite{GK}). A {\em w--beaded Jacobi diagram on $X$} is a beaded Jacobi diagram on $X$ whose skeleton is viewed as a union of edges ---defined by the embedded vertices--- that are labeled by powers of $t$, with the condition that the product of the labels on each component of $X$ is~1. The {\em degree} of a unitrivalent diagram is the number of its trivalent vertices (sometimes called i--degree); the vertices embedded in $X$ are univalent and are not counted. 
Set:
$$\tA(X)=\mathrm{dc}\left(\frac{\Q\langle \textrm{beaded Jacobi diagrams on }X \rangle}{\Q\langle \textrm{AS, IHX, STU, LE, OR, Hol} \rangle}\right),$$
$$\tAw(X)=\mathrm{dc}\left(\frac{\Q\langle \textrm{w--beaded Jacobi diagrams on }X \rangle}{\Q\langle \textrm{AS, IHX, STU, LE, OR, Hol, \holw} \rangle}\right),$$
$$\tA(*_C)=\mathrm{dc}\left(\frac{\Q\langle \textrm{beaded Jacobi diagrams on }C \rangle}{\Q\langle \textrm{AS, IHX, LE, OR, Hol} \rangle}\right),$$
where the relations are defined in Figures \ref{figrelAStoHol} and \ref{figrelSTUHolw}, and $\mathrm{dc}(E)$ is the degree completion of the space $E$. 
In the relation STU, corresponding edges have the same orientation and label.
In the pictures, the skeleton is represented with bold lines and the graph with thin lines. 

\begin{figure}[htb] 
\begin{center}
\begin{tikzpicture} [scale=0.3]
%STU
\begin{scope} 
 \draw[very thick,->] (0,0) -- (0,4);
 \draw[->] (0,2) -- (0.8,2) node[above] {1};
 \draw (0.7,2) -- (1.3,2) -- (2.6,3) (1.3,2) -- (2.6,1);
 \draw (3.8,2) node{$=$};
 \draw[very thick,->] (5,0) -- (5,4);
 \draw (5,2.6) -- (7.6,2.6) (5,1.4) -- (7.6,1.4);
 \draw (8.8,2) node {$-$};
 \draw[very thick,->] (10,0) -- (10,4);
 \draw (10,2.6) -- (12.6,1.4) (10,1.4) -- (12.6,2.6);
 \draw (6.3,-1.5) node{STU};
\end{scope}
%Holw
\begin{scope} [xshift=19cm]
 \draw[very thick,->] (0,0) -- (0,4);
 \draw (0,3) node {$\scriptscriptstyle{\bullet}$} node[left] {$t^i$};
 \draw (0,1) node {$\scriptscriptstyle{\bullet}$} node[left] {$t^j$};
 \draw (0,2) -- (2.6,2);
 \draw[->] (0,2) -- (1.5,2) node[above] {$P$};
 \draw (4,2) node {$=$};
 \draw[very thick,->] (8,0) -- (8,4);
 \draw (8,3) node {$\scriptscriptstyle{\bullet}$} node[left] {$t^{i+1}$};
 \draw (8,1) node {$\scriptscriptstyle{\bullet}$} node[left] {$t^{j-1}$};
 \draw (8,2) -- (10.6,2);
 \draw[->] (8,2) -- (9.5,2) node[above] {$tP$};
 \draw (5,-1.5) node {\holw};
\end{scope}
%Holw'
\begin{scope} [xshift=38cm]
 \draw[very thick] (-1.5,2) circle (1.5);
 \draw[very thick,->] (-1.45,3.5) -- (-1.55,3.5);
 \draw (-3,2) node {$\scriptscriptstyle{\bullet}$} node[left] {$t^i$};
 \draw (0,2) -- (2.6,2);
 \draw[->] (0,2) -- (1.5,2) node[above] {$P$};
 \draw (4,2) node {$=$};
 \draw[very thick] (8,2) circle (1.5);
 \draw[very thick,->] (8.05,3.5) -- (7.95,3.5);
 \draw (6.5,2) node {$\scriptscriptstyle{\bullet}$} node[left] {$t^i$};
 \draw (9.5,2) -- (12.1,2);
 \draw[->] (9.5,2) -- (11,2) node[above] {$tP$};
 \draw (5,-1.5) node {\holw\ (special case)};
\end{scope}
\end{tikzpicture}
\end{center} \caption{Relations STU and \holw\ on Jacobi diagrams} \label{figrelSTUHolw}
\end{figure}

\begin{remark}
  For diagrams in $\tA(X)$, the condition on
  the labels on the skeleton implies that all labels can be pushed off each component of the skeleton using the relation \holw. When the component is an interval, there is a unique way to do so. 
\end{remark}

For a finite set $C$, denote by $\xd_C$ ({\em resp} $\xc_C$) the manifold made of $|C|$ disjoint intervals ({\em resp} circles) indexed by the elements of $C$. The above remark provides an isomorphism between~$\tAw(\xd_C)$ and $\tA(\xd_C)$. In the case of a skeleton with closed components, we need to add a relation to get such an isomorphism.

Given a beaded Jacobi diagram $D$ on $\xc_C$, a label $c\in C$ and an integer $k$, the associated {\em winding relation} identifies $D$ with the diagram obtained from $D$ by {\em pushing $t^k$ at each vertex glued on $\xc_c$}, {\em ie} by multiplying the label of each edge adjacent to a univalent vertex glued on $\xc_c$ by~$t^k$ if the orientation of the edge goes backward the vertex and by $t^{-k}$ otherwise (with the same~$k$ for all vertices), see Figure~\ref{figwindingrelskel}. 
Denote by $\eqw$ the induced equivalence relation. It provides an isomorphism $\tAw(\xc_C)\cong\fract{\tA(\xc_C)}{\eqw}$.
We want to relate this quotient to the space~$\tA(\xd_C)$. The winding relation $\eqw$ is defined as above on $\tA(\xd_C)$. We also define a {\em link relation} on $\tA(\xd_C)$ as follows. Given two beaded Jacobi diagrams $D_1$ and $D_2$ on $\xd_C$, we have $D_1\eql D_2$ if, for an index $c\in C$ and two extra indices $c_1$ and $c_2$ not in $C$, there is a beaded Jacobi diagram $D$ on $\xd_{(C\setminus\{c\})\cup\{c_1,c_2\}}$ such that $D_1$ and $D_2$ are obtained from $D$ by gluing together the skeleton components $\xd_{c_1}$ and $\xd_{c_2}$ in the two possible orders. It is easily checked that~\mbox{$\fract{\tA(\xc_C)}{\eqw}\cong\fract{\tA(\xd_C)}{\eqw,\eql}$.} The link relation $\eql$ is similarly
defined on $\tAw(\xd_C)$. 

\begin{figure}[htb] 
\begin{center}
\begin{tikzpicture} [scale=0.45]
 \foreach \x/\i/\j/\k in {0/2/3/-2,16//4/-1} {
 \begin{scope} [xshift=\x cm]
 \draw[very thick,->] (-5,2) arc (-180:180:2) node[left] {$\scriptstyle{1}$};
 \draw[very thick,->] (7,2) arc (360:0:2) node[right] {$\scriptstyle{2}$};
 \draw[->] (1,3.5) .. controls +(1,0) and +(-0.5,0.3) .. (3.2,3)  (-1.2,3) .. controls +(0.5,0.3) and +(-1,0) .. (1,3.5) node[above] {$\scriptstyle{t^{\i}}$};
 \draw[->] (-1.2,1) .. controls +(0.4,-0.3) and +(-0.8,0) .. (0.2,0.5) node[below] {$\scriptstyle{t^7}$};
 \draw (0.2,0.5) -- (1,0.5) .. controls +(0.2,-0.2) and +(-0.4,-0.2) .. (2,0.5);
 \draw[->] (3.2,1) -- (2,0.5) node[below] {$\scriptstyle{t^{\k}}$};
 \draw[->] (1,0.5) .. controls +(0,0.4) and +(-0.4,-0.2) .. (2,1.5) (3,2) -- (2,1.5) node[above] {$\scriptstyle{t^\j}$};
 \end{scope}}
 \draw (9,2) node {$=$};
\end{tikzpicture}
\end{center} \caption{A winding relation (for $c=2$ and $k=1$)}
\label{figwindingrelskel}
\end{figure}

In \cite[Theorem~8]{BN}, Bar-Natan defines a formal PBW linear isomorphism:
$$\chi_C:\tA(*_C)\iso\tA(\xd_C).$$
For a beaded Jacobi diagram $D$, the image $\chi_C(D)$ is the average of all possible ways to attach the $c$--colored vertices of $D$ on the interval $\xd_c$ for each $c\in C$. The setting of \cite{BN} is not exactly the same, but the argument adapts directly. 
To recover an isomorphism onto $\tA(\xc_C)$, one needs a version of the link relations on $\tA(*_C)$.
These relations were first introduced in \cite[Section 5.2]{AA2}; the ones we use here mainly come from \cite{GK}.

Given a beaded Jacobi diagram $D$ on $C$ and distinct elements $c,\bar c\in C$, we define $\langle D\rangle_{c-\bar c}$ as the sum of all diagrams obtained from $D$ by gluing all $c$--labeled vertices to all $\bar c$--labeled vertices when there are as many $c$ and $\bar c$--labeled vertices in $D$, and as $0$ otherwise. We say that a beaded Jacobi diagram is $(c-\bar c)$--substantial if it contains no strut $\struts{c}{\bar c}{}$. We extend the definition of $\langle D\rangle_{c-\bar c}$ by linearity to infinite series of $(c-\bar c)$--substantial diagrams. We also denote $D_{|c\to ce^h}$ the diagram obtained from $D$ by \emph{pushing $e^h$} on each $c$--labeled vertex of $D$, where pushing $e^h$ is the operation pictured in Figure~\ref{figpushexp}; note that the $h$--ended edges are added on the right when going toward the $c$--labeled vertex. We extend the definition by linearity to infinite series of diagrams.

We define the {\em link relation} $\eql$ on $\tA(*_C)$ as generated by the following: given $G_1,G_2\in\tA(*_C)$, we have $G_1\eql G_2$ if, for some $c\in C$ and some extra vertices labeled by $h,\breve h$, there is an infinite series $G$ of $(h-\breve h)$--substantial beaded Jacobi diagrams on $C\cup\{h,\breve h\}$ such that $G_1=\langle G\rangle_{h-\breve h}$ and $G_2=\langle G_{|c\to ce^h}\rangle_{h-\breve h}$.

\begin{figure}[htb] 
\begin{center}
\begin{tikzpicture} [yscale=0.5]
 \foreach \x in {0,4} {
 \draw[dashed] (\x,-2) -- (\x,-1);
 \draw[-%>
 ] (\x,-1) -- (\x,0);
 \draw (\x,0) -- (\x,2) node[above] {$c$};}
 \draw (1.3,0) node {$\rightsquigarrow$};
 \draw (2.7,0) node {$\displaystyle\sum_{n\geq0}\frac1{n!}$};
 \foreach \y in {0.2,0.6,1.6} {
 \draw (4,\y) -- (4.5,\y) node[right] {$\scriptscriptstyle h$};}
 \draw (4.2,1.3) node {$\vdots$};
 \draw[decorate,decoration={brace,amplitude=4pt},xshift=-0.1cm] (4,0.2) -- (4,1.6) node [midway,left,xshift=-0.1cm] {$\scriptscriptstyle n$};
\end{tikzpicture}
\caption{Pushing $e^h$ on a $c$--labeled vertex} \label{figpushexp}
\end{center}
\end{figure}

\begin{proposition}[\textup{Garoufalidis--Kricker \cite[Lemma~3.6]{GK}}]\label{proplemmaGK}
 For $G_1,G_2\in\tA(*_C)$, we have $G_1\eql G_2$ if and only if $\chi_C(G_1)\eql\chi_C(G_2)$.
\end{proposition}

Pushing $t^k$ at the $c$--labeled vertices defines as above a winding relation $\eqw$ on $\tA(*_C)$, see Figure~\ref{figwindingrel}. 
Set $\tA(\xxlw_C)=\fract{\tA(*_C)}{\eqw,\eql}$. The following is a corollary of Proposition~\ref{proplemmaGK}.

\begin{figure}[htb] 
\begin{center}
\begin{tikzpicture} [scale=0.45]
\foreach \x/\i/\j in {0/2/3,10//4} {
\begin{scope} [xshift=\x cm]
 \draw[->] (-2,2) -- (-2,3) node[above] {$\scriptstyle{c}$} (-2,1) node[below] {$\scriptstyle{c}$} -- (-2,2) node[left] {$\scriptstyle{t^{-1}}$};
 \draw[->] (-0.5,1) node[below] {$\scriptstyle{e}$} -- (0.3,1) node[above] {$\scriptstyle{1}$};
 \draw[->] (0.3,1) -- (1,1) -- (1.75,0.5) node[below] {$\scriptstyle{t^5}$};
 \draw[->] (1.75,0.5) -- (2.5,0) node[right] {$\scriptstyle{d}$} (1,1) -- (1.75,1.5) (2.5,2) node[right] {$\scriptstyle{c}$} -- (1.75,1.5) node[above] {$\scriptstyle{t^\j}$};
 \draw[->] (1.2,3) -- (2.3,3) node[right] {$\scriptstyle{c}$} (0,3) node[left] {$\scriptstyle{d}$} -- (1.2,3) node[above] {$\scriptstyle{t^{\i}}$};
\end{scope}}
 \draw (5,2) node {$=$};
\end{tikzpicture}
\end{center} \caption{A winding relation on $\tA(*_C)$ (on $c$ for $k=1$)} \label{figwindingrel}
\end{figure}

\begin{proposition}
 The isomorphism $\chi_C:\tA(*_C)\iso\tA(\xd_C)$ descends to an isomorphism $\chi_C:\tA(\xxlw_C)\iso\fract{\tA(\xc_C)}{\eqw}.$
\end{proposition}

We finally have the following commutative diagram of diagram spaces.
$$\xymatrix{
\tAw(\xd_C) \ar[r]_\cong \ar@{->>}[d]^(.4){/\eql} & \tA(\xd_C) \ar[r]^{\chi_C^{-1}}_\cong \ar@{->>}[d]^(.4){/\eql,\eqw} & \tA(*_C) \ar@{->>}[d]^(.4){/\eql,\eqw} \\
\tAw(\xc_C) \ar[r]_\cong & \fract{\tA(\xc_C)}{\eqw} \ar[r]^{\chi_C^{-1}}_\cong & \tA(\xxlw_C)
}$$

Given a subset $S$ of a finite set $C$, one can also consider Jacobi diagrams with univalent vertices either labeled by $S$ or embedded in $\xd_{C\setminus S}$ or $\xc_{C\setminus S}$. This provides diagram spaces $\tA(*_S,\xd_{C\setminus S})$ and $\tA(*_S,\xc_{C\setminus S})$, and their quotients. As above, we have an isomorphism $$\chi_S:\tA(*_S,\xd_{C\setminus S})\iso\tA(\xd_C).$$

\subsection{Product and coproduct}

We first define a coproduct on the diagram spaces of the previous subsection. Given a (w--)beaded Jacobi diagram $D$ on $X$ or $C$, denote by $\dddot{D}$ the diagram $D$ with its skeleton removed, and write $\dddot D=\sqcup_{i\in I}\dddot{D}_i$, where the underlying graph of each $\dddot{D}_i$ is connected. For $J\subset I$, define a beaded Jacobi diagram $D_J=D\setminus(\sqcup_{i\in I\setminus J}\dddot{D}_i)$ where the label of an edge of the skeleton is the product of the labels appearing in $D$ on this edge (an edge of the skeleton of $D_J$ may correspond to a union of edges of $D$); all other data (vertex orientation, edge orientation, labels) are the same as in $D$; see Figure~\ref{figCoproduct}. The coproduct of a diagram $D$ is defined by $$\Delta(D)=\sum_{J\subset I}D_J\otimes D_{I\setminus J},$$
and we extend this definition to infinite series of diagrams by linearity.
Note that the different relations on beaded Jacobi diagrams respect the coproduct. This provides a notion of {\em group-like} elements, {\em ie} elements $G$ such that $\Delta(G)=G\otimes G$. Also, the isomorphisms $\chi$ of the previous subsection preserve the coproduct. 

\begin{figure}[htb] 
\begin{center}
\begin{tikzpicture} [xscale=0.8,yscale=0.25]
 \begin{scope} [yshift=1cm]
 \foreach \x in {0,3}
 \draw[->,very thick] (\x,0) -- (\x,12.5);
 \draw (0,10) -- (3,10) (0,8) -- (1.5,10) (0,6) -- (3,8) (3,6) -- (2,4) -- (3,2) (2,4) -- (3,4);
 \foreach \y/\k in {3/2,7/-1,9/3,11/-4}
 \draw (0,\y) node {$\scriptscriptstyle{\bullet}$} node[left] {$\scriptstyle{t^{\k}}$};
 \foreach \y/\k in {1/,3/5,5/-2,7/3,9/-1,11/-6}
 \draw (3,\y) node {$\scriptscriptstyle{\bullet}$} node[right] {$\scriptstyle{t^{\k}}$};
 \draw[->] (0.7,10) -- (0.8,10) node[above] {$\scriptstyle{t^2}$};
 \draw[->] (2.3,10) -- (2.4,10) node[above] {$\scriptstyle{t}$};
 \draw[->] (0.9,9.2) -- (0.6,8.8) node[below] {$\scriptstyle{t^3}$};
 \draw[->] (1.5,7) -- (1.2,6.8) node[below] {$\scriptstyle{t^{-1}}$};
 \draw[->] (2.8,4) -- (2.7,4) node[above] {$\scriptstyle{t}$};
 \draw[->] (2.4,4.8) -- (2.5,5) node[above] {$\scriptstyle{t^2}$};
 \draw[->] (2.5,3) -- (2.6,2.8) node[below] {$\scriptstyle{t^3}$};
 \draw (1.5,-1) node {$D=D_{\{1,2,3\}}$};
\end{scope}
\begin{scope} [yshift=-7cm]
 \foreach \x in {0,3}
 \draw[->,very thick] (\x,0) -- (\x,2.5);
 \draw (1.5,-1) node {$D_{\emptyset}$};
 \draw (0,1) node {$\scriptscriptstyle{\bullet}$} node[left] {$\scriptstyle{1}$};
 \draw (3,1) node {$\scriptscriptstyle{\bullet}$} node[right] {$\scriptstyle{1}$};
\end{scope}
\begin{scope} [xshift=5cm,yshift=1cm]
 \foreach \x in {0,3}
 \draw[->,very thick] (\x,6) -- (\x,12.5);
 \draw (0,10) -- (3,10) (0,8) -- (1.5,10) ;
 \foreach \y/\k in {7/,9/3,11/-4}
 \draw (0,\y) node {$\scriptscriptstyle{\bullet}$} node[left] {$\scriptstyle{t^{\k}}$};
 \foreach \y/\k in {8/6,11/-6}
 \draw (3,\y) node {$\scriptscriptstyle{\bullet}$} node[right] {$\scriptstyle{t^{\k}}$};
 \draw[->] (0.7,10) -- (0.8,10) node[above] {$\scriptstyle{t^2}$};
 \draw[->] (2.3,10) -- (2.4,10) node[above] {$\scriptstyle{t}$};
 \draw[->] (0.9,9.2) -- (0.6,8.8) node[below] {$\scriptstyle{t^3}$};
 \draw (1.5,5) node {$D_{\{1\}}$};
\end{scope}
\begin{scope} [xshift=10cm,yshift=5cm]
 \foreach \x in {0,3}
 \draw[->,very thick] (\x,4) -- (\x,8.5);
 \draw (0,6) -- (3,6);
 \foreach \y/\k in {5/2,7/-2}
 \draw (0,\y) node {$\scriptscriptstyle{\bullet}$} node[left] {$\scriptstyle{t^{\k}}$};
 \foreach \y/\k in {5/7,7/-7}
 \draw (3,\y) node {$\scriptscriptstyle{\bullet}$} node[right] {$\scriptstyle{t^{\k}}$};
 \draw[->] (1.5,6) -- (1.2,6) node[below] {$\scriptstyle{t^{-1}}$};
 \draw (1.5,3) node {$D_{\{2\}}$};
\end{scope}
\begin{scope} [xshift=15cm,yshift=3cm]
 \foreach \x in {0,3}
 \draw[->,very thick] (\x,2) -- (\x,10.5);
 \draw (3,8) -- (2,6) -- (3,4) (2,6) -- (3,6);
 \draw (0,6) node {$\scriptscriptstyle{\bullet}$} node[left] {$\scriptstyle{1}$};
 \foreach \y/\k in {3/,5/5,7/-2,9/-4}
 \draw (3,\y) node {$\scriptscriptstyle{\bullet}$} node[right] {$\scriptstyle{t^{\k}}$};
 \draw[->] (2.8,6) -- (2.7,6) node[above] {$\scriptstyle{t}$};
 \draw[->] (2.4,6.8) -- (2.5,7) node[above] {$\scriptstyle{t^2}$};
 \draw[->] (2.5,5) -- (2.6,4.8) node[below] {$\scriptstyle{t^3}$};
 \draw (1.5,1) node {$D_{\{3\}}$};
\end{scope}
\begin{scope} [xshift=15cm,yshift=-11cm]
 \foreach \x in {0,3}
 \draw[->,very thick] (\x,4) -- (\x,12.5);
 \draw (0,10) -- (3,10) (0,8) -- (1.5,10) (0,6) -- (3,8);
 \foreach \y/\k in {5/2,7/-1,9/3,11/-4}
 \draw (0,\y) node {$\scriptscriptstyle{\bullet}$} node[left] {$\scriptstyle{t^{\k}}$};
 \foreach \y/\k in {6/7,9/-1,11/-6}
 \draw (3,\y) node {$\scriptscriptstyle{\bullet}$} node[right] {$\scriptstyle{t^{\k}}$};
 \draw[->] (0.7,10) -- (0.8,10) node[above] {$\scriptstyle{t^2}$};
 \draw[->] (2.3,10) -- (2.4,10) node[above] {$\scriptstyle{t}$};
 \draw[->] (0.9,9.2) -- (0.6,8.8) node[below] {$\scriptstyle{t^3}$};
 \draw[->] (1.5,7) -- (1.2,6.8) node[below] {$\scriptstyle{t^{-1}}$};
 \draw (1.5,3) node {$D_{\{1,2\}}$};
\end{scope}
\begin{scope} [xshift=10cm,yshift=-9cm]
 \foreach \x in {0,3}
 \draw[->,very thick] (\x,2) -- (\x,12.5);
 \draw (0,10) -- (3,10) (0,8) -- (1.5,10) (3,8) -- (2,6) -- (3,4) (2,6) -- (3,6);
 \foreach \y/\k in {7/,9/3,11/-4}
 \draw (0,\y) node {$\scriptscriptstyle{\bullet}$} node[left] {$\scriptstyle{t^{\k}}$};
 \foreach \y/\k in {3/,5/5,7/-2,9/2,11/-6}
 \draw (3,\y) node {$\scriptscriptstyle{\bullet}$} node[right] {$\scriptstyle{t^{\k}}$};
 \draw[->] (0.7,10) -- (0.8,10) node[above] {$\scriptstyle{t^2}$};
 \draw[->] (2.3,10) -- (2.4,10) node[above] {$\scriptstyle{t}$};
 \draw[->] (0.9,9.2) -- (0.6,8.8) node[below] {$\scriptstyle{t^3}$};
 \draw[->] (2.8,6) -- (2.7,6) node[above] {$\scriptstyle{t}$};
 \draw[->] (2.4,6.8) -- (2.5,7) node[above] {$\scriptstyle{t^2}$};
 \draw[->] (2.5,5) -- (2.6,4.8) node[below] {$\scriptstyle{t^3}$};
 \draw (1.5,1) node {$D_{\{1,3\}}$};
\end{scope}
\begin{scope} [xshift=5cm,yshift=-7cm]
 \foreach \x in {0,3}
 \draw[->,very thick] (\x,0) -- (\x,10.5);
 \draw (0,6) -- (3,8) (3,6) -- (2,4) -- (3,2) (2,4) -- (3,4);
 \foreach \y/\k in {3/2,8/-2}
 \draw (0,\y) node {$\scriptscriptstyle{\bullet}$} node[left] {$\scriptstyle{t^{\k}}$};
 \foreach \y/\k in {1/,3/5,5/-2,7/3,9/-7}
 \draw (3,\y) node {$\scriptscriptstyle{\bullet}$} node[right] {$\scriptstyle{t^{\k}}$};
 \draw[->] (1.5,7) -- (1.2,6.8) node[below] {$\scriptstyle{t^{-1}}$};
 \draw[->] (2.8,4) -- (2.7,4) node[above] {$\scriptstyle{t}$};
 \draw[->] (2.4,4.8) -- (2.5,5) node[above] {$\scriptstyle{t^2}$};
 \draw[->] (2.5,3) -- (2.6,2.8) node[below] {$\scriptstyle{t^3}$};
 \draw (1.5,-1) node {$D_{\{2,3\}}$};
\end{scope}
\end{tikzpicture}
\caption{The diagrams $D_J$ defined from a w--beaded Jacobi diagram $D$}
\label{figCoproduct}
\end{center}
\end{figure}

We now define a Hopf algebra structure on $\tA(*_C)$. Define the product of two diagrams as the disjoint union. 
The unit $e:\Q\to\tA(*_C)$ is defined by $e(1)=\varnothing$ and the counit $\varepsilon:\tA(*_C)\to\Q$ is given by $\varepsilon(D)=0$ if $D\neq\varnothing$ and $\varepsilon(\varnothing)=1$. The antipode is given by $D\mapsto(-1)^{s}D$, where $s$ is the number of connected components in $D$. We finally have a structure of graded Hopf algebra on $\tA(*_C)$, where the grading is given by the degree. It is known that an element in a graded Hopf algebra is group-like if and only if it is the exponential of a {\em primitive} element, 
{\em ie} an element $G$ such that $\Delta(G)=1\otimes G+G\otimes 1$. 

\begin{lemma} \label{lemmaPrimitive}
 The primitive elements of $\tA(*_C)$ are the series of connected diagrams.
\end{lemma}
\begin{proof}
 Let $A$ be the rational vector space generated by beaded Jacobi diagrams on $C$, and let $R$ be the subspace of $A$ generated by the relations AS, IHX, LE, OR, Hol, so that $\tA(*_C)$ is the degree completion of $\tA=\fract AR$.
 We have a Hopf algebra structure on $A$ and $\tA$, given by the same definitions as for $\tA(*_C)$.
 In $\tA$ and in $A$, it is immediate that linear combination of connected diagrams are primitive. In $A$, the converse is also clear since, in $\Delta(D)$ for some linear combination $D$ of beaded Jacobi diagrams, a term $D_1\otimes D_2$ can only arise from the term $D_1\sqcup D_2$ in $D$, and no cancellation can occur in $A\otimes A$.

 To see that this also holds in $\tA$, we consider a complementary space $S$ of $R$ in $A$ such that, for every $\sum_i \alpha_i D_i\in S$, all subdiagrams of all $D_i$ are also in $S$. This can be done as follows, noting that the relations defining $R$ all act on a single connected component of a diagram: pick a maximal free family of connected diagrams $(d_i)$ such that no non trivial linear combination of the $d_i$'s lies in $R$, and define $S$ as the subspace generated by the $d_i$'s and their products. Now let $G$ be a primitive element in $\tA$, $\hat G$ its unique representative in $S\subset A$, and $\hat H=\Delta(\hat G)-(1\otimes \hat G+\hat G\otimes 1)$. By construction, $\hat H\in S\otimes S$; and since $G$ is primitive, we also have $\hat H\in A\otimes R+R\otimes A$. It follows that $\hat H=0$, meaning that $\hat G$ is primitive in $A$, hence a linear combination of connected diagrams, and thus so is $G$.

 By completion, the result then holds for $\tA(*_C)$.
\end{proof}

\subsection{Group-like elements}
As usual for similar constructions, it will appear that our invariant takes values in the set of group-like elements, which makes this property important to us.

\begin{lemma}
 Every group-like element $G\in\tA(*_{\{1,\dots,n\}})$ can be uniquely decomposed as $G= \expd\big(\frac{1}{2}W(t)\big)\sqcup H$, where $W(t)\in\tA(*_{\{1,\dots,n\}})$ is associated to a hermitian matrix of size $n$ with coefficients in $\Qtt$, also denoted $W(t)$, and $H$ is group-like and substantial.
\end{lemma}
\begin{proof}
 Since $G$ is group-like, it can be written as $G=\expd P$, where $P\in\tA(*_{\{1,\dots,n\}})$ is primitive. By Lemma~\ref{lemmaPrimitive}, $P$ is a series of connected diagrams. Write $P=S+T$ where $S$ is a series of struts and $T$ is substantial. Further write $S=\sum_{1\leq i\leq j\leq n}\struts ij{s_{ij}}$. For $i,j\in\{1,\dots,n\}$, we must have $s_{ij}=W_{ij}+\overline W_{ji}=2W_{ij}$. Set also $H=\expd T$. Then $G=\expd(\frac{1}{2}W)\sqcup H$, $W$ is hermitian, and $H$ is group-like and substantial.
 Now $P$ is the sum of all connected terms in $G$, so that $G$ fully determines $P$, and subsequently $S$, $T$, $W$ and $H$. 
\end{proof}

A group-like element $G= \expd\big(\frac{1}{2}W(t)\big)\sqcup H$ in $\tA(*_{\{1,\dots,n\}})$ is {\em non-degenerate} if $\det(W(t))\neq0$. In this case, we set $\omega(G)=\omega_{\frac12W}(H)$.

We define a {\em group-like link relation} by restricting the definition of the link relation. The relation $\eqlgl$ on $\tA(*_C)$ is generated by the following: given group-like elements $G_1,G_2\in\tA(*_C)$, we have $G_1\eqlgl G_2$ if, for some $c\in C$ and some extra vertices labeled by $h,\breve h$, there is an element $G\in\tA\left(*_{C\sqcup\{h,\breve h\}}\right)$, group-like and $(h-\breve h)$--substantial, such that $G_1=\langle G\rangle_{h-\breve h}$ and $G_2=\langle G_{|c\to ce^h}\rangle_{h-\breve h}$. 

We adapt similarly the link relation on $\tA(\xd_C)$ to define a {\em group-like link relation} on $\tA(\xd_C)$. Proposition~\ref{proplemmaGK} applies to these relations: for $G_1,G_2\in\tA(*_C)$, we have $G_1\eqlgl G_2$ if and only if $\chi_C(G_1)\eqlgl\chi_C(G_2)$. Note that the reverse implication of this equivalence will be reviewed in the proof of Lemma~\ref{lemmaGK}.

\subsection{Operation $\omega$}

This part is devoted to the definition of an operation on $\tA(*_{\{1,\dots,n\}})$ that will play, in our refinement of the Kricker invariant, the role of the formal Gaussian integration that was introduced in \cite[Section~2.2]{AA2} and used in the construction of the Kricker invariant~\cite{GK}. 

Let $W(t)$ be a hermitian matrix with coefficients in $\Qtt$ such that $\det\big(W(1)\big)\neq0$, where {\em hermitian} means that $^tW=\overline W$. Note that $W$ is invertible within the matrices with coefficients in $\Q(t)$. We shall associate with $W$ a Blanchfield module $(\Al,\bl)$ (see \cite{M1} for details). Define $\Al$ as a $\Qtt$--module with presentation matrix $^tW$: $\Al$ is given by generators $x_1,\dots,x_n$ and relations $\sum_{j=1}^nW_{ij}x_j=0$ for~$1\leq i\leq n$. Define the Blanchfield form on $\Al$ by $\bl(x_i,x_j)=-(W^{-1})_{ij}(t)\ mod\ \Qtt$. Given a beaded Jacobi diagram $D$ on $\{1,\dots,n\}$, we define $\omega_W(D)\in\A(\Al,\bl)$ as the class of the $(\Al,\bl)$--colored diagram obtained from $D$ by replacing the label $i$ on univalent vertices of $D$ by $x_i$ for each $i$, and fixing $f_{vv'}(t)=-(W^{-1})_{ij}(t)$ if the univalent vertices $v$ and $v'$ of $D$ are labeled by $i$ and $j$ respectively. 

A \emph{strut} is an isolated edge in a graph. 
To a square matrix $W$ of size $n$, we associate the sum of struts $\sum_{1\leq i,j\leq n}\struts ij{W_{ij}}\in\tA(*_{\{1,\dots,n\}})$; abusing notation, we denote again by $W$ this element of $\tA(*_{\{1,\dots,n\}})$. We say that a beaded Jacobi diagram on some finite set is {\em substantial} if it has no strut.

\begin{notation}
 For a Jacobi diagram $D$ on some finite set $C$, the subscript $D_{|x\to y}$ means that the label~$x\in C$ on univalent vertices of $D$ is replaced by the label $y\in C$. This should not be confused with the notation $D_{|c\to ce^h}$ defined before Proposition~\ref{proplemmaGK}.
\end{notation}

\begin{proposition} \label{propkey}
 Let $G_1=\expd\left(\frac{1}{2}W_1(t)\right)\sqcup H_1$ and  $G_2=\expd\left(\frac{1}{2}W_2(t)\right)\sqcup H_2$ be non-degenerate group-like elements in $\tA(*_{\{1,\dots,n\}})$.
 \begin{itemize}
  \item If $G_1\eqw G_2$, then $\omega(G_1)=\omega(G_2)$.
  \item If $G_1\eqlgl G_2$, then $W_1(t)=W_2(t)$ and $\omega(G_1)=\omega(G_2)$.
 \end{itemize}
\end{proposition}
The proof of the second point is based on technical lemmas that are postponed to Section~\ref{subsec:tech}.
\begin{proof}
 Assume $G_2$ is obtained from $G_1$ by pushing $t^k$ on the $i$--labeled vertices. Denote by $\textrm{Diag}_i(t)$ the diagonal matrix with a $t$ at the $i^{th}$ position and $1$'s elsewhere. We have $W_2(t)=\textrm{Diag}_i(t^k).W_1(t).\textrm{Diag}_i(t^{-k})$. Hence the map $\frac{\Qtt^n}{^tW_1\Qtt^n}\to\frac{\Qtt^n}{^tW_2\Qtt^n}$ that maps $x_i^{\scriptscriptstyle{(1)}}$ to $t^{-k}x_i^{\scriptscriptstyle{(2)}}$ and $x_j^{\scriptscriptstyle{(1)}}$ to $x_j^{\scriptscriptstyle{(2)}}$ is an isomorphism of Blanchfield modules. Pushing a $t^k$ on the $i$--labeled vertices of $H$ precisely applies this isomorphism to the univalent vertices of $\omega(G_1)$, thanks to the relation~EV. Hence $\omega(G_1)=\omega(G_2)$.
 
 Now assume $G_2$ is obtained from $G_1$ by a single group-like link relation on the $n$--labeled vertices.
 Lemma~\ref{lemmaGK} gives the equality $W_2=W_1$ and provides a group-like element $G=\expd\left(\frac12W\right)\sqcup H$ in $\tA(*_{\{1,\dots,n,h,\breve h\}})$, with $W=\begin{pmatrix}W_1&\bar\zeta&0\\ ^t\zeta&\lambda&0\\0&0&0\end{pmatrix}$, such that $G_1=\langle G\rangle_{h-\breve h}$ and $G_2=\langle G_{|n\to ne^{\breve h}}\rangle_{h-\breve h}$. 
 We wish to prove $\omega(G_1)=\omega(G_2)$ using the expressions of $G_1$ and $G_2$ in terms of $G$. Since the matrix $W$ is singular, the operation $\omega_W$ is not defined. To overcome this, we introduce another matrix $\hW=\begin{pmatrix}W_1&\bar\zeta&0\\ ^t\zeta&\lambda&-1\\0&-1&0\end{pmatrix}$. 
 Let us show that the Blanchfield modules $(\widehat\Al,\widehat\bl)$ and $(\Al_1,\bl_1)$ defined by the matrices $\hW$ and $W_1$ respectively are isomorphic. The module $\widehat\Al$ associated to $\hW$ has a presentation with generators $x_1,\dots,x_n,x_{h},x_{\breve h}$ and relations given by the rows of $\hW$.
 The last two rows of $\hW$ give $x_{h}=0$ and $x_{\breve h}=\sum_{i=1}^n\zeta_ix_i$. It follows that $\widehat\Al$ is generated by $x_1,\dots,x_n$, which satisfy the relations given by the rows of $W_1$. This provides an isomorphism between $\widehat\Al$ and $\Al_1$. Moreover, since $\hW^{-1}=\begin{pmatrix}W_1^{-1}&0&W_1^{-1}\bar\zeta\\ 0&0&-1\\ ^t\zeta W_1^{-1}&-1& ^t\zeta W_1^{-1}\bar\zeta-\lambda\end{pmatrix}$, this isomorphism identifies the Blanchfield forms $\widehat\bl$ and~$\bl_1$.

 We have $G_1=\langle \expd(\frac{1}{2}W)\sqcup H\rangle_{h-{\breve h}}$, so Lemma~\ref{lemma:tech1} gives $\omega(G_1)=\omega_{\hW}(H)$. 
 Also, we have $G_2=\left\langle G_{|n\to ne^h}\right\rangle_{\breve h-h}$, where $G_{|n\to ne^{\breve h}}=\Big(\expd\left(\frac12W\right)\sqcup H\Big)_{|n\to ne^{\breve h}}=\expd\left(\frac12W_{|n\to ne^{\breve h}}\right)\sqcup H_{|n\to ne^{\breve h}}.$ Now $\frac12W_{|n\to ne^{\breve h}}=\frac12W+J$ where $J$ is substantial. It follows that $G_{|n\to ne^{\breve h}}=\expd\left(\frac12W\right)\sqcup H'$, where $H'=\expd J\sqcup H_{|n\to ne^{\breve h}}$ is substantial. By Lemma~\ref{lemma:tech1}, we conclude that $\omega(G_2)=\omega_{\hW}(H')$.
 Now, Lemma~\ref{lemma:tech2} gives $\omega_{\hW}(H)=\omega_{\hW}(H')$, so that finally $\omega(G_1)=\omega(G_2)$.
\end{proof}

\subsection{Technical lemmas} \label{subsec:tech}

This section gathers technical lemmas on the value of $\omega$ for group-like elements satisfying certain relations. Lemmas~\ref{lemmaGK}, \ref{lemma:tech1} and~\ref{lemma:tech2} are needed in the proof of Proposition~\ref{propkey} and Lemmas~\ref{lemmaK2omega1} and~\ref{lemmaK2omega2} will be used in the proof of Proposition~\ref{prop:invariance}.

\begin{notation}
  We define the following exponential notation on diagrams: $$\raisebox{-0.1cm}{
\begin{tikzpicture} [scale=0.25]
 \draw (0,0) -- (5,0);
 \draw[dashed] (-2,0) -- (0,0) (5,0) -- (7,0);
 \foreach \x in {1,4} \draw (\x,0) -- (\x,1.5);
 \draw (2.5,0.8) node {$\scriptstyle \dots$}; \draw (2.5,2.5) node {$\scriptstyle {e^h}$};
\end{tikzpicture}}
=\sum_{s\geq0} \frac1{s!} \raisebox{-0.1cm}{
\begin{tikzpicture} [scale=0.25]
 \draw (0,0) -- (5,0);
 \draw[dashed] (-2,0) -- (0,0) (5,0) -- (7,0);
 \foreach \x in {1,4} \draw (\x,0) -- (\x,1.5) (\x,1.2) node[above] {$\scriptstyle h$};
 \draw (2.5,0.8) node {$\scriptstyle \dots$}; \draw[decorate,decoration={brace,amplitude=5pt}] (0.6,2.8) -- (4.4,2.8) node [midway,above,yshift=0.1cm] {$\scriptscriptstyle {s\text{ times}}$};
\end{tikzpicture}}.$$ 
 A term $e^h-1$ instead of $e^h$ means that the sum is over $s>0$, whereas a term $e^{-h}$ means that the $h$--ended edges are added on the other side of the supporting edge. Thanks to the relation AS, this latter variant is equivalent to replacing the $\frac1{s!}$ factor by $\frac{(-1)^s}{s!}$; note in particular that consecutive notations $e^h$ and $e^{-h}$ cancel each other, namely just amount to the supporting edge without $h$--ended edge.
\end{notation}

We start with two lemmas that give identities involving this exponential notation, which will be useful in the proof of Lemma~\ref{lemmaGK}.

\begin{lemma} \label{lemmaSTUexp}
\begin{tikzpicture} [scale=0.25]
 \draw[very thick,->] (0,0) -- (10,0);
 \foreach \x in {1,3,4,6,7,9} \draw (\x,0) -- (\x,1);
 \foreach \x/\l in {2/h,5/k,8/-h} {
 \draw (\x,0.5) node {$\scriptstyle \dots$};
 \draw (\x,1.8) node {$\scriptstyle e^{\l}$};}
\end{tikzpicture}
$=$
\begin{tikzpicture} [scale=0.25]
 \draw[very thick,->] (0,0) -- (10,0);
 \foreach \x in {3,8} {
 \draw (\x,0) -- (\x,4);
 \foreach \y in {1,3} \draw (\x,\y) -- (\x-1,\y);
 \draw (\x-0.5,2.3) node {$\scriptstyle \vdots$};
 \draw (\x-1.7,2) node {$\scriptstyle e^{h}$};}
 \draw[decorate,decoration={brace,amplitude=4pt}] (2.8,4.1) -- (8.2,4.1) node [midway,above,yshift=0.1cm] {$\scriptstyle e^{k}$};
 \draw (4.3,1.5) node {$\dots$};
\end{tikzpicture}
\end{lemma}
\begin{proof}
 We start with the relation STU 
\begin{tikzpicture} [scale=0.25]
 \draw[very thick,->] (-1,0) -- (3,0);
 \draw (1,0) -- (1,2) node[above] {$\scriptstyle k$};
 \draw (0,1) node[left] {$\scriptstyle h$} -- (1,1);
\end{tikzpicture}
 $=$
\begin{tikzpicture} [scale=0.25]
 \draw[very thick,->] (-1,0) -- (3,0);
 \foreach \x/\n in {0/h,1.7/k} 
 \draw (\x,0) -- (\x,2) node[above] {$\scriptstyle \n$};
\end{tikzpicture}
 $-$
\begin{tikzpicture} [scale=0.25]
 \draw[very thick,->] (-1,0) -- (3,0);
 \foreach \x/\n in {0/k,1.7/h} 
 \draw (\x,0) -- (\x,2) node[above] {$\scriptstyle \n$};
\end{tikzpicture}.
 Iterating this relation, we get 
\begin{tikzpicture} [scale=0.25]
 \draw[very thick,->] (0,0) -- (5,0);
 \draw (3,0) -- (3,4) node[above] {$\scriptstyle k$};
 \foreach \y in {1,3} \draw (3,\y) -- (2,\y) node[left] {$\scriptstyle h$};
 \draw (2.5,2.3) node {$\scriptstyle \vdots$};
 \draw[decorate,decoration={brace,amplitude=2pt}] (3.5,3) -- (3.5,1) node [midway,right] {$\scriptstyle {s}$};
\end{tikzpicture}
 $\displaystyle =\sum_{t=0}^s(-1)^t\binom st$
\raisebox{-0.7cm}{
\begin{tikzpicture} [scale=0.25]
 \draw[very thick,->] (0,0) -- (8,0);
 \foreach \x/\n in {1/h,3/h,4/k,5/h,7/h}
 \draw (\x,0) -- (\x,2) node[above] {$\scriptstyle \n$};
 \foreach \x/\n in {2/s-t,6/t} {
 \draw (\x,1) node {$\scriptstyle \dots$};
 \draw[decorate,decoration={brace,amplitude=2pt}] (\x+1,-0.5) -- (\x-1,-0.5) node [midway,below] {$\scriptstyle {\n}$};}
\end{tikzpicture}},
 which gives 
\begin{tikzpicture} [scale=0.25]
 \draw[very thick,->] (0,0) -- (5,0);
 \draw (3,0) -- (3,4) node[above] {$\scriptstyle k$};
 \foreach \y in {1,3} \draw (3,\y) -- (2,\y);
 \draw (2.5,2.3) node {$\scriptstyle \vdots$};
 \draw (1.3,2) node {$\scriptstyle e^{h}$};
\end{tikzpicture}
 $=$
\begin{tikzpicture} [scale=0.25]
 \draw[very thick,->] (0,0) -- (8,0);
 \foreach \x in {1,3,4,5,7}
 \draw (\x,0) -- (\x,2);
 \foreach \x/\n in {2/e^h,4/k,6/\,e^{-h}}
 \draw (\x,2) node[above] {$\scriptstyle {\n}$};
 \foreach \x in {2,6}
 \draw (\x,1) node {$\scriptstyle \dots$};
\end{tikzpicture}. 
 Thanks to the equality  
\begin{tikzpicture} [scale=0.25]
 \draw[very thick,->] (0,0) -- (7,0);
 \foreach \x in {1,3,4,6}
 \draw (\x,0) -- (\x,2);
 \foreach \x/\n in {2/e^h,5/\,e^{-h}}
 \draw (\x,2) node[above] {$\scriptstyle {\n}$};
 \foreach \x in {2,5}
 \draw (\x,1) node {$\scriptstyle \dots$};
\end{tikzpicture}
 $=$
\begin{tikzpicture} [scale=0.25]
 \draw[very thick,->] (0,0) -- (2,0);
\end{tikzpicture},
 this implies 
\raisebox{-0.6cm}{
\begin{tikzpicture} [scale=0.25]
 \draw[very thick,->] (0,0) -- (10,0);
 \foreach \x in {3,8} {
 \draw (\x,0) -- (\x,4) node[above] {$\scriptstyle k$};
 \foreach \y in {1,3} \draw (\x,\y) -- (\x-1,\y);
 \draw (\x-0.5,2.3) node {$\scriptstyle \vdots$};
 \draw (\x-1.7,2) node {$\scriptstyle e^{h}$};}
 \draw (4.3,1.5) node {$\dots$};
 \draw[decorate,decoration={brace,amplitude=4pt}] (8,-0.5) -- (3,-0.5) node [midway,below,yshift=-0.1cm] {$\scriptstyle s$};
\end{tikzpicture}}
 $=$
\raisebox{-0.5cm}{
\begin{tikzpicture} [scale=0.25]
 \draw[very thick,->] (0,0) -- (10,0);
 \foreach \x in {1,3,4,6,7,9}
 \draw (\x,0) -- (\x,2);
 \foreach \x/\n in {2/e^h,4/k,6/k,8/\,e^{-h}}
 \draw (\x,2) node[above] {$\scriptstyle {\n}$};
 \foreach \x in {2,5,8}
 \draw (\x,1) node {$\scriptstyle \dots$};
 \draw[decorate,decoration={brace,amplitude=2pt}] (6,-0.5) -- (4,-0.5) node [midway,below] {$\scriptstyle s$};
\end{tikzpicture}},
which finally leads to the desired relation.
\end{proof}

\begin{lemma} \label{lemmaIHXexp}
 \newcommand{\armexp}{
 \draw (0,0) -- (0,-4);
 \foreach \y in {-1,-3}
 \draw (0,\y) -- (1,\y);
 \draw (1.7,-2) node {$\scriptstyle {e^h}$};
 \foreach \y in {-1.5,-2,-2.5}
 \draw (0.5,\y) node {$\scriptstyle \cdot$};}
 \raisebox{-0.4cm}{
\begin{tikzpicture} [scale=0.2]
 \foreach \x in {0,12,24}
 \draw (\x,0) -- (\x,-4);
 \foreach \y in {-1,-3} 
 \draw (0,\y) -- (-1,\y);
 \draw (-1,-2) node[left] {$\scriptstyle {e^h}$};
 \draw (-0.5,-1.5) node {$\scriptstyle \vdots$};
 \foreach \t in {120,240}
 \draw[rotate=\t] (0,0) -- (0,-4);
 \draw (6,-1.5) node {$=$};
 \draw (18,-1.5) node {$+$};
 \foreach \x/\t in {24/120,12/240} {
 \begin{scope} [xshift=\x cm,rotate=\t]
  \armexp
 \end{scope}}
 \foreach \x/\t in {12/120,24/240} 
 \draw[xshift=\x cm,rotate=\t] (0,0) -- (0,-4);
\end{tikzpicture}}
\end{lemma}
\begin{proof}
 We start with the relation IHX 
\raisebox{-0.2cm}{
\begin{tikzpicture} [scale=0.5]
 \foreach \x in {0,3,6.5} {
 \begin{scope} [xshift=\x cm]
  \foreach \t in {0,120,240}
  \draw[rotate=\t] (0,0) -- (0,-1);
 \end{scope}}
 \draw[rotate=240] (0,-0.5) -- (0.5,-0.5) node[left] {$\scriptstyle{h}$};
 \draw[xshift=3cm,rotate=120] (0,-0.5) -- (0.5,-0.5) node[above] {$\scriptstyle{h}$};
 \draw (6.5,-0.5) -- (6,-0.5) node[left] {$\scriptstyle{h}$};
 \draw (1.5,-0.2) node {$+$} (4.5,-0.2) node {$=$};
\end{tikzpicture}
}. Iterating, this gives 
\newcommand{\armh}{
 \draw (0,0) -- (0,-4);
 \foreach \y in {-1,-3} {
 \draw (0,\y) -- (1,\y);
 \draw (1.7,\y) node {$\scriptstyle h$};}
 \foreach \y in {-1.5,-2,-2.5}
 \draw (0.5,\y) node {$\scriptstyle \cdot$};
 \draw[decorate,decoration={brace,amplitude=2pt}] (-0.4,-3) -- (-0.4,-1);}
\[
\begin{tikzpicture} [scale=0.25]
 \foreach \x in {0,17}
 \draw (\x,0) -- (\x,-4);
 \foreach \y in {-1,-3} 
 \draw (0,\y) -- (-1,\y) node[left] {$\scriptstyle h$};
 \draw (-0.5,-1.7) node {$\scriptstyle \vdots$};
 \draw[decorate,decoration={brace,amplitude=2pt}] (0.5,-1) -- (0.5,-3) node [midway,right] {$\scriptstyle {s}$};
 \foreach \t in {120,240}
 \draw[rotate=\t] (0,0) -- (0,-4);
 \draw (8,-1) node {$\displaystyle =\sum_{t=0}^s\binom st$};
 \foreach \t/\e in {120/s-t,240/s} {
 \begin{scope} [xshift=17cm,rotate=\t]
  \armh
  \draw (-1.3,-2) node {$\scriptstyle {\e}$};
 \end{scope}}
\end{tikzpicture},
\]
from which we get the desired relation.
\end{proof}

\begin{lemma}\label{lemmaGK}
 Let $G_1=\expd(\frac{1}{2}W_1)\sqcup H_1$ and
 $G_2=\expd(\frac{1}{2}W_2)\sqcup H_2$ be non-degenerate group-like elements in $\tA(*_{\{1,\dots,n\}})$. Assume $\chi_{\{1,\dots,n\}} (G_1)$ and $\chi_{\{1,\dots,n\}} (G_2)$ are related by one group-like link relation on $\xd_i$. Then $W_1=W_2$ and $G_1\eqlgl G_2$. More precisely, there is a group-like element $G=\expd(\frac{1}{2}W)\sqcup H$ in $\tA(*_{\{1,\dots,n,h,\breve h\}})$ such that $G_1=\langle G\rangle_{h-\breve h}$ and $G_2=\langle G_{|i\to ie^{\breve h}}\rangle_{h-\breve h}$, with $W$ of the form $\begin{pmatrix}W_1&\bar\zeta&0\\ ^t\zeta&\lambda&0\\0&0&0\end{pmatrix}$ where $\bar\zeta(t)=\zeta(t^{-1})$.
\end{lemma}
\begin{proof}
\newcommand{\hairbox}[2]{
\begin{tikzpicture} [scale=0.25]
 \draw (0,0) -- (7,0) -- (7,2) -- (0,2) -- (0,0);
 \draw (3.5,1) node {$\scriptstyle D_j^{#1}$};
 \foreach \x in {1,3,4,6} \draw (\x,0) -- (\x,-1);
 \foreach \x/\l in {2/k,5/#2} {
 \draw (\x,-0.5) node {$\scriptstyle \dots$};
 \draw (\x,-1.5) node {$\scriptscriptstyle {\l}$};}
\end{tikzpicture}}
\newcommand{\hairexp}[2]{
\begin{tikzpicture} [scale=0.25]
 \draw[very thick,->] (0,0) -- (4,0) node[below] {$\scriptstyle #1$};
 \foreach \x in {1,3} \draw (\x,0) -- (\x,1);
 \draw (2,0.5) node {$\scriptstyle \dots$};
 \draw (2,1.8) node {$\scriptstyle e^{\breve #2}$};
\end{tikzpicture}}
\newcommand{\hairexpd}[2]{
\begin{tikzpicture} [scale=0.25]
 \draw[very thick,->] (0,0) -- (7,0) node[below] {$\scriptstyle n$};
 \foreach \x in {1,3,4,6} \draw (\x,0) -- (\x,1);
 \foreach \x/\l in {2/#1,5/#2} {
 \draw (\x,0.5) node {$\scriptstyle \dots$};
 \draw (\x,1.8) node {$\scriptstyle e^{\breve \l}$};}
\end{tikzpicture}}
\newcommand{\hairexpq}{
\begin{tikzpicture} [scale=0.25]
 \draw[very thick,->] (0,0) -- (13,0) node[below] {$\scriptstyle n$};
 \foreach \x in {1,3,4,6,7,9,10,12} \draw (\x,0) -- (\x,1);
 \foreach \x/\l in {2/\breve h,5/\breve k,8/-\breve h,11/\breve h} {
 \draw (\x,0.5) node {$\scriptstyle \dots$};
 \draw (\x,1.8) node {$\scriptstyle e^{\l}$};}
\end{tikzpicture}}
\newcommand{\hairexpc}{
\begin{tikzpicture} [scale=0.25]
 \draw[very thick,->] (0,0) -- (12,0) node[below] {$\scriptstyle n$};
 \foreach \x in {3,8} {
 \draw (\x,0) -- (\x,4);
 \foreach \y in {1,3} \draw (\x,\y) -- (\x-1,\y);
 \draw (\x-0.5,2.3) node {$\scriptstyle \vdots$};
 \draw (\x-1.7,2) node {$\scriptstyle e^{\breve h}$};}
 \draw[decorate,decoration={brace,amplitude=4pt}] (2.8,4.1) -- (8.2,4.1) node [midway,above,yshift=0.1cm] {$\scriptstyle e^{\breve k}$};
 \foreach \x in {9,11} \draw (\x,0) -- (\x,1);
 \draw (4.3,1.5) node {$\dots$};
 \draw (10,0.5) node {$\scriptstyle \dots$};
 \draw (10,1.8) node {$\scriptstyle e^{\breve h}$};
\end{tikzpicture}}
\newcommand{\hairexpv}[1]{\hspace{-0.5ex}
\begin{tikzpicture} [scale=0.2]
 \draw (0,0) node[below] {$\scriptscriptstyle n$} -- (0,4) node[above] {$\scriptscriptstyle \breve h#1$};
 \foreach \y in {1,3} \draw (0,\y) -- (-1,\y);
 \draw (-0.5,2.5) node {$\scriptstyle \vdots$};
 \draw (-1.8,2) node {$\scriptstyle e^{\breve h}$};
\end{tikzpicture}}
\renewcommand{\strut}[1]{\hspace{-0.5ex}
\begin{tikzpicture} [scale=0.2]
\draw (0,0) node[below] {$\scriptstyle n$} -- (0,4) node[above] {$\scriptstyle\breve #1$};
\end{tikzpicture}}
\newcommand{\sstrut}[1]{\hspace{-0.7ex}
\begin{tikzpicture} [scale=0.1]
\draw (0,0) node[below] {$\scriptscriptstyle n$} -- (0,4) node[above] {$\scriptscriptstyle\breve #1$};
\end{tikzpicture}}
\newcommand{\y}{
\begin{tikzpicture} [scale=0.2]
\draw (0,0) node[below] {$\scriptstyle n$} -- (0,2) -- (-2,4) node[above] {$\scriptstyle\breve k$} (0,2) -- (2,4) node[above] {$\scriptstyle\breve h$};
\end{tikzpicture}}
\newcommand{\yd}[1]{
\begin{tikzpicture} [scale=0.2]
\draw (0,0) node[below] {$\scriptstyle n$} -- (0,2) -- (-2,4) node[above] {$\scriptstyle\breve #1$} (0,2) -- (2,4) node[above] {$\scriptstyle\breve h$} (1,3) -- (0,4) node[above] {$\scriptstyle\breve k$};
\end{tikzpicture}}
 The proof follows very closely that of \cite[Lemma~3.6]{GK}. Up to relabeling, we assume that $i=n$. There is a group-like element $D\in\tA\left(\xd_{\{1,\dots,n-1,n_1,n_2\}}\right)$ such that $\chi_{\{1,\dots,n\}} (G_1)=m_n^{n_2n_1}(D)$ and $\chi_{\{1,\dots,n\}} (G_2)=m_n^{n_1n_2}(D)$, where, on each diagram, $m^{n_2n_1}_n$ glues the head of $\xd_{n_2}$ to the tail of $\xd_{n_1}$ to form $\xd_n$, and $m_n^{n_1n_2}$ is defined similarly. Hence $\chi_{\{n\}}(G_1)=m_n^{n_2n_1}\circ\chi^{-1}_{\{1,\dots,n-1\}}(D)$ and $\chi_{\{n\}}(G_2)=m_n^{n_1n_2}\circ\chi^{-1}_{\{1,\dots,n-1\}}(D)$. Writing $\chi_{\{1,\dots,n-1,n_1,n_2\}}^{-1}(D)_{\left|\substack{n_1\to h\\n_2\to k}\right.}=\sum \alpha_j D_j$ with, for each~$j$, $\alpha_j\in\Q$ and $D_j$ a beaded Jacobi diagram on $\{1,\dots,n-1,h,k\}$, we have:
 \[
 \chi^{-1}_{\{1,\dots,n-1\}}(D)=\chi_{\{n_1,n_2\}}\circ \chi_{\{1,\dots,n-1,n_1,n_2\}}^{-1}(D)
 =\left\langle\sum\alpha_j\raisebox{-0.6cm}{\hairbox{}{h}}\ \raisebox{-0.5cm}{\hairexp{n_2}{k}\hairexp{n_1}{h}}\right\rangle_{\substack{h-\breve h\\k-\breve k}},
 \]
 and hence
 \[
 \hspace{-.5cm}
 \chi_{\{n\}} (G_1)=\left\langle\sum\alpha_j\raisebox{-0.6cm}{\hairbox{}{h}}\ \raisebox{-0.5cm}{\hairexpd{k}{h}}\right\rangle_{\substack{h-\breve h\\k-\breve k}}.
 \]
 The bracket in the right hand side replaces the diagram $D_j$ with the sum over all possible ways of gluing the $h$--labeled vertices of $D_j$ on the ``head half'' of $\xd_n$ and the $k$--labeled vertices of $D_j$ on the ``tail half'' of $\xd_n$. So basically, we have
 \[
 \chi_{\{n\}} (G_1)=m^{kh}_n\circ\chi_{\{k,h\}}\left(\sum\alpha_j\raisebox{-0.6cm}{\hairbox{}{h}}\right).
 \]
Then, we can use \cite[Prop. 5.4]{AA2} to write
\[
G_1=\left\langle\expd\Big(\Lambda^{\breve k\breve h}_n\Big)\sqcup \sum\alpha_j\raisebox{-0.6cm}{\hairbox{}{h}}\,\right\rangle_{\substack{h-\breve
h\\k-\breve k}}
\]
where $\Lambda^{\breve k\breve h}_n$ is the Baker--Campbell--Hausdorff sum
\begin{equation} \tag{$*$}
\Lambda^{\breve k\breve h}_n=\raisebox{-0.7cm}{\strut k}+\raisebox{-0.7cm}{\strut h}+\frac12 \raisebox{-0.7cm}{\y}+\frac1{12}\raisebox{-0.6cm}{\yd k}-\frac1{12}\raisebox{-0.7cm}{\yd h}+\cdots.
\label{eq:BCH}
\end{equation}
The strut $\struts n{\breve h}{}$ is the only term in $\Lambda^{\breve k\breve h}_n$ with no $\breve k$--labeled vertex. For this reason, we will see when working out an expression for $G_2$ that it behaves differently, and we put it apart.
From the above expression for $G_1$, we obtain
\[
G_1=
\left\langle\expd\bigg(\Lambda^{\breve k\breve
    h}_n-\raisebox{-0.45cm}{\sstrut h}\bigg)\sqcup
  \expd\bigg(\raisebox{-0.45cm}{\sstrut {h'}}\bigg)\sqcup
  \sum_{j,p}\alpha_j\raisebox{-0.6cm}{\hairbox{p}{{h,h'}}}\,\right\rangle_{\substack{h'-\breve h'\\h-\breve h\\k-\breve k}},
\]
where, for each $j$, the $D_j^p$ are all the diagrams obtained from $D$ by labeling each $h$--labeled vertex of $D$ with either $h$ or $h'$.
Setting
\[
  G=
  \left\langle\expd\bigg(\Lambda^{\breve k\breve
    h}_n-\raisebox{-0.45cm}{\sstrut h}\bigg)\sqcup
  \expd\bigg(\raisebox{-0.45cm}{\sstrut {h'}}\bigg)\sqcup
  \sum_{j,p}\alpha_j\raisebox{-0.6cm}{\hairbox{p}{{h,h'}}}\,\right\rangle_{\substack{h'-\breve h'\\k-\breve k}},
\]
we obtain $G_1=\langle G\rangle_{h-\breve h}$.

Reviewing the construction of $G$, we see that it is group-like: $D$ is group-like, so that the sum $\sum_j\alpha_j\raisebox{-0.6cm}{\hairbox{}{h}}\,$, and further the sum $\sum_{j,p}\alpha_j\raisebox{-0.6cm}{\hairbox{p}{{h,h'}}}\,$, are group-like, and finally $G$ is group-like. 
Moreover, observe that $G$ contains no strut with an $\breve h$--labeled end. 
Hence we can write $G=\expd(\frac{1}{2}W)\sqcup H$ with $W=\begin{pmatrix}W_0&\bar\zeta&0\\ ^t\zeta&\lambda&0\\0&0&0\end{pmatrix}$. 
Now $G_1=\left\langle \expd\left(\frac{1}{2}W\right)\sqcup H\right\rangle_{h-\breve h}$ gives 
$$G_1=\left\langle \expd\left(\frac{1}{2}W_0\right)\sqcup\expd\left(\sum_{j=1}^n\struts hj{\zeta_j}+\frac12\struts hh\lambda\right)\sqcup H\right\rangle_{h-\breve h}=\expd\left(\frac{1}{2}W_0\right)\sqcup H'$$
with $H'$ substantial. It follows that $W_0=W_1$, which finally provides the required expression for~$G$.

We now consider $G_2$. As for $G_1$, we have
\[
\chi_{\{n\}} (G_2)=\left\langle\sum\alpha_j\raisebox{-0.6cm}{\hairbox{}{h}}\ \raisebox{-0.5cm}{\hairexpd hk}\right\rangle_{\substack{h-\breve h\\k-\breve k}}.
\]
Further, we have by lemma~\ref{lemmaSTUexp}:
\[
\raisebox{-0.5cm}{\hairexpd hk}=\raisebox{-0.5cm}{\hairexpq}=\raisebox{-0.5cm}{\hairexpc},
\]
so
\begin{align*}
 \chi_{\{n\}} (G_2)= &
 \left\langle\sum\alpha_j\raisebox{-0.6cm}{\hairbox{}{h}}\ \raisebox{-0.8cm}{\hairexpc}\right\rangle_{\substack{h-\breve h\\k-\breve k}} \\
 = & 
 \left\langle\sum\alpha_j\raisebox{-0.6cm}{\hairbox{\breve h}{h}}\ \raisebox{-0.5cm}{\hairexpd{k}{h}}\right\rangle_{\substack{h-\breve h\\k-\breve k}}
\end{align*}
where $D_j^{\breve h}$ is obtained from $D_j$ by pushing $e^{\breve h}$ on each $k$--labeled vertex. As above, we get then
\[
G_2=\left\langle\expd\Big(\Lambda^{\breve k\breve h}_n\Big)\sqcup \sum\alpha_j\raisebox{-0.6cm}{\hairbox{\breve h}{h}}\,\right\rangle_{\substack{h-\breve
h\\k-\breve k}}.
\]
Now, the $e^{\breve h}$ next to the $k$--labeled vertices can be pushed on the $\breve k$--labeled vertices of $\expd\left(\Lambda_n^{\breve k\breve h}\right)$, where they become $e^{-\breve h}$. Some $e^{-\breve h}$ can also be freely added next to the $\breve h$--labeled vertices of
$\expd\left(\Lambda_n^{\breve k\breve h}\right)$ as all the diagrams corresponding to the non trivial terms in $e^{-\breve h}$ will vanish thanks to the relation AS. All these $e^{-\breve h}$ can then be pushed down through $\Lambda^{\breve k\breve h}$ thanks to Lemma~\ref{lemmaIHXexp}, so that
\[
G_2=\left\langle \expd\left(\widetilde\Lambda^{\breve k\breve h}_n\right)\sqcup \sum\alpha_j\raisebox{-0.6cm}{\hairbox{}{h}}\,\right\rangle_{\substack{h-\breve
h\\k-\breve k}}
\]
where $\widetilde \Lambda^{\breve k\breve h}_n$ is obtained from $\Lambda^{\breve k\breve h}_n$ by pushing $e^{\breve h}$ on each $n$--labeled vertex. This can be rewritten as
\[
G_2=\Bigg\langle \bigg\langle \expd\bigg(\widetilde\Lambda^{\breve k\breve h}_n-\raisebox{-0.6cm}{\hairexpv{}}\bigg)\sqcup\expd\bigg(\raisebox{-0.6cm}{\hairexpv{'}}\bigg)\sqcup \sum\alpha_j\raisebox{-0.6cm}{\hairbox{p}{{h,h'}}}\,\bigg\rangle_{\substack{h'-\breve
h'\\k-\breve k}}\Bigg\rangle_{h-\breve h},
\]
that is $G_2=\left\langle G_{|n\to ne^{\breve h}}\right\rangle_{h-\breve h}$.

From $G=\expd\left(\frac{1}{2}W\right)\sqcup H$, we get $G_{|n\to ne^{\breve h}}=\expd\left(\frac{1}{2}W\right)\sqcup H''$ with $H''$ substantial. Hence the above proof of $W_0=W_1$ applies again and gives $W_0=W_2$. Thus $W_1=W_2$.
\end{proof}

\begin{lemma} \label{lemma:tech1}
 Assume $G_1=\expd(\frac{1}{2}W_1)\sqcup H_1$ is a group-like element in $\tA(*_{\{1,\dots,n\}})$ and $G=\expd(\frac{1}{2}W)\sqcup H$ is a group-like element in $\tA(*_{\{1,\dots,n,h,\breve h\}})$ such that $G_1=\langle G\rangle_{h-{\breve h}}$ and $W=\begin{pmatrix} W_1&\bar\zeta&0\\^t\zeta&\lambda&0\\0&0&0 \end{pmatrix}$. Set $\hW=\begin{pmatrix} W_1&\bar\zeta&0\\^t\zeta&\lambda&-1\\0&-1&0 \end{pmatrix}$. Then $\omega(G_1)=\omega_{\hW}(H)$.
\end{lemma}
\begin{proof}
 We have $G_1=\langle \expd(\frac{1}{2}W)\sqcup H\rangle_{h-{\breve h}}$, which gives:
 \begin{align*}
  G_1&=\left\langle \expd\left(\frac{1}{2}W_1\right)\sqcup\expd\left(\sum_{i=1}^n \struts hi{\zeta_i}+\frac12 \struts hh\lambda\right) \sqcup H\right\rangle_{h-{\breve h}}\\
  &=\expd\left(\frac{1}{2}W_1\right)\sqcup\left\langle\expd\left(\sum_{i=1}^n \struts hi{\zeta_i}+\frac12 \struts hh\lambda\right) \sqcup H\right\rangle_{h-{\breve h}}.
 \end{align*}
 Hence $H_1=\left\langle\expd\left(\sum_{i=1}^n \struts hi{\zeta_i}+\frac12 \struts hh\lambda\right) \sqcup H\right\rangle_{h-{\breve h}}$, and we need to prove the equality $\omega_{W_1}(H_1)=\omega_{\hW}(H)$. The operation $\omega$, as well as the contraction, are applied on each diagram of a series. Hence it suffices to prove the result for each summand, so we assume here that $H$ is a single diagram.
 Recall that, when applying the operation $\omega_{\hW}$, the linkings are given by the matrix $-\hW^{-1}=\begin{pmatrix}-W_1^{-1}&0&-W_1^{-1}\bar\zeta\\ 0&0&1\\ -^t\zeta W_1^{-1}&1& \lambda-{ }^t\zeta W_1^{-1}\bar\zeta\end{pmatrix}$.
 
 Let $V$ be the set of univalent vertices of $\omega_{\hW}(H)$, and denote $V_h,V_{\breve h}\subset V$ its subsets of vertices labeled by $x_h=0$ and $x_{\breve h}$ respectively. We apply generalized relations LD at each vertex in~$V_{\breve h}$, using the equality $x_{\breve h}=0+\sum_{i=1}^n\zeta_ix_i$. It shows that $\omega_{\hW}(H)$ is the sum of all diagrams obtained from $\omega_{\hW}(H)$ by labeling the vertices in $V_{\breve h}$ by $0$ or $\zeta_ix_i$ for some $i=1,\dots,n$, and setting the following linkings for $v\in V_{\breve h}$ and $w\in V$:
 \begin{itemize}
  \item if $v$ is $0$--labeled,
  \begin{itemize}
   \item $f_{vw}=1$ if $w\in V_h$,
   \item $f_{vw}=\lambda$ if $w\in V_{\breve h}$ is $0$--labeled,
   \item $f_{vw}=0$ otherwise,
  \end{itemize}
  \item if $v$ is $\zeta_ix_i$--labeled,
  \begin{itemize}
   \item $f_{vw}=0$ if $w\in V_h$ or $w\in V_{\breve h}$ is $0$--labeled,
   \item $f_{vw}=-(W_1^{-1})_{ij}\zeta_i\bar\zeta_j$ if $w\in V_{\breve h}$ is $\zeta_jx_j$--labeled,
   \item $f_{vw}=-(W_1^{-1})_{ij}\zeta_i$ if $w\in V\setminus(V_h\cup V_{\breve h})$ is $x_j$--labeled. 
  \end{itemize}
 \end{itemize}
 Then we apply a relation EV at each $\zeta_ix_i$--labeled vertex $v\in V_{\breve h}$, so that the label of the vertex becomes $x_i$. At this stage, for all pair of vertices $v,w\in V$ labeled by $x_i$ and $x_j$ respectively, we have $f_{vw}=-(W_1^{-1})_{ij}$. Finally, we apply relations LD at each pair of $0$--labeled vertices $v,w\in V_h\cup V_{\breve h}$ such that $f_{vw}=0$, until there is no more such pairs of vertices. 
 
 If there are initially more $h$--labeled vertices than $\breve h$--labeled ones in $H$, then some $0$--labeled vertex trivially linked to all other vertices remains in each diagram of the sum, and the relation LV shows that $\omega_{\hW}(H)=0$. In this case, we also have $\omega_{W_1}(H_1)=0$.
 
 Otherwise, we obtain that $\omega_{\hW}(H)$ equals the sum of all diagrams obtained from $\omega_{\hW}(H)$ by:
 \begin{itemize}
  \item pairing each vertex in $V_h$ with a vertex in $V_{\breve h}$,
  \item pairing some vertices in $V_{\breve h}$ together, adding a $\lambda$ on the involved edge,
  \item labeling the remaining vertices in $V_{\breve h}$ by $x_i$ for some $i=1,\dots,n$, adding a $\zeta_i$ on the adjacent edge,
  \item setting $f_{vw}=-(W_1^{-1})_{ij}$ if $v$ is labeled by $x_i$ and $w$ by $x_j$,
 \end{itemize}
which is exactly $\omega_{W_1}(H_1)$.
\end{proof}

\begin{lemma} \label{lemma:tech2}
 Let $H$ and $H'$ be substantial elements in $\tA(*_{\{1,\dots,n,h,\breve h\}})$. Let $W_1$ be a non-singular matrix of size $n$ with coefficients in $\Qtt$. Set $W=\begin{pmatrix} W_1&\bar\zeta&0\\^t\zeta&\lambda&0\\0&0&0 \end{pmatrix}$, $\hW=\begin{pmatrix} W_1&\bar\zeta&0\\^t\zeta&\lambda&-1\\0&-1&0 \end{pmatrix}$, and $J=\frac12W_{|n\to ne^{\breve h}}-\frac12W$. Assume $H'=\expd J\sqcup H_{|n\to ne^{\breve h}}$. Then $\omega_{\hW}(H)=\omega_{\hW}(H')$.
\end{lemma}
\begin{proof}
 By definition of $J$, we have
 $$J=\sum_{1\leq i\leq n}
 \raisebox{-0.1cm}{
 \begin{tikzpicture} [scale=0.25]
 \draw (0,0) node[left] {$\scriptstyle n$} -- (8,0) node[right] {$\scriptstyle i$};
 \foreach \x in {1,4} \draw (\x,0) -- (\x,1.5);
 \draw (2.5,0.8) node {$\scriptstyle \dots$}; \draw (2.5,2.5) node {$\scriptstyle {e^{\breve h}-1}$};
 \draw[->] (5,0) -- (6,0) node[above] {$\scriptstyle {W_{ni}}$};
 \end{tikzpicture}}
 +\frac12
 \raisebox{-0.8cm}{
 \begin{tikzpicture} [scale=0.25]
 \draw (0,0) node[left] {$\scriptstyle n$} -- (12,0) node[right] {$\scriptstyle n$};
 \foreach \x/\s in {1/1,4/1,8/-1,11/-1} \draw (\x,0) -- (\x,1.5*\s);
 \foreach \x/\s in {2.5/1,9.5/-1} {\draw (\x,0.8*\s) node {$\scriptstyle \dots$}; \draw (\x,2.5*\s) node {$\scriptstyle {e^{\breve h}-1}$};}
 \draw[->] (5,0) -- (6,0) node[above] {$\scriptstyle {W_{nn}}$};
 \end{tikzpicture}}
 +
 \raisebox{-0.1cm}{
 \begin{tikzpicture} [scale=0.25]
 \draw (0,0) node[left] {$\scriptstyle n$} -- (8,0) node[right] {$\scriptstyle h$};
 \foreach \x in {1,4} \draw (\x,0) -- (\x,1.5);
 \draw (2.5,0.8) node {$\scriptstyle \dots$}; 
 \draw (2.5,2.5) node {$\scriptstyle {e^{\breve h}-1}$};
 \draw[->] (5,0) -- (6,0) node[above] {$\scriptstyle {\bar\zeta_n}$};
 \end{tikzpicture}}
 .
 $$
Writing $H=\sum_\kappa \alpha_\kappa\raisebox{-0.6cm}{
\begin{tikzpicture} [scale=0.2]
 \draw (0,0) -- (5,0) -- (5,8) -- (0,8) -- (0,0);
 \draw (2.5,4) node {$D_\kappa$};
 \foreach \y in {1.5,6.5} {
 \draw (5,\y) -- (8,\y) node[right] {$\scriptstyle n$};}
 \draw (6.5,4.5) node {$\vdots$};
\end{tikzpicture}}$, where the $D_\kappa$ are substantial beaded Jacobi diagrams on $\{1,\dots,n,h,\breve h\}$, we get
\[
 H'=\sum\frac{\alpha_\kappa}{s!r_h!\prod_{i=1}^n r_i!} \raisebox{-1.5cm}{
 \begin{tikzpicture} [scale=0.25]
 \draw (0,0) -- (5,0) -- (5,8) -- (0,8) -- (0,0);
 \draw (2.5,4) node {$D_\kappa$};
 \foreach \y in {1,7} {
 \draw (5,\y) -- (10,\y) node[right] {$\scriptstyle n$};
 \foreach \x in {6,9} \draw (\x,\y) -- (\x,\y-1.5);
 \draw (7.5,\y-0.8) node {$\scriptstyle \dots$};
 \draw (7.5,\y-2.2) node {$\scriptstyle e^{\breve h}$};
 }
 \draw (7.5,3) node {$\vdots$};
 \foreach \z/\j in {7/1,-1/n} {
\begin{scope} [xshift=15cm,yshift=\z cm]
 \draw (-1.5,0.8) node {\Huge (};
 \draw (9.5,0.8) node {\Huge )};
 \draw (10.5,2) node {$\scriptstyle r_\j$};
 \draw (0,0) node[left] {$\scriptstyle n$} -- (8,0) node[right] {$\scriptstyle \j$};
 \foreach \x in {1,4} \draw (\x,0) -- (\x,1.5);
 \draw (2.5,0.8) node {$\scriptstyle \dots$}; \draw (2.5,2.5) node {$\scriptstyle {e^{\breve h}-1}$};
 \draw[->] (5,0) -- (6,0) node[above] {$\scriptstyle {W_{n\j}}$};
\end{scope}}
 \draw (19,4.5) node {$\vdots$};
\begin{scope} [xshift=30cm,yshift=7cm]
 \draw (-1.5,0.5) node {\Huge (};
 \draw (14,0.5) node {\Huge )};
 \draw (14.5,2) node {$\scriptstyle s$};
 \draw (0,0) node[left] {$\scriptstyle n$} -- (12,0) node[right] {$\scriptstyle n$};
 \foreach \x/\s in {1/1,4/1,8/-1,11/-1} \draw (\x,0) -- (\x,1.5*\s);
 \foreach \x/\s in {2.5/1,9.5/-1} {\draw (\x,0.8*\s) node {$\scriptstyle \dots$}; 
 \draw (\x,2.5*\s) node {$\scriptstyle {e^{\breve h}-1}$};}
 \draw[->] (5,0) -- (6,0) node[above] {$\scriptstyle {\frac12 W_{nn}}$};
\end{scope}
\begin{scope} [xshift=32cm,yshift=-1cm]
 \draw (-1.5,0.8) node {\Huge (};
 \draw (9.5,0.8) node {\Huge )};
 \draw (10.2,2) node {$\scriptstyle r_h$};
 \draw (0,0) node[left] {$\scriptstyle n$} -- (8,0) node[right] {$\scriptstyle h$};
 \foreach \x in {1,4} \draw (\x,0) -- (\x,1.5);
 \draw (2.5,0.8) node {$\scriptstyle \dots$}; \draw (2.5,2.5) node {$\scriptstyle {e^{\breve h}-1}$};
 \draw[->] (5,0) -- (6,0) node[above] {$\scriptstyle {\bar\zeta_n}$};
\end{scope}
\end{tikzpicture}
},
\]
where the sum is over all $\kappa$ and all $r_1,\dots,r_n,r_h,s\geq0$.
Applying $\omega_{\hW}$ and EV leads to the following.
\[
 \omega_{\hW}(H')=\sum\frac{\alpha_\kappa}{s!r_h!\prod_{i=1}^n r_i!} \raisebox{-1.5cm}{
 \begin{tikzpicture} [scale=0.25]
 \draw (0,0) -- (5,0) -- (5,8) -- (0,8) -- (0,0);
 \draw (2.5,4) node {$D_\kappa$};
 \foreach \y in {1,7} {
 \draw (5,\y) -- (10,\y) node[right] {$\scriptstyle x_n$};
 \foreach \x in {6,9} \draw (\x,\y) -- (\x,\y-1.5);
 \draw (7.5,\y-0.8) node {$\scriptstyle \dots$};
 \draw (7.5,\y-2.2) node {$\scriptstyle e^{(x_{\breve h})}$};
 }
 \draw (7.5,3) node {$\vdots$};
 \foreach \z/\j in {7/1,-1/n} {
\begin{scope} [xshift=15cm,yshift=\z cm]
 \draw (-2,0.8) node {\Huge (};
 \draw (9.5,0.8) node {\Huge )};
 \draw (10.5,2) node {$\scriptstyle r_\j$};
 \draw (0,0) node[left] {$\scriptstyle x_n$} -- (5,0) node[right] {$\scriptstyle W_{n\j}x_\j$};
 \foreach \x in {1,4} \draw (\x,0) -- (\x,1.5);
 \draw (2.5,0.8) node {$\scriptstyle \dots$}; \draw (2.5,2.5) node {$\scriptstyle {e^{(x_{\breve h})}-1}$};
\end{scope}}
 \draw (19,4.5) node {$\vdots$};
\begin{scope} [xshift=29cm,yshift=7cm]
 \draw (-2,0.5) node {\Huge (};
 \draw (14,0.5) node {\Huge )};
 \draw (14.5,2) node {$\scriptstyle s$};
 \draw (0,0) node[left] {$\scriptstyle x_n$} -- (12,0) node[right] {$\scriptstyle x_n$};
 \foreach \x/\s in {1/1,4/1,8/-1,11/-1} \draw (\x,0) -- (\x,1.5*\s);
 \foreach \x/\s in {2.5/1,9.5/-1} {\draw (\x,0.8*\s) node {$\scriptstyle \dots$}; 
 \draw (\x,2.5*\s) node {$\scriptstyle {e^{(x_{\breve h})}-1}$};}
 \draw[->] (5,0) -- (6,0) node[above] {$\scriptstyle {\frac12 W_{nn}}$};
\end{scope}
\begin{scope} [xshift=32cm,yshift=-1cm]
 \draw (-2,0.8) node {\Huge (};
 \draw (8.5,0.8) node {\Huge )};
 \draw (9.2,2) node {$\scriptstyle r_h$};
 \draw (0,0) node[left] {$\scriptstyle x_n$} -- (5,0) node[right] {$\scriptstyle \bar\zeta_n x_h$};
 \foreach \x in {1,4} \draw (\x,0) -- (\x,1.5);
 \draw (2.5,0.8) node {$\scriptstyle \dots$}; 
 \draw (2.5,2.5) node {$\scriptstyle {e^{(x_{\breve h})}-1}$};
\end{scope}
\end{tikzpicture}
}
\]
We claim that:
\[
 \omega_{\hW}(H')=\sum_{\substack{r,s\geq0\\\kappa}}\frac{\alpha_\kappa}{2^ss!r!} \raisebox{-1.5cm}{
 \begin{tikzpicture} [scale=0.25]
 \draw (0,0) -- (5,0) -- (5,8) -- (0,8) -- (0,0);
 \draw (2.5,4) node {$D_\kappa$};
 \foreach \y in {1,7} {
 \draw (5,\y) -- (10,\y) node[right] {$\scriptstyle x_n$};
 \foreach \x in {6,9} \draw (\x,\y) -- (\x,\y-1.5);
 \draw (7.5,\y-0.8) node {$\scriptstyle \dots$};
 \draw (7.5,\y-2.2) node {$\scriptstyle e^{(x_{\breve h})}$};
 }
 \draw (7.5,3) node {$\vdots$};
\begin{scope} [xshift=16cm,yshift=3cm]
 \draw (-2,0.8) node {\Huge (};
 \draw (6.5,0.8) node {\Huge )};
 \draw (7,2) node {$\scriptstyle r$};
 \draw (0,0) node[left] {$\scriptstyle x_n$} -- (5,0) node[right] {$\scriptstyle 0$};
 \foreach \x in {1,4} \draw (\x,0) -- (\x,1.5);
 \draw (2.5,0.8) node {$\scriptstyle \dots$}; \draw (2.5,2.5) node {$\scriptstyle {e^{(x_{\breve h})}-1}$};
\end{scope}
\begin{scope} [xshift=28cm,yshift=3cm]
 \draw (-2,0.5) node {\Huge (};
 \draw (14,0.5) node {\Huge )};
 \draw (14.5,2) node {$\scriptstyle s$};
 \draw (0,0) node[left] {$\scriptstyle x_n$} -- (12,0) node[right] {$\scriptstyle x_n$};
 \foreach \x/\s in {1/1,4/1,8/-1,11/-1} \draw (\x,0) -- (\x,1.5*\s);
 \foreach \x/\s in {2.5/1,9.5/-1} {\draw (\x,0.8*\s) node {$\scriptstyle \dots$}; 
 \draw (\x,2.5*\s) node {$\scriptstyle {e^{(x_{\breve h})}-1}$};}
 \draw[->] (5,0) -- (6,0) node[above] {$\scriptstyle {W_{nn}}$};
\end{scope}
\end{tikzpicture}
},
\]
where the linkings are as follows: if $v$ a $0$--labeled vertex, then $f_{vw}=-1$ if $w$ is labeled by $x_n$, $f_{vw}=-W_{nn}$ if $w$ is labeled by $0$, and $f_{vw}=0$ otherwise; the other linkings are as prescribed by~$-\hW^{-1}$.
To see this, apply a relation LV at each $0$--labeled vertex of the diagrams in the above sum, using the relation $\sum_{i=1}^nW_{ni}x_i+\bar\zeta_nx_h=0$ given by the $n$--th line of $\hW$.

\begin{center}
\begin{figure}
\begin{tikzpicture} [scale=0.25]
 \draw (0,0) -- (5,0) -- (5,10) -- (0,10) -- (0,0);
 \draw (2.5,5) node {$D_\kappa$};
 \foreach \y/\i in {1/p,9/1} {
 \draw (5,\y) -- (18,\y) node[right] {$\scriptstyle x_n$};
 \foreach \x in {6,8,10,12,15,17} \draw (\x,\y) -- (\x,\y-1.5);
 \foreach \x in {7,11,16} 
 \draw (\x,\y-0.8) node {$\scriptstyle \dots$};
 \draw (7,\y-2.2) node {$\scriptstyle e^{(x_{\breve h})}$};
 \foreach \x in {11,16} 
 \draw (\x,\y-2.2) node {$\scriptstyle {e^{(x_{\breve h})}-1}$};
 \draw (13.5,\y-1) node {$\scriptstyle \dots$};
 \draw[decorate,decoration={brace,amplitude=0.2cm}] (18,\y-3) -- (9,\y-3) node [midway,below,yshift=-0.2cm] {$\scriptstyle a_\i$};
 \draw (11,3.5) node {$\vdots$};
 }
\foreach \y/\i/\j in {-1/2u-1/2u,9.5/1/2} {
\begin{scope} [xshift=24cm,yshift=\y cm]
 \draw (0,0) node[left] {$\scriptstyle x_n$} -- (20,0) node[right] {$\scriptstyle x_n$};
 \foreach \x/\s in {1/1,3/1,6/1,8/1,12/-1,14/-1,17/-1,19/-1} \draw (\x,0) -- (\x,1.5*\s);
 \foreach \x/\s in {2/1,7/1,13/-1,18/-1} {\draw (\x,0.8*\s) node {$\scriptstyle \dots$}; \draw (\x,2.5*\s) node {$\scriptstyle {e^{(x_{\breve h})}-1}$};}
 \draw[->] (9,0) -- (10,0) node[above] {$\scriptstyle{W_{nn}}$};
 \foreach \x/\s in {4.5/1,15.5/-1} \draw (\x,\s) node {$\scriptstyle \dots$};
 \draw[decorate,decoration={brace,amplitude=0.2cm}] (20,-3.3) -- (11,-3.3) node [midway,below,yshift=-0.2cm] {$\scriptstyle b_{\j}+1$};
 \draw[decorate,decoration={brace,amplitude=0.2cm}] (0,3.3) -- (9,3.3) node [midway,above,yshift=0.2cm] {$\scriptstyle b_{\i}+1$};
\end{scope}
}
 \draw (34,4.5) node {$\vdots$};
\foreach \y/\i in {6/1,-7/q} {
\begin{scope} [xshift=50cm,yshift=\y cm]
 \draw (-1,-1) .. controls +(-1,2) and +(-1,-2) .. (-1,10);
 \draw (7.6,-1) .. controls +(1,2) and +(1,-2) .. (7.6,10);
 \draw (8.6,9.5) node {$\scriptstyle\zeta_\i$};
 \draw (-1,0) -- (-1,9) arc (180:0:0.5) -- (0,0) arc (0:-180:0.5);
 \foreach \z in {1,3,6,8}
 \draw (0,\z) -- (1.5,\z);
 \foreach \y in {1.7,2,2.3,6.7,7,7.3}
 \draw (0.8,\y) node {$\scriptstyle\cdot$};
 \foreach \y in {4.2,4.5,4.8}
 \draw (1,\y) node {$\scriptstyle\cdot$};
 \foreach \y in {2.3,7.3}
 \draw (4,\y) node {$\scriptstyle {e^{(x_{\breve h})}-1}$};
 \draw[decorate,decoration={brace,amplitude=0.2cm}] (6,9) -- (6,0) node [midway,right,xshift=0.1cm] {$\scriptstyle c_{\i}$};
\end{scope}
}
\draw (52,4) node {$\vdots$};
\end{tikzpicture}
\caption{The diagram $D$} \label{figgros}
\end{figure}
\end{center}

We now apply a relation LD for each pair of vertices labeled one by $0$ and the other by $0$ or~$x_n$. Each of these relations produces a diagram where the linking between the two involved vertices is zero and a diagram where these two vertices have been glued together to produce an edge which inherits a factor $-1$ or $-W_{nn}$; in both cases the sign can be put into the coefficient of the diagram using the relation~LE. In the sum we obtain, each diagram with a remaining $0$--labeled vertex is trivial thanks to the relation LV. This provides the following expression: 
$$\omega_{\hW}(H')=\displaystyle{\sum\frac{(-1)^{\sum_{i=1}^pa_i+\sum_{i=1}^{2u}b_i+\sum_{i=1}^qc_i\zeta_i}\alpha_\kappa}{2^u\displaystyle\prod_{i=1}^q\zeta_i!c_i^{\zeta_i}}\sum_{s=0}^u\frac{(-1)^{u-s}}{s!(u-s)!}}D,$$
where the sum runs over all $\kappa$, all $u\geq s\geq0$, all $p,q,a_1,\dots,a_p,b_1,\dots,b_{2u}\geq0$, all $c_1,\dots,c_q,$ $\zeta_1,\dots,\zeta_q>0$ and $D$ is the diagram in Figure~\ref{figgros}. The variable $r$ of the previous expression is equal to $r=2(u-s)+\sum_{i=1}^pa_i+\sum_{i=1}^{2u}b_i+\sum_{i=1}^qc_i\zeta_i$. 
To obtain the coefficient, we start with the coefficient $\frac{\alpha_\kappa}{2^ss!r!}$ from the previous expression. The factor $(-1)^{r-u+s}$ comes from the $r-u+s$ relations LD followed by~LE. Setting $t=u-s$, the choice of the $2t$ components
\raisebox{-1ex}{
\begin{tikzpicture} [scale=0.2]
 \draw (0,0) node[left] {$\scriptstyle x_n$} -- (5,0) node[right] {$\scriptstyle 0$};
 \foreach \x in {1,4} \draw (\x,0) -- (\x,1.5);
 \draw (2.5,0.8) node {$\scriptstyle \dots$}; \draw (2.5,2.5) node {$\scriptstyle {e^{(x_{\breve h})}-1}$};
\end{tikzpicture}} to be glued pairwise along the $0$--labeled vertices gives a factor~$\binom{r}{2t}$; there are $\frac{(2t)!}{2^tt!}$ possible ways to pair them. 
The remaining $r-2t$ components
\raisebox{-1ex}{
\begin{tikzpicture} [scale=0.2]
 \draw (0,0) node[left] {$\scriptstyle x_n$} -- (5,0) node[right] {$\scriptstyle 0$};
 \foreach \x in {1,4} \draw (\x,0) -- (\x,1.5);
 \draw (2.5,0.8) node {$\scriptstyle \dots$}; \draw (2.5,2.5) node {$\scriptstyle {e^{(x_{\breve h})}-1}$};
\end{tikzpicture}}
have been distributed into subsets of size $a_1$ to $a_p$, $b_1$ to $b_{2u}$, $c_1\zeta_1$ to $c_q\zeta_q$, which gives a factor $\frac{(r-2t)!}{\prod_{i=1}^pa_i!\prod_{i=1}^{2u}b_i!\prod_{i=1}^q(c_i\zeta_i)!}$. For each $i=1,\dots,p$, the $a_i$ corresponding components have been attached in all possible orders; this gives a factor $\prod_{i=1}^pa_i!$. Similarly, we get a factor $\prod_{i=1}^{2u}b_i!$. Finally, there are $\prod_{i=1}^q\frac{(c_i\zeta_i)!}{\zeta_i!c_i^{\zeta_i}}$ possible ways to glue together $c_i\zeta_i$ components
\raisebox{-1ex}{
\begin{tikzpicture} [scale=0.2]
 \draw (0,0) node[left] {$\scriptstyle x_n$} -- (5,0) node[right] {$\scriptstyle 0$};
 \foreach \x in {1,4} \draw (\x,0) -- (\x,1.5);
 \draw (2.5,0.8) node {$\scriptstyle \dots$}; \draw (2.5,2.5) node {$\scriptstyle {e^{(x_{\breve h})}-1}$};
\end{tikzpicture}}
in order to get $\zeta_i$ loops as prescribed. Multiplying all these factors gives the indicated coefficient. Since this coefficient vanishes for $u>0$, we get:
$$\omega_{\hW}(H')=\displaystyle{\sum\frac{(-1)^r\alpha_\kappa}{\displaystyle\prod_{i=1}^q\zeta_i!c_i^{\zeta_i}}}
\raisebox{-2.2cm}{
\begin{tikzpicture} [scale=0.25]
 \draw (0,0) -- (5,0) -- (5,10) -- (0,10) -- (0,0);
 \draw (2.5,5) node {$D_\kappa$};
 \foreach \y/\i in {1/p,9/1} {
 \draw (5,\y) -- (18,\y) node[right] {$\scriptstyle x_n$};
 \foreach \x in {6,8,10,12,15,17} \draw (\x,\y) -- (\x,\y-1.5);
 \foreach \x in {7,11,16} 
 \draw (\x,\y-0.8) node {$\scriptstyle \dots$};
 \draw (7,\y-2.2) node {$\scriptstyle e^{(x_{\breve h})}$};
 \foreach \x in {11,16} 
 \draw (\x,\y-2.2) node {$\scriptstyle {e^{(x_{\breve h})}-1}$};
 \draw (13.5,\y-1) node {$\scriptstyle \dots$};
 \draw[decorate,decoration={brace,amplitude=0.2cm}] (18,\y-3) -- (9,\y-3) node [midway,below,yshift=-0.2cm] {$\scriptstyle a_\i$};
 \draw (11,3.5) node {$\vdots$};
 }
\foreach \x/\i in {23.5/1,38/q} {
\begin{scope} [xshift=\x cm,yshift=0 cm]
 \draw (-1,-1) .. controls +(-1,2) and +(-1,-2) .. (-1,10);
 \draw (7.6,-1) .. controls +(1,2) and +(1,-2) .. (7.6,10);
 \draw (8.6,9.5) node {$\scriptstyle\zeta_\i$};
 \draw (-1,0) -- (-1,9) arc (180:0:0.5) -- (0,0) arc (0:-180:0.5);
 \foreach \z in {1,3,6,8}
 \draw (0,\z) -- (1.5,\z);
 \foreach \y in {1.7,2,2.3,6.7,7,7.3}
 \draw (0.8,\y) node {$\scriptstyle\cdot$};
 \foreach \y in {4.2,4.5,4.8}
 \draw (1,\y) node {$\scriptstyle\cdot$};
 \foreach \y in {2.3,7.3}
 \draw (4,\y) node {$\scriptstyle {e^{(x_{\breve h})}-1}$};
 \draw[decorate,decoration={brace,amplitude=0.2cm}] (6,9) -- (6,0) node [midway,right,xshift=0.1cm] {$\scriptstyle c_{\i}$};
\end{scope}
}
\draw (34.3,4) node {$\dots$};
\end{tikzpicture}}
$$
where the sum runs over all $\kappa$, all $p,q,a_1,\dots,a_p\geq0$, all $c_1,\dots,c_q,$ $\zeta_1,\dots,\zeta_q>0$, and $r=\sum_{i=1}^pa_i+\sum_{i=1}^qc_i\zeta_i$.
We can rewrite this expression as follows:
\begin{center}
\begin{tikzpicture} [scale=0.25]
 \draw (-8.5,4.5) node {$\displaystyle{\omega_{\hW}(H')=\sum\alpha_\kappa\Lambda(\underline k,\underline \ell)}$};
 \draw (0,0) -- (5,0) -- (5,10) -- (0,10) -- (0,0);
 \draw (2.5,5) node {$D_\kappa$};
 \foreach \y/\i in {1/p,9/1} {
 \draw (5,\y) -- (10,\y) node[right] {$\scriptstyle x_n$};
 \foreach \x in {6,9} \draw (\x,\y) -- (\x,\y-1.5) node[below] {$\scriptstyle x_{\breve h}$};
 \draw (7.5,\y-0.8) node {$\scriptstyle \dots$};
 \draw[decorate,decoration={brace,amplitude=0.1cm}] (9.7,\y-3) -- (5.3,\y-3) node [midway,below,yshift=-0.1cm] {$\scriptstyle {k_\i}$};
 }
 \draw (7.5,3) node {$\vdots$};
\foreach \x/\y/\i/\j in {15/6/1/1,26/6/1/{\zeta_1},15/-4/q/1,26/-4/q/{\zeta_q}} {
\begin{scope} [xshift=\x cm,yshift=\y cm]
 \draw (-1,0) -- (-1,5) arc (180:0:0.5) -- (0,0) arc (0:-180:0.5);
 \foreach \z in {1,4}
 \draw (0,\z) -- (1.5,\z) node [right] {$\scriptstyle {x_{\breve h}}$};
 \foreach \y in {2.2,2.5,2.8}
 \draw (0.8,\y) node {$\scriptstyle\cdot$};
 \draw[decorate,decoration={brace,amplitude=0.1cm}] (4,5) -- (4,0) node [midway,right,xshift=0.1cm] {$\scriptstyle \ell_{\i\j}$};
\end{scope}
}
\foreach \y in {-1.5,8.5}
\draw (23.2,\y) node {$\dots$};
\draw (23,4) node {$\vdots$};
\end{tikzpicture}
\end{center}
where $\underline k=(k_1,\dots,k_p)$, $\underline \ell=(\ell_{11},\dots,\ell_{1\zeta_1},\dots,\ell_{q1},\dots,\ell_{q\zeta_q})$, and the sum is over all $ k_1,\dots,k_p\geq0$ and all $\ell_{11},\dots,\ell_{1\zeta_1},\dots,\ell_{q1},\dots,\ell_{q\zeta_q}>0$.
We now compute the coefficient $\Lambda(\underline k,\underline \ell)$. 
We have
\[ \Lambda(\underline k,\underline \ell)=\sum\frac{(-1)^r}{
\prod_{i=1}^q\zeta_i!c_i^{\zeta_i}
\prod_{i=1}^p\prod_{j=0}^{a_i}k_i^{(j)}!
\prod_{i=1}^q\prod_{\iota=1}^{\zeta_i}\prod_{j=1}^{c_i}\ell_{i\iota}^{(j)}!}
\]
where the sum is over the set of integers
\[
  \left\{
    \begin{array}{c}
  k_1^{(0)},\ldots,k_p^{(0)}\geq0\\[.2cm]k_i^{(j)},\ell_{i\iota}^{(j)}>0
    \end{array}
    \left|
  \begin{array}{c}
    k_i^{(0)}+\cdots+k_i^{(a_i)}=k_i \quad \forall i=1,\dots,p\\ 
    \ell_{i\iota}^{(1)}+\cdots+\ell_{i\iota}^{(c_i)}=\ell_{i\iota}\quad \forall i=1,\dots,q,\ \forall\iota=1,\dots,\zeta_i
  \end{array} 
  \right.
  \right\} 
  \]
and $r=\sum_{i=1}^pa_i+\sum_{i=1}^qc_i\zeta_i$. 
We have $$\Lambda(\underline k,\underline \ell)=
\prod_{i=1}^pA_i
\prod_{i=1}^q\left(\frac1{\zeta_i!}\prod_{\iota=1}^{\zeta_i}C_{i\iota}\right)$$ where 
$$A_i=\sum_{a_i=0}^{k_i}\sum_{\substack{k_i^{(0)}+\cdots+k_i^{(a_i)}=k_i\\k_i^{(0)}\geq0,\ k_i^{(1)},\ldots,k_i^{(a_i)}>0}}\frac{(-1)^{a_i}}{\prod_{j=0}^{a_i}k_i^{(j)}!}\qquad\text{and}\qquad C_{i\iota}=\sum_{c_i=1}^{\ell_{i\iota}}\sum_{\substack{\ell_{i\iota}^{(1)}+\cdots+\ell_{i\iota}^{(c_i)}=\ell_{i\iota}\\\ell_{i\iota}^{(1)},\ldots,\ell_{i\iota}^{(c_i)}>0}}\frac{(-1)^{c_i}}{c_i\prod_{j=1}^{c_i}\ell_{i\iota}^{(j)}!}.
$$
By Lemmas~\ref{lemma:calcul} and~\ref{lemma:calculbis}, $A_i=0$ if $k_i>0$, and $C_{i\iota}=0$ if $\ell_{i\iota}>1$. Now, if some $\ell_{i\iota}$ equals $1$, the relation AS shows that the corresponding diagram, including a term
\raisebox{-1.5ex}{
\begin{tikzpicture} [scale=0.25]
 \draw (-1,0) -- (-1,2) arc (180:0:0.5) -- (0,0) arc (0:-180:0.5);
 \draw (0,1) -- (1.5,1) node [right] {$\scriptstyle {x_{\breve h}}$};
\end{tikzpicture}},
is trivial. Finally, in the last expression for $\omega_{\hW}(H')$, the only non trivial terms are those with $q=0$, $k_i=0$ for all~$i$, and the associated coefficient $\Lambda(\underline0,\emptyset)$ is equal to one. Hence $\omega_{\hW}(H')=\omega_{\hW}(H)$.
\end{proof}

\begin{lemma} \label{lemmaK2omega1}
 Let $G\in\tA(*_{\{1,\dots,n\}})$ be a group-like element. 
 Then $\omega(G_{|i\to i+j})=\omega(G)$.
\end{lemma}
\begin{proof}
 Write $G=\expd\left(\frac12 W\right)\sqcup H$. Then $G_{|i\to i+j}=\expd\left(\frac12 W_{|i\to i+j}\right)\sqcup H_{|i\to i+j}$. Set $W'=W_{|i\to i+j}$. Note that $W'=(I+E_{ji})W(I+E_{ij})$, where $I$ is the identity matrix and $E_{k\ell}$ is the matrix whose single non trivial coefficient is a $1$ at the $k$--th line and $\ell$--th column, and thus $(W')^{-1}=(I-E_{ij})W^{-1}(I-E_{ji})$. Hence, if $x_k$ ({\em resp} $x_k'$) are the generators of the Blanchfield module associated to $W$ ({\em resp} $W'$), the map defined by $x_i\mapsto x_i'+x_j'$ and $x_k\mapsto x_k'$ for $k\neq i$ is an isomorphism of Blanchfield modules (in particular, it respects the Blanchfield form). It follows that, first, $\omega(G)$ and $\omega(G_{|i\to i+j})$ indeed live in the same diagram space, and, second, $\omega(G_{|i\to i+j})=\omega_{W'}(H_{|i\to i+j})=\omega_W(H)=\omega(G)$.
\end{proof}

\begin{lemma} \label{lemmaK2omega2}
 Let $G=\expd(\frac12 W)\sqcup H$ and $G'=\expd(\frac12W')\sqcup H'$ be group-like elements in $\tA(*_{\{1,\dots,n\}})$. Assume that
 \[G'=\left\langle\expd\big(\Lambda_{j}^{\breve j'\breve i'}\big)\sqcup G_{\left|\substack{i\to i+i'\\j\to j'\phantom{+i}}\right.}\right\rangle_{\substack{i'-\breve i'\\j'-\breve j'}},\]
 where $\Lambda_{j}^{\breve j'\breve i'}$ is the Campbell--Baker--Hausdorff sum (see (\ref{eq:BCH}) in the proof of Lemma~\ref{lemmaGK}).
 Then $W'=W_{|i\to i+j}$ and $\omega(G')=\omega(G_{|i\to i+j})$.
\end{lemma}
\begin{proof}
 We set $\expd^\bullet\big(\Lambda_{j}^{\breve j'\breve i'}\big)=\expd \left(\Lambda_{j}^{\breve j'\breve i'}-\struts{j}{\breve j'}{}-\struts{j}{\breve i'}{}\right)$. Then the above formula for $G'$ becomes
\begin{eqnarray*}
  G'&=&\left\langle\expd^\bullet\big(\Lambda_{j}^{\breve j'\breve i'}\big)\sqcup\expd\left(\struts{j}{\breve j'}{}+\struts{j}{\breve i'}{}\right)\sqcup G_{\left|\substack{i\to i+i'\\j\to j'\phantom{+i}}\right.}\right\rangle_{\substack{i'-\breve i'\\j'-\breve j'}}\\
    &=&
        \left(\left\langle\expd^\bullet\big(\Lambda_{j}^{\breve j'\breve i'}\big)\sqcup G_{\left|{\substack{i\to i+i'+i''\\j\to j'+j''\phantom{+i}}}\right.}\right\rangle_{\substack{i'-\breve i'\\j'-\breve j'}}\right)_{\left|{\substack{i''\to j\\j''\to j}}\right.}
    \ =\ 
        \left\langle\expd^\bullet\big(\Lambda_{j}^{\breve j'\breve i'}\big)\sqcup G_{\left|{\substack{i\to i+i'+j\\j\to j+j'\phantom{+i'}}}\right.}\right\rangle_{\substack{i'-\breve i'\\j'-\breve j'}},
\end{eqnarray*}
the indices $i''$ and $j''$ corresponding to the vertices glued to $\expd\left(\struts{j}{\breve j'}{}+\struts{j}{\breve i'}{}\right)$.

Hence we have $W'=W_{|i\to i+j}$ and 
  \[
  H'=\left\langle\expd^\bullet\big(\Lambda_{j}^{\breve j'\breve i'}\big)\sqcup\expd\left(\frac12\bigg(W_{\left|{\substack{i\to i+i'+j\\j\to j+j'\phantom{+i}}}\right.}-W_{|i\to i+j}\bigg)\right)\sqcup H_{\left|{\substack{i\to i+i'+j\\j\to j+j'\phantom{+i}}}\right.}\right\rangle_{\substack{i'-\breve i'\\j'-\breve j'}}.
\]
Recall from the previous lemma that $W'=(I+E_{ji})W(I+E_{ij})$; in particular $W'_{ii}=W_{ii}$ and $W'_{ji}=W_{ji}+W_{ii}$. We set $W''=W(I+E_{ij})=(I-E_{ji})W'$ so that, for every $k$ and $\ell$,
\[
  W_{k\ell}+\delta_{\ell j}W_{ki}=W''_{k\ell}=W'_{k\ell}-\delta_{kj}W'_{i\ell}.
  \]
Note that the matrix $W''$ is not hermitian in general. 

We prove now that $\omega_{W'}(H')=\omega_{W'}(H_{|i\to i+j})$. 
 A direct computation gives 
 \[W_{\left|{\substack{i\to i+i'+j\\j\to j+j'\phantom{+i}}}\right.}-W_{|i\to i+j}=
 \sum_{k=1}^n\struts{i'}{k}{W'_{ik}}+
 \sum_{\ell=1}^n\struts{j'}{\ell}{W''_{j\ell}}+\struts{i'}{j'}{W_{ij}}+\frac12\struts{i'}{i'}{W_{ii}}+\frac12\struts{j'}{j'}{W_{jj}}.
 \]
 Writing $\expd^\bullet\big(\Lambda_{j}^{\breve j'\breve i'}\big)=\sum_{\lambda}D_\lambda$ and $H=\sum_hD_h$, we have 
  \[
 H'=\sum\alpha_\lambda\beta_h\raisebox{-2.3cm}{
 \begin{tikzpicture} [scale=0.25]
 \draw (0,0) -- (22,0) -- (22,9) -- (0,9) -- (0,0) (11,4.5) node {$D_\lambda$};
 \draw (25,0) -- (37,0) -- (37,9) -- (25,9) -- (25,0) (31,4.5) node {$D_h$};
 \foreach \y in {1,6} {
 \draw[->] (0,\y) node[right] {$\scriptscriptstyle{\breve j'}$} .. controls +(-2.5,0) and +(0,-0.7) .. (-3,1+\y);
 \draw (-3,1+\y) node[left] {$\scriptstyle W_{ji}$} .. controls +(0,0.7) and +(-2.5,0) .. (0,2+\y) node[right] {$\scriptscriptstyle{\breve i'}$};}
 \foreach \y/\i in {1/j,4/j,5/i,8/i}
 \draw (22,\y) node[left] {$\scriptscriptstyle{\breve \i'}$} -- (25,\y) node[right] {$\scriptscriptstyle{\i}$};
 \foreach \x/\y in {-1/5,23.5/3,23.5/7}
 \draw (\x,\y) node {$\vdots$};
 \foreach \y/\s/\i/\p/\q in {0/-1/j/above/below,9/1/i/below/above} {
 \foreach \x in {2,7} {
 \draw (\x,\y) node[\p] {$\scriptscriptstyle{\breve \i'}$} .. controls +(0,2.5*\s) and +(-0.7,0) .. (1+\x,3*\s+\y) node[\q] {$\scriptstyle W_{\i\i}$} .. controls +(0.7,0) and +(0,2.5*\s) .. (2+\x,\y) node[\p] {$\scriptscriptstyle{\breve \i'}$};}
 \foreach \x in {12,19} {
 \draw[->] (\x,\y) node[\p] {$\scriptscriptstyle{\breve \i'}$} -- (\x,1.5*\s+\y);
 \draw (\x,1.5*\s+\y) -- (\x,3*\s+\y);}
 }
 \foreach \x/\y/\z/\i/\k/\e in {12/-4/-1.5/j/\ell_1/'',19/-4/-1.5/j/\ell_s/'',12/13/10.5/i/k_1/',19/13/10.5/i/k_r/'} {
 \draw (\x,\z) node[right] {$\scriptstyle W^{\e}_{\i{\k}}$};
 \draw (\x,\y) node {$\scriptstyle \k$};}
 \foreach \y in {-1.2,10.2} \foreach \x in {5.5,17}
 \draw (\x,\y) node {$\dots$};
 \foreach \x/\i in {27/i,30/i,32/j,35/j} 
 \draw (\x,9) node[below] {$\scriptscriptstyle i$} -- (\x,12) node[above] {$\scriptstyle \i$};
 \foreach \x in {29.5,32.5}
 \draw (\x,0) node[above] {$\scriptscriptstyle j$} -- (\x,-3) node[below] {$\scriptstyle j$};
 \foreach \x/\y in {28.5/10,33.5/10,31/-2}
 \draw (\x,\y) node {$\dots$};
 \end{tikzpicture}
 },
 \]
 where the sum runs over all $\lambda$, all $h$, all ways of partitioning the sets of $\breve i'$-- and $\breve j'$--labeled vertices of $D_\lambda$ each into four subsets, all ways of partitioning the set of $i$--labeled ({\em resp} $j$--labeled) vertices of $D_h$ into three ({\em resp} two) subsets, and all choices of $k_1,\dots,k_r,\ell_1,\dots,\ell_s\in\{1,\dots,n\}$. The $\breve i'$, $\breve j'$, $i$ and $j$ in the boxes are here to recall which gluing each arc comes from. Note indeed that the coefficient of each product of struts coming from the exponential cancels with the number of choices of gluing a given strut; and that the edges with diagonal labels $W_{ii}$ or $W_{jj}$ come with a factor $1=2.\frac12$ since, for each strut, there are two ways to glue it.

 Let $x_1,\dots,x_n$ be the generators of the Blanchfield module $\Al'$ associated to $W'$, and apply $\omega_{W'}$ to $H'$. Thanks to the relation EV, we get: 
 \[
 \omega_{W'}(H')=\sum\alpha_\lambda\beta_h\raisebox{-2.3cm}{
 \begin{tikzpicture} [scale=0.25]
 \draw (0,0) -- (22,0) -- (22,9) -- (0,9) -- (0,0) (11,4.5) node {$D_\lambda$};
 \draw (25,0) -- (37,0) -- (37,9) -- (25,9) -- (25,0) (31,4.5) node {$D_h$};
 \foreach \y in {1,6} {
 \draw[->] (0,\y) node[right] {$\scriptscriptstyle{\breve j'}$} .. controls +(-2.5,0) and +(0,-0.7) .. (-3,1+\y);
 \draw (-3,1+\y) node[left] {$\scriptstyle W_{ji}$} .. controls +(0,0.7) and +(-2.5,0) .. (0,2+\y) node[right] {$\scriptscriptstyle{\breve i'}$};}
 \foreach \y/\i in {1/j,4/j,5/i,8/i}
 \draw (22,\y) node[left] {$\scriptscriptstyle{\breve \i'}$} -- (25,\y) node[right] {$\scriptscriptstyle{\i}$};
 \foreach \x/\y in {-1/5,23.5/3,23.5/7}
 \draw (\x,\y) node {$\vdots$};
 \foreach \y/\s/\i/\p/\q in {0/-1/j/above/below,9/1/i/below/above} {
 \foreach \x in {2,7} {
 \draw (\x,\y) node[\p] {$\scriptscriptstyle{\breve \i'}$} .. controls +(0,2.5*\s) and +(-0.7,0) .. (1+\x,3*\s+\y) node[\q] {$\scriptstyle W_{\i\i}$} .. controls +(0.7,0) and +(0,2.5*\s) .. (2+\x,\y) node[\p] {$\scriptscriptstyle{\breve \i'}$};}
 \foreach \x in {13,19} {
 \draw (\x,\y) node[\p] {$\scriptscriptstyle{\breve \i'}$} -- (\x,3*\s+\y);}
 }
 \foreach \x/\y/\z/\i/\k/\e in {13/-4/-1.5/j/\ell_1/'',19/-4/-1.5/j/\ell_s/'',13/13/10.5/i/k_1/',19/13/10.5/i/k_r/'} {
 \draw (\x,\y) node {$\scriptstyle W^{\e}_{\i{\k}}x_{\k}$};}
 \foreach \y in {-1.2,10.2} \foreach \x in {5.5,16}
 \draw (\x,\y) node {$\dots$};
 \foreach \x/\i in {27/i,30/i,32/j,35/j} 
 \draw (\x,9) node[below] {$\scriptscriptstyle{i}$} -- (\x,12) node[above] {$\scriptstyle{x_{\i}}$};
 \foreach \x in {29.5,32.5}
 \draw (\x,0) node[above] {$\scriptscriptstyle j$} -- (\x,-3) node[below] {$\scriptstyle x_j$};
 \foreach \x/\y in {28.5/10,33.5/10,31/-2}
 \draw (\x,\y) node {$\dots$};
 \end{tikzpicture}
 }.
 \]
 The relations defining $\Al'$ are given by $W'\begin{pmatrix} x_1 \\ \vdots \\ x_n \end{pmatrix}=0$. It implies that $W''\begin{pmatrix} x_1 \\ \vdots \\ x_n \end{pmatrix}=0$, so that $\sum_{k=1}^n W'_{ik}x_k=0$ and $\sum_{\ell=1}^n W''_{j\ell}x_\ell=0$.
 Successive applications of the relation LV then give
 \[
 \omega_{W'}(H')=\sum\alpha_\lambda\beta_h\raisebox{-2.3cm}{
 \begin{tikzpicture} [scale=0.25]
 \draw (0,0) -- (22,0) -- (22,9) -- (0,9) -- (0,0) (11,4.5) node {$D_\lambda$};
 \draw (25,0) -- (37,0) -- (37,9) -- (25,9) -- (25,0) (31,4.5) node {$D_h$};
 \foreach \y in {1,6} {
 \draw[->] (0,\y) node[right] {$\scriptscriptstyle{\breve j'}$} .. controls +(-2.5,0) and +(0,-0.7) .. (-3,1+\y);
 \draw (-3,1+\y) node[left] {$\scriptstyle W_{ji}$} .. controls +(0,0.7) and +(-2.5,0) .. (0,2+\y) node[right] {$\scriptscriptstyle{\breve i'}$};}
 \foreach \y/\i in {1/j,4/j,5/i,8/i}
 \draw (22,\y) node[left] {$\scriptscriptstyle{\breve \i'}$} -- (25,\y) node[right] {$\scriptscriptstyle{\i}$};
 \foreach \x/\y in {-1/5,23.5/3,23.5/7}
 \draw (\x,\y) node {$\vdots$};
 \foreach \y/\s/\i/\p/\q in {0/-1/j/above/below,9/1/i/below/above} {
 \foreach \x in {2,7} {
 \draw (\x,\y) node[\p] {$\scriptscriptstyle{\breve \i'}$} .. controls +(0,2.5*\s) and +(-0.7,0) .. (1+\x,3*\s+\y) node[\q] {$\scriptstyle W_{\i\i}$} .. controls +(0.7,0) and +(0,2.5*\s) .. (2+\x,\y) node[\p] {$\scriptscriptstyle{\breve \i'}$};}
 \foreach \x in {13,19} {
 \draw (\x,\y) node[\p] {$\scriptscriptstyle{\breve \i'}$} -- (\x,3*\s+\y);}
 }
 \foreach \x/\y/\i in {13/-4/j,19/-4/j,13/13/i,19/13/i}
 \draw (\x,\y) node {$\scriptstyle 0_{\i}$};
 \foreach \y in {-1.2,10.2} \foreach \x in {5.5,16}
 \draw (\x,\y) node {$\dots$};
 \foreach \x/\i in {27/i,30/i,32/j,35/j} 
 \draw (\x,9) node[below] {$\scriptscriptstyle{i}$} -- (\x,12) node[above] {$\scriptstyle{x_{\i}}$};
 \foreach \x in {29.5,32.5}
 \draw (\x,0) node[above] {$\scriptscriptstyle j$} -- (\x,-3) node[below] {$\scriptstyle x_j$};
 \foreach \x/\y in {28.5/10,33.5/10,31/-2}
 \draw (\x,\y) node {$\dots$};
 \end{tikzpicture}
 }.
 \]
In this picture, $0_{j}$ and $0_i$ are $0$--labels which differ by their linkings with other vertices.
The $0_j$--labeled vertices are linked by
\begin{itemize}
 \item $-\sum_{\ell}W''_{j\ell}W'^{-1}_{\ell k}=-\sum_{\ell}W'_{j\ell}W'^{-1}_{\ell k}+\sum_{\ell}W'_{i\ell}W'^{-1}_{\ell k}=-\delta_{kj}+\delta_{ki}$ to $x_k$--labeled vertices;
 \item $-\sum_{k,\ell}W''_{j\ell}W'^{-1}_{\ell k}W'_{ki}=-\sum_\ell \delta_{i\ell}W''_{j\ell}=-W''_{ji}=-W_{ji}$ to $0_i$--labeled vertices;
  \item $-\sum_{k,\ell}W''_{j\ell}W'^{-1}_{\ell k}\overline{W''_{jk}}=\sum_k(\delta_{ki}-\delta_{kj})\overline{W''_{jk}}=\overline{W''_{ji}}-\overline{W''_{jj}}=-W_{jj}$ to $0_j$--labeled vertices.
\end{itemize}
Similarly, $0_i$--labeled vertices are linked by
\begin{itemize}
\item $-\sum_{\ell}W'_{i\ell}W'^{-1}_{\ell k}=-\delta_{ki}$ to $x_k$--labeled vertices;
\item $-\sum_{k,\ell}W'_{i\ell}W'^{-1}_{\ell k}W'_{ki}=-\sum_k\delta_{ki}W'_{ki}=-W'_{ii}=-W_{ii}$ to $0_i$--labeled
  vertices.
\end{itemize}
Now, we apply relations LD on all pairs of vertices labeled by $0_j$ and $x_i$, $0_j$ and $x_j$, $0_j$ and $0_j$, $0_j$ and $0_i$, $0_i$ and $x_i$, or $0_i$ and $0_i$. After applying all these relations, if a $0_i$-- or $0_j$--labeled vertex remains in a diagram, then it is linked by $0$ to all other vertices, so that the diagram vanishes thanks to the relation LV. Finally, $\omega_{W'}(H')$ is a linear combination of diagrams as follows.
\[
 \begin{tikzpicture} [scale=0.25]
 \draw (0,0) -- (22,0) -- (22,9) -- (0,9) -- (0,0) (11,4.5) node {$D_\lambda$};
 \draw (25,0) -- (37,0) -- (37,9) -- (25,9) -- (25,0) (31,4.5) node {$D_h$};
 \foreach \y in {1,6} {
 \draw[->] (0,\y) node[right] {$\scriptscriptstyle{\breve j'}$} .. controls +(-2.5,0) and +(0,-0.7) .. (-3,1+\y);
 \draw (-3,1+\y) node[left] {$\scriptstyle W_{ji}$} .. controls +(0,0.7) and +(-2.5,0) .. (0,2+\y) node[right] {$\scriptscriptstyle{\breve i'}$};}
 \foreach \y/\i in {1/j,4/j,5/i,8/i}
 \draw (22,\y) node[left] {$\scriptscriptstyle{\breve \i'}$} -- (25,\y) node[right] {$\scriptscriptstyle{\i}$};
 \foreach \x/\y in {-1/5,23.5/3,23.5/7}
 \draw (\x,\y) node {$\vdots$};
 \foreach \y/\s/\i/\p/\q in {0/-1/j/above/below,9/1/i/below/above} {
 \foreach \x in {2,7} {
 \draw (\x,\y) node[\p] {$\scriptscriptstyle{\breve \i'}$} .. controls +(0,2.5*\s) and +(-0.7,0) .. (1+\x,3*\s+\y) node[\q] {$\scriptstyle W_{\i\i}$} .. controls +(0.7,0) and +(0,2.5*\s) .. (2+\x,\y) node[\p] {$\scriptscriptstyle{\breve \i'}$};}
 }
 \foreach \x in {12,17} {
 \draw (\x,9) node[below] {$\scriptscriptstyle{j}$} .. controls +(0,2.5) and +(-0.7,0) .. (1+\x,12) .. controls +(0.7,0) and +(0,2.5) .. (2+\x,9) node[below] {$\scriptscriptstyle{\breve j'}$};}
 \foreach \x/\y in {5.5/-1.2,5.5/10.2,15.5/10.2}
 \draw (\x,\y) node {$\dots$};
 \draw (14,0) node[above] {$\scriptscriptstyle{\breve j'}$} .. controls +(0,-4) and +(-2,0) .. (23,-4) .. controls +(2,0) and +(0,-4) .. (32,0) node[above] {$\scriptscriptstyle{i}$};
 \draw (17,0) node[above] {$\scriptscriptstyle{\breve j'}$} .. controls +(0,-2) and +(-1,0) .. (23,-2) .. controls +(1,0) and +(0,-2) .. (29,0) node[above] {$\scriptscriptstyle{i}$};
 \draw (23,-2.7) node {$\vdots$};
 \end{tikzpicture}
\]
We consider a given diagram in this sum and compute its coefficient. Denote by $n_{ii}$,  $n_{jj}$, $n_{ji}$ respectively the number of $W_{ii}$--, $W_{jj}$--, $W_{ji}$--labeled edges. These come either from the gluing of struts in first step, or from the relations LD in the second step. Denote by $k_{ii}$, $k_{jj}$, $k_{ji}$ respectively the number of those coming from the second step, hence with a factor $-1$.
Similarly, denote by $n_i$ and $n_j$ respectively the number of edges relating a $\breve i'$ ({\em resp} $\breve j'$) from $D_\lambda$ to an $i$ ({\em resp} $j$) from~$D_h$, and denote by $k_i$ ({\em resp} $k_j$) the number of those coming from the relations LD, hence with a factor $-1$. 
Also, denote by $m_{ij}$ the number of edges relating a $\breve j'$ from $D_\lambda$ to an $i$ from~$D_h$. These all come from relations LD, applied to a pair of vertices labeled by $0_j$ and $x_i$ or by $0_j$ and~$x_j$; denote by $\ell_{ij}$ the number of the second kind, which come with a factor $-1$. Finally, denote by $m_j$ the number of edges relating a $\breve j'$ to a $j$, both from $D_\lambda$; these all come from the relations LD and bring a factor $-1$.
With these notations, the coefficient of the above diagram is
$$\alpha_\lambda\beta_h\sum(-1)^{k_j}\binom{n_j}{k_j}(-1)^{k_{i}}\binom{n_{i}}{k_{i}}(-1)^{k_{jj}}\binom{n_{jj}}{k_{jj}}(-1)^{k_{ji}}\binom{n_{ji}}{k_{ji}}(-1)^{k_{ii}}\binom{n_{ii}}{k_{ii}}(-1)^{\ell_{ij}}\binom{m_{ij}}{\ell_{ij}}(-1)^{m_j},$$
where the sum is over all $k_i$, $k_j$, $k_{ii}$, $k_{jj}$, $k_{ji}$, $\ell_{ij}$ such that $0\leq k_i\leq n_i$, $0\leq k_j\leq n_j$, $0\leq k_{ii}\leq n_{ii}$, $0\leq k_{ji}\leq n_{ji}$, $0\leq k_{jj}\leq n_{jj}$, $0\leq \ell_{ij}\leq m_{ij}$.
This coefficient vanishes if $n_i+n_j+n_{jj}+n_{ji}+n_{ii}+m_{ij}>0$. 
Hence, if the coefficient is non trivial, then $D_\lambda$ contains no $\breve i'$--labeled vertex, which implies that $D_\lambda=\emptyset$. 
It follows that 
\[\omega(G')=\omega_{W'}(H')
=\omega_{W'}\left(H_{\left|i\to i+j\right.}\right)
=\omega\left(G_{\left|i\to i+j\right.}\right).\qedhere\]
\end{proof}

\section{Construction of the invariant $\tZ$} \label{secinvariant}

\subsection{Invariant of a surgery presentation}

We use the functor $Z$ defined in \cite{CHM}, which is a renormalization of the Le--Murakami functor \cite{LM1,LM2}. 

The domain of this functor is the category $\tang$ with objects the non-associative words in the letters $(+,-)$ and morphisms the $q$--tangles. Composition is given by vertical juxtaposition. We also define a tensor product by horizontal juxtaposition. 

\begin{figure}[htb] 
\begin{center}
\begin{tikzpicture} [scale=0.4]
 \draw (6,8) .. controls +(0,-2) and +(1,0) .. (5.2,4.7);
 \draww{(3.5,2) .. controls +(0,1) and +(0,-2) .. (6,4) .. controls +(0,0.7) and +(1,0) .. (5.2,5.3);}
 \draw (5.2,5.3) .. controls +(-1,0) and +(1,0) .. (4.7,3);
 \draww{(5.2,4.7) .. controls +(-0.8,0) and +(0,-1) .. (4.5,6) .. controls +(0,1) and +(0,-1) .. (4,8);}
 \draww{(4,0) .. controls +(0,1) and +(0,-1) .. (5.5,2) .. controls +(0,0.5) and +(0.8,0) .. (4.7,3.5);}
 \draww{((4.7,3.5) .. controls +(-0.8,0) and +(0,0.5) .. (4.5,2) .. controls +(0,-1) and +(0,1) .. (2,0);}
 \draww{(4.7,3) .. controls +(-1,0) and +(0,-3) .. (2,8);}
 \draww{(6,0) .. controls +(0,2) and +(0,-1) .. (3.5,2);}
 \draw (0,0) -- (0,8) -- (8,8) -- (8,0) -- (0,0);
 \draw (2,8) node[above] {$(-$} (4,8) node[above] {$+)$} (6,8) node[above] {$-$};
 \draw (2,0) node[below] {$+$} (4,0) node[below] {$(-$} (6,0) node[below] {$-)$};
 \draw[->] (5.8,0.8) -- (5.7,0.9); \draw[->] (4.5,0.83) -- (4.6,0.9); \draw[->] (6.01,6.9) -- (6.01,7);
\end{tikzpicture}
\end{center} \caption{Diagram of a $q$--tangle} \label{figtang}
\end{figure}

Define a category $\tA$ whose objects are associative words in the letters $(+,-)$ and whose sets of morphisms are $\tA(v,u)=\oplus_X\tA(X)$, where $X$ runs over all homeomorphism classes of compact oriented 1--manifolds with boundary identified with the set of letters of $u$ and $v$, with the following sign convention: for~$u$, a ``$+$'' when the orientation of $X$ goes towards the boundary point and a ``$-$'' when it goes backward, 
and the converse for~$v$. Composition is given by vertical juxtaposition, where the label of a created edge is the product of the labels on the initial two edges. The tensor product given by disjoint union defines a strict monoidal structure on $\tA$.

We recall in Figure \ref{figfunctorZ} the definition of $Z$ on the elementary $q$--tangles, where $\nu\in\A(\xc)\cong\A(\xd)$ is the value of the 
Kontsevich integral on the zero framed unknot, $\Phi\in\A(\xdop\xdop\xdop)$ is a Drinfeld associator with rational coefficients and 
$\Delta^{+++}_{u_1,u_2,u_3}:\A(\xdop\xdop\xdop)\to\A(\xdop_{\{u_1,u_2,u_3\}})$ is obtained by applying $(|u_i|-1)$ times the coproduct $\Delta$ on the $i$-th factor.

\begin{figure}[htb] 
$$Z\left(\hspace{-1ex}\raisebox{-0.8cm}{
\begin{tikzpicture} [scale=0.6]
 \draw[->] (0,1) node[above] {$\scriptstyle{(+}$} -- (1,0) node[below] {$\scriptstyle{+)}$};
 \draww{[->] (1,1) node[above] {$\scriptstyle{+)}$} -- (0,0) node[below] {$\scriptstyle{(+}$};}
\end{tikzpicture}}
\right)=\exp\left(\frac{1}{2}\raisebox{-0.5cm}{
\begin{tikzpicture} [scale=0.6]
 \draw[->,very thick] (0,2) -- (0,1) -- (1,0);
 \draw[->,very thick] (1,2) -- (1,1) -- (0,0);
 \draw (0,1.5) -- (1,1.5);
\end{tikzpicture}}\,\right)\in\A\left(\hspace{-0.7ex}\raisebox{-0.2cm}{
\begin{tikzpicture} [scale=0.6]
 \draw[->] (0,1) -- (1,0);
 \draw[->] (1,1) -- (0,0);
\end{tikzpicture}}\right)
\qquad
Z\left(\hspace{-1ex}\raisebox{-0.8cm}{
\begin{tikzpicture} [scale=0.6]
 \draw[->] (1,1) node[above] {$\scriptstyle{+)}$} -- (0,0) node[below] {$\scriptstyle{(+}$};
 \draww{[->] (0,1) node[above] {$\scriptstyle{(+}$} -- (1,0) node[below] {$\scriptstyle{+)}$};}
\end{tikzpicture}}
\right)=\exp\left(-\frac{1}{2}\raisebox{-0.5cm}{
\begin{tikzpicture} [scale=0.6]
 \draw[->,very thick] (0,2) -- (0,1) -- (1,0);
 \draw[->,very thick] (1,2) -- (1,1) -- (0,0);
 \draw (0,1.5) -- (1,1.5);
\end{tikzpicture}}\,\right)\in\A\left(\hspace{-0.7ex}\raisebox{-0.2cm}{
\begin{tikzpicture} [scale=0.6]
 \draw[->] (0,1) -- (1,0);
 \draw[->] (1,1) -- (0,0);
\end{tikzpicture}}\right)$$
$$Z\left(\hspace{-1ex}\raisebox{-0.4cm}{
\begin{tikzpicture} [scale=0.6]
 \draw[->] (1,0) node[below] {$\scriptstyle{-)}$} .. controls +(0,1) and +(0,1) .. (0,0) node[below] {$\scriptstyle{(+}$};
\end{tikzpicture}}
\right)=\raisebox{-0.1cm}{
\begin{tikzpicture} [scale=0.6]
 \draw[->,very thick] (2,0) .. controls +(0,0.5) and +(0.5,0) .. (1.5,1) -- (0.5,1) .. controls +(-0.5,0) and +(0,0.5) .. (0,0);
 \draw[fill=white] (1,1) circle (0.3);
 \draw (1,1) node {$\nu$};
\end{tikzpicture}}\,\in\A\left(\hspace{-0.7ex}\raisebox{-0.05cm}{
\begin{tikzpicture} [scale=0.6]
 \draw[->] (1,0) .. controls +(0,1) and +(0,1) .. (0,0);
\end{tikzpicture}}\right)
\qquad
Z\left(\hspace{-1ex}\raisebox{-0.5cm}{
\begin{tikzpicture} [scale=0.6]
 \draw[->] (0,0) node[above] {$\scriptstyle{(+}$} .. controls +(0,-1) and +(0,-1) .. (1,0) node[above] {$\scriptstyle{-)}$};
\end{tikzpicture}}
\right)=\raisebox{-0.3cm}{
\begin{tikzpicture} [scale=0.7]
 \draw[->,very thick] (0,0) .. controls +(0,-1) and +(0,-1) .. (1,0);
\end{tikzpicture}}\,\in\A\left(\hspace{-0.7ex}\raisebox{-0.25cm}{
\begin{tikzpicture} [scale=0.6]
 \draw[->] (0,0) .. controls +(0,-1) and +(0,-1) .. (1,0);
\end{tikzpicture}}\right)$$
$$Z\left(\hspace{-1ex}\raisebox{-0.7cm}{
\begin{tikzpicture} [scale=0.6]
 \draw[->] (0,1) node[above] {$\scriptstyle{(u}$} -- (0,0) node[below] {$\scriptstyle{((u}\ $};
 \draw[->] (1.5,1) node[above] {$\scriptstyle{(v}\,$} -- (0.5,0) node[below] {$\,\scriptstyle{v)}$};
 \draw[->] (2,1) node[above] {$\ \scriptstyle{w))}$} -- (2,0) node[below] {$\scriptstyle{w)}$};
\end{tikzpicture}}\right)=\Delta^{+++}_{u,v,w}(\Phi)\,\in\A(\xdop_{\{u,v,w\}})$$
\caption{The functor $Z:\tang\to\A$} \label{figfunctorZ}
\end{figure}

Let $L$ be a surgery presentation of a $\Q$\SK--pair. Fixing an admissible diagram of $L$ (as defined in Section~\ref{subsecBlanchfield}), one can view the surgery presentation as a $q$--tangle with empty top and bottom words and write it as the product of two $q$--tangles $\gamma_t$ and $\gamma_b$, see Figure~\ref{figdecouptangle}.
The word at the top of $\gamma_b$ and at the bottom of $\gamma_t$ is a product $(v)(w)$, where $w$ corresponds to the part of the tangle which meets the disk $\DD$. Set:
$$\Zp(L)=Z(\gamma_b)\circ(I_{v}\otimes G_{w})\circ Z(\gamma_t)\ \in\tA(\xc_{\pi_0(L)}),$$
where $I_v$ is the identity on the word $v$ and $G_{w}$ is obtained from $I_{w}$ by adding a label $t$ ({\em resp} $t^{-1}$) on skeleton components associated with 
a $-$ sign ({\em resp} a $+$ sign), see Figure~\ref{figIvGv}. 
The invariance of $\Zp$, with respect to isotopy of $L$ and to the decomposition of the admissible diagram as the product of $\gamma_t$ and $\gamma_b$, is due to the invariance of the functor $Z$ and the following observation of Kricker \cite[Lemma~3.2.4]{Kri}.

\begin{figure}[htb] 
\begin{center}
\begin{tikzpicture} 
\begin{scope}
 \draw (-1,-1) -- (1,-1) -- (1,1) -- (-1,1) -- (-1,-1);
 \draw[dashed] (-1,0) -- (0,0); \draw (0,0) node {$\scriptstyle{\bullet}$} -- (1,0); \draw (1,0) node[right] {$(v)(w)$};
 \draw (0,-0.5) node {$\gamma_b$};
 \draw (0,+0.5) node {$\gamma_t$};
\end{scope}
 \draw (4.5,0) node {Ex:};
\begin{scope}[xshift=6.3cm]
 \draw (-1,-1) -- (1,-1) -- (1,1) -- (-1,1) -- (-1,-1);
 \draw (0,0) node {$\scriptstyle{\bullet}$} -- (1,0); 
 \draw (0.1,-0.3) arc (0:-90:0.3) arc (270:75:0.6);
 \draw (-0.1,0.3) arc (180:90:0.3) arc (90:-105:0.6);
 \draw (0.1,0.4) arc (0:-90:0.2) arc (90:270:0.2);
 \draw (-0.1,-0.4) arc (180:90:0.2) arc (-90:90:0.2);
 \draw (0.1,0.4) arc (0:20:0.2);
 \draw (-0.1,-0.4) arc (180:200:0.2);
 \draw[->] (0.22,0.6) -- (0.2,0.6);
 \draw[->] (-0.22,0.6) -- (-0.2,0.6);
\end{scope}
 \draw (8,0) node {$\rightsquigarrow$};
\begin{scope}[xshift=11cm,yshift=0.7cm]
 \draw (-1.5,0.5) node {$\gamma_t=$};
 \draw (-1,0) -- (1,0) -- (1,1) -- (-1,1) -- (-1,0);
 \draw (-0.8,0) arc (180:75:0.6);
 \draw (-0.1,0.3) arc (180:90:0.3) arc (90:0:0.6);
 \draw (0.1,0.4) arc (0:-90:0.2) arc (90:180:0.2);
 \draw (0.3,0) arc (0:90:0.2);
 \draw (0.1,0.4) arc (0:20:0.2);
 \draw[->] (0.22,0.6) -- (0.2,0.6);
 \draw[->] (-0.22,0.6) -- (-0.2,0.6);
\end{scope}
\begin{scope}[xshift=11cm]
 \draw (-0.5,0) node {$(v)(w)=(-+)(+-)$};
\end{scope}
\begin{scope}[xshift=11cm,yshift=-0.7cm]
 \draw (-1.5,-0.5) node {$\gamma_b=$};
 \draw (-1,-1) -- (1,-1) -- (1,0) -- (-1,0) -- (-1,-1);
 \draw (0.1,-0.3) arc (0:-90:0.3) arc (270:180:0.6);
 \draw (0.8,0) arc (0:-105:0.6);
 \draw (-0.3,0) arc (180:270:0.2);
 \draw (-0.1,-0.4) arc (180:90:0.2) arc (-90:0:0.2);
 \draw (-0.1,-0.4) arc (180:200:0.2);
 \draw[->] (0.2,-0.6) -- (0.22,-0.6);
 \draw[->] (-0.2,-0.6) -- (-0.22,-0.6);
\end{scope}
\end{tikzpicture}
\end{center} \caption{Decomposing a surgery presentation as the product of two $q$--tangles} \label{figdecouptangle}
\end{figure}

\begin{figure}[htb] 
\begin{center}
\begin{tikzpicture} [xscale=0.6,yscale=0.5]
 \draw (0,1) node {$I_{\scriptscriptstyle--+-}=$};
 \foreach \x in {2,3,5} {\draw[->,very thick] (\x,0) -- (\x,2);}
 \draw[<-,very thick] (4,0) -- (4,2);
\begin{scope} [xshift=10cm]
 \draw (0,1) node {$G_{\scriptscriptstyle--+-}=$};
 \foreach \x in {2,3,5} {\draw[->,very thick] (\x,0) -- (\x,2);}
 \draw[<-,very thick] (4,0) -- (4,2);
 \foreach \x in {2,3,4,5} {\draw (\x,1) node {$\scriptscriptstyle{\bullet}$};}
 \foreach \x in {2,3,5} {\draw (\x,1) node[right] {$\scriptstyle{t}$};}
 \draw (3.9,1.1) node[right] {$\scriptstyle{t^{-1}}$};
\end{scope}
\end{tikzpicture}
\end{center} \caption{The diagrams $I_v$ and $G_v$ for $v=--+-$} \label{figIvGv}
\end{figure}

\begin{lemma}
  For a beaded Jacobi diagram $D\in\tA(w,v)$, we have $G_v\circ D=D\circ G_w$.
\end{lemma}
\begin{proof}
 Apply the relations Hol and \holw\ at all vertices of the diagram.
\end{proof}

\begin{lemma} \label{lemmagrouplike}
 For any surgery presentation $L$ of a $\Q$--sphere, $\Zp(L)$ is group-like.
\end{lemma}
\begin{proof}
 The fact that $Z(\gamma)$ is group-like for a $q$--tangle $\gamma$ follows from \cite[Theorem~5.1]{LM3}. This concludes since the $G_v$ are obviously group-like and the coproduct commutes with the composition. 
\end{proof}

    \subsection{Invariant of $\Q\SK$--pairs}

Set:
\[
  \Zc(L)=\chi^{-1}\left(\nu^{\otimes\pi_0(L)}\sharp_{\pi_0(L)}\Zp(L)\right)\ \in\tA\big(\xxlw_{\pi_0(L)}\big)
\]
where the connected sum means that a copy of $\nu$ is summed to each component of $L$. Note that $\Zc(L)$ is group-like since $\Zp(L)$ and $\nu$ are group-like and $\chi$ preserves the coproduct.

We now define a specific lift $\overline{\Zc(L)}\in\tA(*_{\pi_0(L)})$ of $\Zc(L)$.
For that, fix an admissible diagram of $L$ and base points~$\star_i$ on each component $L_i$ of $L$. Construct $\overline{\Zc(L)}$ following the construction from the beginning of Section \ref{secinvariant} for this diagram, with the skeleton components corresponding to the components of~$L$ defined as intervals by cutting each component $L_i$ at the base point~$\star_i$. 

\begin{lemma} \label{lemmawdmatrix}
 The lift $\overline{\Zc(L)}$ is group-like and we have:
 $$\overline{\Zc(L)}=\expd\left(\frac{1}{2}W_L\right)\sqcup H,$$ 
 where $W_L$ is the winding matrix associated with our choice of diagram and base points and $H$ is substantial. 
\end{lemma}
\begin{proof}
 The same argument as in Lemma \ref{lemmagrouplike} shows that $\overline{\Zc(L)}$ is group-like. We have to compute the part of $\overline{\Zc(L)}$ made of struts. The group-like property implies that $\overline{\Zc(L)}$ is the exponential of a series of connected diagrams. Since $\nu$ and the associator $\Phi$ have no terms with exactly two vertices, the only contributions to the strut part 
 come from the crossings between components of~$L$. For $i\neq j$, the definition of $Z$ and the \holw\ relation show that the contribution of a crossing $c$ between $L_i$ and $L_j$ is 
 $\chi^{-1}\left(\frac{1}{2}\textrm{sg}(c)\hspace{-1ex}
 \raisebox{-0.75cm}{
 \begin{tikzpicture} [xscale=0.6,yscale=0.5]
 \draw[->,very thick] (0,0) node[below] {$\scriptstyle{L_i}$} -- (0,2); \draw[->,very thick] (2,0) node[below] {$\scriptstyle{L_j}$} -- (2,2); 
 \draw[->] (0,1) -- (1,1); \draw (1,1) node[above] {$\scriptstyle{t^{\varepsilon_{ij}(c)}}$} -- (2,1);
 \end{tikzpicture}}\right)$. 
 Hence the contribution of all crossings between $L_i$ and $L_j$ is \raisebox{-1cm}{
 \begin{tikzpicture} [xscale=0.7,yscale=0.5]
 \begin{scope}
  \draw[->] (0,1) node[right] {$\scriptstyle{(W_L)_{ij}}$} -- (0,2) node[above] {$\scriptstyle{L_j}$} (0,0) node[below] {$\scriptstyle{L_i}$} -- (0,1);
 \end{scope}
  \draw (2.3,0.7) node {$=$};
 \begin{scope} [xshift=3.3cm]
  \draw[->] (0,0) -- (0,1); \draw (0,1) -- (0,2);
  \draw (0,0) node[below] {$\scriptstyle{L_j}$}; \draw (0,2) node[above] {$\scriptstyle{L_i}$}; \draw (0,1) node[right] {$\scriptstyle{(W_L)_{ji}}$};
 \end{scope}
 \end{tikzpicture}}.
 For $i=j$, the contribution of a self-crossing of $L_i$ is:
 $$\chi^{-1}\left(\frac{1}{2}\textrm{sg}(c)
 \raisebox{-0.9cm}{
 \begin{tikzpicture} [xscale=0.7,yscale=0.7]
  \draw[->,very thick] (0,0) -- (0,2); \draw (0,0) node[below] {$\scriptstyle{L_i}$};
  \draw[->] (0,0.5) arc (-90:0:0.5); \draw (0.5,1) arc (0:90:0.5); \draw (0.5,1) node[right] {$\scriptstyle{t^{\varepsilon_{ii}(c)}}$};
 \end{tikzpicture}}\right)
 =\textrm{sg}(c)\left(\raisebox{-0.9cm}{
 \begin{tikzpicture} [xscale=0.7,yscale=0.5]
  \draw[->] (0,0) -- (0,1); \draw (0,1) -- (0,2);
  \draw (0,0) node[below] {$\scriptstyle{L_i}$}; \draw (0,2) node[above] {$\scriptstyle{L_i}$}; \draw (0,1) node[right] {$\scriptstyle{t^{\varepsilon_{ii}(c)}}$};
 \end{tikzpicture}}
 +\frac{1}{2}\raisebox{-0.9cm}{
 \begin{tikzpicture} [xscale=0.5,yscale=0.5]
  \draw (0,0) -- (0,1); \draw[->] (0,1) arc (-90:90:0.5); \draw (0,2) arc (90:270:0.5);
  \draw (0,0) node[below] {$\scriptstyle{L_i}$}; \draw (0,2) node[above] {$\scriptstyle{t^{\varepsilon_{ii}(c)}}$};
 \end{tikzpicture}}\right).
 $$ 
 Summed over all self-crossings of $L_i$, we get as strut part:
 $$\sum_c\frac{1}{2}\textrm{sg}(c)\raisebox{-0.9cm}{
 \begin{tikzpicture} [xscale=0.7,yscale=0.5]
  \draw[->] (0,0) -- (0,1); \draw (0,1) -- (0,2);
  \draw (0,0) node[below] {$\scriptstyle{L_i}$}; \draw (0,2) node[above] {$\scriptstyle{L_i}$}; \draw (0,1) node[right] {$\scriptstyle{t^{\varepsilon_{ii}(c)}}$};
 \end{tikzpicture}}
 =\frac{1}{2}\raisebox{-0.9cm}{
 \begin{tikzpicture} [xscale=0.7,yscale=0.5]
  \draw[->] (0,0) -- (0,1); \draw (0,1) -- (0,2);
  \draw (0,0) node[below] {$\scriptstyle{L_i}$}; \draw (0,2) node[above] {$\scriptstyle{L_i}$}; \draw (0,1) node[right] {$\scriptstyle{(W_L)_{ii}}$};
 \end{tikzpicture}}.$$
 Hence $\overline{\Zc(L)}=\expd\left(\frac{1}{2}W_L\right)\sqcup H$ where $H\in\tA(*_{\pi_0(L)})$ is substantial.
\end{proof}
The matrix $W_L(1)$ is the linking matrix of the link $L$, hence it is a presentation matrix for the first homology group of a $\Q$--sphere. 
Thus $\det\big(W_L(1)\big)\neq0$ and $\overline{\Zc(L)}$ is a non-degenerate group-like element.

%Proposition~\ref{propkey} implies the following lemma.
\begin{lemma} \label{lemma:omega}
 The diagram $\omega\left(\overline{\Zc(L)}\right)\in\A(\Al,\bl)$ does not depend on the basepoints $\star_i$ chosen to construct the lift $\overline{\Zc(L)}\in\tA\left(*_{\pi_0(L)}\right)$ of $\Zc(L)\in\tA\left(\xxlw_{\pi_0(L)}\right)$.
\end{lemma}
\begin{proof}
 Fix $i\in\{1,\dots,n\}$ and let $\overline{\Zc(L)}'$ be another lift of $\Zc(L)$, constructed with the same basepoints $\star_j$ for $j\neq i$ but a different basepoint $\star_i'$. Let $D$ be the diagram obtained by following the construction of $\overline{\Zc(L)}$, but with the $i-th$ component of the skeleton cut in two parts, along the two basepoints $\star_i$ and $\star_i'$. By the same argument as for $\overline{\Zc(L)}$, we see that $D$ is group-like. Thus it defines a group-like relation between $\overline{\Zc(L)}$ and $\overline{\Zc(L)}'$. We conclude with Proposition~\ref{propkey}.
\end{proof}

%Considering the construction of the lift $\overline{\Zc(L)}$ given above, Lemma~\ref{lemma:omega} says that $\omega\left(\overline{\Zc(L)}\right)$ does not depend on the choice of basepoints on the components of $L$.
This lemma allows to set:
$$\omega\big(\Zc(L)\big)=\omega\left(\overline{\Zc(L)}\right)\quad\in\A(\Al,\bl).$$

\begin{proposition} \label{prop:invariance}
 Let $U_{\pm}$ be a trivial knot with framing $\pm1$, unlinked with $\un\subset S^3$. For a $\Q\SK$--pair $(M,K)$ with an admissible surgery presentation $L$ and Blanchfield module $(\Al,\bl)$, we denote $\sigma_\pm(L)$ the number of positive/negative eigenvalues of the linking matrix of $L$. Then
 $$\tZ(M,K)=\Zc(U_+)^{-\sigma_+(L)}\sqcup \Zc(U_-)^{-\sigma_-(L)}\sqcup \omega\big(\Zc(L)\big)\quad\in\A(\Al,\bl)$$
 defines an invariant $\tZ$ of $\Q\SK$--pairs.
\end{proposition}
\begin{proof}
 We have to check that $\tZ(M,K)$ does not depend on the surgery presentation. We first consider the orientation of the components of $L$. Let $L'$ be the surgery link that differs from $L$ only by the orientation of the $i$--th component. By \cite[Theorem~4]{LM2}, $\Zp(L')$ (and thus $\Zc(L')$) is obtained from $\Zp(L)$ ({\em resp} $\Zc(L)$) by flipping the sign of all diagrams with an odd number of univalent vertices glued to the $i$--th skeleton component ({\em resp} labeled by $L_i$). Now, the winding matrix $W_{L'}$ is obtained from $W_L$ by multiplying the $i$--th row and column by $-1$, and the meridian $m(L_j')$ equals $m(L_j)$ if $j\neq i$ and $-m(L_i)$ if $j=i$. Hence, using the relation~LV, we see that the signs cancel and we get $\omega\big(\Zc(L')\big)=\omega\big(\Zc(L)\big)$. It concludes since the signature is unmodified.
 
 \begin{figure}[htb] 
 \begin{center}
 \begin{tikzpicture} [scale=0.3]
  \draw (0,0) -- (0,-6);
  \draw[dashed] (0,2) -- (0,0) (0,-6) -- (0,-8);
  \draw (-4,-4.5) arc (-90:90:1.5);
  \draw[dashed] (-4,-1.5) arc (90:270:1.5);
  \draw (-4,-0.5) node{$L_i$};
  \draw (0,-1) node[right] {$L_j$};
  \draw[<->,line width=1.5pt] (7.5,-3) -- (9.5,-3);
 \begin{scope} [xshift=19.5cm]
  \draw[rounded corners=2pt] (-4,-4.8) arc (-90:-16:1.8) -- (0,-3.5) -- (0,-6) (-4,-1.2) arc (90:16:1.8) -- (0,-2.5) -- (0,0);
  \draw[dashed] (0,2) -- (0,0) (0,-6) -- (0,-8);
  \draw (-4,-4.5) arc (-90:90:1.5);
  \draw[dashed] (-4,-1.5) arc (90:270:1.5) (-4,-1.2) arc (90:270:1.8);
 \end{scope}
 \end{tikzpicture} \caption{Kirby II move.} \label{figK2}
 \end{center}
 \end{figure}

 If $L$ and $L'$ are two surgery links, a theorem of Habiro and Widmer \cite[Theorem~3.1]{HW} ---building on results of Kirby \cite{Kir} and Fenn--Rourke \cite{FennRourke}--- applied in our setting asserts that $L$ and $L'$ represent homeomorphic $\Q$\SK--pairs if and only if they are related by a sequence of Kirby moves: the Kirby~I move is the addition or removal of a $\pm1$--framed unknot unlinked with the other components, and the Kirby~II move is the sliding of a component over another one, see Figure~\ref{figK2}.
 
\begin{figure}[htb] 
\begin{center}
\begin{tikzpicture} [scale=0.4]
 \foreach \y in {0,5} \draw[very thick,dashed] (1,\y) -- (2,\y);
 \draw[very thick,rounded corners=10pt,->] (2,0) -- (4,0) -- (4,3.1) (2,5) -- (4,5) -- (4,3) node[left] {$L_j$};
 \draw (3,0) node {\scalebox{1.5}{$\star$}} node[below right] {$\scriptstyle j$};
 \draw[very thick,rounded corners=10pt,->] (5,1.9) -- (5,5) -- (10,5) -- (10,0) -- (5,0) -- (5,2) node[right] {$L_i$};
 \draw (5,3.8) node {\scalebox{1.5}{$\star$}} node[below right] {$\scriptstyle i$};
 \foreach \x in {6.4,7,8.6} \draw (\x,5) -- (\x,8);
 \draw (7.8,6) node {$\dots$};
 \draw[line width=2pt,->] (12.5,3)--(14,3);
 \begin{scope} [xshift=16cm]
  \foreach \y in {-0.2,5.2} \draw[very thick,dashed] (1,\y) -- (2,\y);
 \draw[very thick,rounded corners=10pt,->] (3.9,5.2) -- (6,5.2) (9,5.2) -- (10.2,5.2) -- (10.2,-0.2) -- (2,-0.2) (2,5.2) -- (4,5.2);
 \draw (3,-0.2) node {\scalebox{1.5}{$\star$}} node[below right] {$\scriptstyle{j'}$};
 \draw[very thick,rounded corners=8pt,->] (5,1.9) -- (5,4.8) -- (6,4.8) (9,4.8) -- (9.8,4.8) -- (9.8,0.2) -- (5,0.2) -- (5,2);
 \draw (5,3.8) node {\scalebox{1.5}{$\star$}} node[below right] {$\scriptstyle i$};
 \foreach \x in {6.4,7,8.6} \draw (\x,5.6) -- (\x,8);
 \draw (7.8,6) node {$\dots$};
 \draw[gray,thick] (6,5.6) -- (9,5.6) -- (9,4.4) -- (6,4.4) -- (6,5.6);
 \draw (7.5,5) node {$\Delta$};
 \end{scope}
\end{tikzpicture}
\caption{Effect of a KII move on $\chi\big(\Zc(L)\big)$\\{\footnotesize The $\Delta$--box stands for the sum over all possibilities to attach each thin line to one of the two skeleton components.}} \label{figK2Kon}
\end{center}
\end{figure}

 The normalization term $\Zc(U_+)^{-\sigma_+(L)}\sqcup \Zc(U_-)^{-\sigma_-(L)}$ ensures independance with respect to the Kirby I move as usual. Independance with respect to the Kirby II move is based on a result of Le--Murakami--Murakami--Ohtsuki \cite[Proposition 1]{LMMO} which expresses the effect of a Kirby II move on the Kontsevich integral. Set $L'=(L\setminus L_j)\cup L_{j'}$, where $L_{j'}$ is obtained from $L_j$ by sliding it over $L_i$. Then $\chi\big(\Zc(L')\big)$ is deduced from $\chi\big(\Zc(L)\big)$ by the operation described on Figure~\ref{figK2Kon} applied to each diagram. Let $G=\overline{\Zc(L)}$ and $G'=\overline{\Zc(L')}$ be the lifts constructed choosing the basepoints as indicated on Figure~\ref{figK2Kon}. We shall prove that $\omega(G')=\omega(G)$. Denoting each element of the $\pi_0$ of a link by the index of the corresponding component, the expression of $\chi(G')$ in terms of $\chi(G)$ can be explicitely written as follows. For a Jacobi diagram $D$ on~$\xd_{\pi_0(L)}$, $\Delta_i^{ii'}(D)$ is defined by duplicating the $i$--th skeleton component (index by $i'$ the created component) and taking the sum of all ways of distributing the univalent vertices initially glued on $L_i$ between $L_i$ and $L_{i'}$. For a Jacobi diagram $D$ on $\xd_{\pi_0(L\sqcup L_{i'})}$, define $m_{j'}^{ji'}(D)$ by gluing the head of $\xd_j$ to the tail of $\xd_{i'}$ to
form $\xd_{j'}$. With these notations, we have by \cite[Proposition~1]{LMMO}:
\[
  G'=\chi^{-1}\circ m_{j'}^{ji'}\circ \Delta_i^{ii'}\circ \chi(G).
  \]
 The maps $\Delta_i^{ii'}$ and $\chi$ commute in the sense that $\Delta_i^{ii'}\circ\chi(G)=\chi\left(G_{|i\to i+i'}\right)$. By \cite[Proposition~5.4]{AA2}, this gives:
 \[
   G'=\chi^{-1}\circ m_{j'}^{ji'}\circ\chi\left(G_{|i\to i+i'}\right)=\left\langle\expd\big(\Lambda_{j'}^{\breve j\breve i'}\big)\sqcup G_{|i\to i+i'}\right\rangle_{\substack{j-\breve j\\i'-\breve i'}},
   \]
 where $\Lambda_{j'}^{\breve j\breve i'}$ is the Campbell--Baker--Hausdorff sum (see (\ref{eq:BCH}) in the proof of Lemma~\ref{lemmaGK}). Thanks to Lemma~\ref{lemmawdmatrix}, $G$ and $G'$ are group-like, hence Lemmas~\ref{lemmaK2omega1} and~\ref{lemmaK2omega2} imply $\omega(G')=\omega(G)$. Lemma~\ref{lemmaK2omega2} also gives $W'=W_{|i\to i+j}$, so that the signature is preserved.
\end{proof}

Note that the effect on the winding matrix can be recovered topologically. As shown on Figure~\ref{figK2mer}, we have in homology $m(L_{j})=m(L_{j'})$ and $m(L_i)=m(L_{i'})+m(L_{j'})$.

\begin{figure}[htb] 
\begin{center}
\begin{tikzpicture} [scale=0.4]
%Entrelacs
\draw (0,0) -- (0,-6);
\draw[dashed] (0,2) -- (0,0) (0,-6) -- (0,-8);
\draw (-4,-4.5) arc (-90:90:1.5);
\draw[dashed] (-4,-1.5) arc (90:270:1.5);
\draw[->] (-3.9,-1.5) -- (-4,-1.5) node[above] {$L_i$};
\draw[->] (0,0) -- (0,-1) node[right] {$L_j$};
%mj
\draw[white] (-3,-2) node{$\bullet$};
\draw (-3.45,-1.8) .. controls +(-0.4,-0.5) and +(-0.5,-0.5) .. (-2.95,-1.92) .. controls +(0.5,0.5) and +(0.5,0.5) .. (-3.2,-1.5);
\draw[->] (-2.7,-1.5) -- (-2.7,-1.4);
\draw (-1.8,-0.8) node {$m(L_i)$};
%mi
\draw[white] (0,-5.5) node{$\bullet$};
\draw (-0.2,-5) .. controls +(-0.5,0) and +(-0.5,-0) .. (0,-5.45) .. controls +(0.5,0) and +(0.5,0) .. (0.2,-5);
\draw[->] (0.48,-5.2) -- (0.48,-5.25) node[right] {$m(L_j)$};
%Fleche
\draw[->,line width=1.5pt] (5,-3) -- (6,-3);
\begin{scope} [xshift=15cm]
%Entrelacs
\draw (0,0) -- (0,-2.3) -- (-2.35,-2.3) (-2.35,-3.7) -- (0,-3.7) -- (0,-6);
\draw[dashed] (0,2) -- (0,0) (0,-6) -- (0,-8);
\draw (-4,-4.5) arc (-90:90:1.5);
\draw[dashed] (-4,-1.5) arc (90:270:1.5);
\draw (-4,-4.8) arc (-90:90:1.8);
\draw[dashed] (-4,-1.2) arc (90:270:1.8);
\draw[color=white,line width=3pt] (-2.26,-2.33) -- (-2.26,-3.68);
%\draw (-2.26,-2.3) -- (0,-2.5) (-2.26,-3.5) -- (0,-3.5);
\draw[->] (-3.9,-1.5) -- (-4,-1.5);
\draw[->] (0,0) -- (0,-1) node[right] {$L_{j'}$};
%Meridiens
%mi
\draw[white] (0,-5.5) node{$\bullet$};
\draw (-0.2,-5) .. controls +(-0.5,0) and +(-0.5,0) .. (0,-5.45) .. controls +(0.5,0) and +(0.5,0) .. (0.2,-5);
\draw[->] (0.48,-5.2) -- (0.48,-5.25) node[right] {$m(L_{j'})$};
%mj'
\draw[white] (-2.5,-3.3) node{$\bullet$};
\draw (-2.7,-2.8) .. controls +(-0.5,0) and +(-0.5,-0) .. (-2.5,-3.25) .. controls +(0.5,0) and +(0.5,0) .. (-2.35,-2.8);
\draw[->] (-2.05,-3) -- (-2.05,-2.9) node[right] {$m(L_{i'})$};
%mj
\draw[color=white,line width=5pt] (-3,-2) -- (-2,-1);
\draw (-3.45,-1.8) .. controls +(-0.4,-0.5) and +(-0.5,-0.5) .. (-2.8,-1.8) .. controls +(0.5,0.5) and +(0.5,0.4) .. (-3.05,-1.3);
\draw[->] (-2.55,-1.3) -- (-2.55,-1.2) node[above] {$m(L_i)$};
\end{scope}
\end{tikzpicture}
\caption{KII move and meridians} \label{figK2mer}
\end{center}
\end{figure}

\subsection{Recovering the Kricker invariant}

We now explicit the fact that our invariant $\tZ$ is a refinement of the Kricker invariant $Z^\Kri$. 

The construction of the invariant $Z^\Kri$ is the same as for $\tZ$ until the last step. Instead of applying the operation $\omega$, Garoufalidis and Kricker apply a formal Gaussian integral which merges the strut part $\expd\left(\frac{1}{2}W_L\right)$ and the substantial part $H$ by summing all possible ways to glue all vertices of $\expd\left(-\frac{1}{2}W_L^{-1}\right)$ with all vertices of $H$ that have the same label. This defines an invariant with values in the space $\A(\delta)$. Now, applying the operation $\omega$ first and the map $\psi:\A(\Al,\bl)\to\A(\delta)$ then has the same effect as applying the formal Gaussian integral.

\begin{proposition} \label{proptZKri}
 For any $\Q\SK$--pair $(M,K)$, we have $Z^\Kri(M,K)=\psi\circ\tZ(M,K)$.
\end{proposition}

\begin{remark} \label{remLMO}
 The formal Gaussian integration was initially introduced by Bar-Natan, Garoufalidis, Rozansky and Thurston to define the Aarhus integral, which recovers the LMO invariant. This version of the LMO invariant is constructed as the Kricker invariant, forgetting the knots in the $3$--manifolds and the beads on the diagrams.
\end{remark}

\begin{remark}
 Since $Z^\Kri$ can be deduced from $\tZ$, the invariance of $Z^\Kri$ and its behaviour with respect to null LP--surgeries stem from that of $\tZ$.
\end{remark}

\subsection{Behaviour of $\tZ$ under connected sum}

By construction, the invariant $\tZ$ behaves well under connected sum.

\begin{lemma} \label{lemma:connectedsum}
 Let $(M_1,K_1)$ and $(M_2,K_2)$ be $\Q\SK$--pairs. Let $(\Al_1,\bl_1)$ and $(\Al_2,\bl_2)$ denote their Blanchfield modules. The invariant $\tZ$ is given on their connected sum by:  
 $$\tZ\big((M_1,K_1)\sharp(M_2,K_2)\big)=\tZ(M_1,K_1)\sqcup \tZ(M_2,K_2)\in\A\big((\Al_1,\bl_1)\oplus(\Al_2,\bl_2)\big).$$
\end{lemma}
\begin{proof}
 If $L_1$ and $L_2$ are surgery links for $(M_1,K_1)$ and $(M_2,K_2)$ respectively, a surgery link for $(M_1,K_1)\sharp(M_2,K_2)$ is obtained by stacking $L_1$ and $L_2$.
 Then $L_1$ and $L_2$ can be separated by an isotopy, see Figure~\ref{figstack}.
\end{proof}

\begin{figure}[htb] 
 \begin{center}
 \begin{tikzpicture} [scale=0.5]
 \begin{scope} %1er diagramme
 %huit
  \draw (0,7.5) .. controls +(1,0) and +(0,1) .. (1,6) .. controls +(0,-2) and +(0,1) .. (-1,3);
  \draw (0,6) .. controls +(-1,0) and +(0,1) .. (-2,4) .. controls +(0,-4) and +(0,-3) .. (1,3);
  \draww{(0,7.5) .. controls +(-1,0) and +(0,1) .. (-1,6) .. controls +(0,-2) and +(0,1) .. (1,3);}
  \draww{(0,6) .. controls +(1,0) and +(0,1) .. (2,4) .. controls +(0,-4) and +(0,-3) .. (-1,3);}
 %Whitehead
  \draww{(0,3) arc (-90:90:2);}
  \draww{(0,2) arc (-90:90:3);}
  \draww{(0,8) .. controls  +(-4,0) and +(-1,0) .. (-2.5,4.5);}
  \draww{(0,3) .. controls  +(-3,0) and +(1,0) .. (-2.5,5.5);}
  \draww{(0,2) .. controls  +(-4,0) and +(-1,0) .. (-2.5,5.5);}
  \draww{(0,7) .. controls  +(-3,0) and +(1,0) .. (-2.5,4.5);}
 %tour&disk
  \draw (-4,0.5) rectangle (4,9.5);
  \draw (0,5) node {$\bullet$} -- (4,5);
  \draw[->] (2.24,7) -- (2.14,7.1) node[right] {$L_1$};
  \draw[->] (1.4,1.5) -- (1.45,1.6) node[right] {$L_2$};
 \end{scope}
  \draw (7,5) node {$\sim$};
 \begin{scope} [xshift=14cm] %2e diagramme
 \begin{scope} [scale=0.6,yshift=3.3cm] %huit
  \draw (0,7.5) .. controls +(1,0) and +(0,1) .. (1,6) .. controls +(0,-2) and +(0,1) .. (-1,3);
  \draw (0,6) .. controls +(-1,0) and +(0,1) .. (-2,4) .. controls +(0,-4) and +(0,-3) .. (1,3);
  \draww{(0,7.5) .. controls +(-1,0) and +(0,1) .. (-1,6) .. controls +(0,-2) and +(0,1) .. (1,3);}
  \draww{(0,6) .. controls +(1,0) and +(0,1) .. (2,4) .. controls +(0,-4) and +(0,-3) .. (-1,3);}
 \end{scope}
 \begin{scope} [xscale=1.1,yscale=1.3,xshift=0.2cm,yshift=-1.4cm] %Whitehead
  \draww{(0,3) arc (-90:90:2);}
  \draww{(0,2) arc (-90:90:3);}
  \draww{(0,8) .. controls  +(-4,0) and +(-1,0) .. (-2.5,4.5);}
  \draww{(0,3) .. controls  +(-3,0) and +(1,0) .. (-2.5,5.5);}
  \draww{(0,2) .. controls  +(-4,0) and +(-1,0) .. (-2.5,5.5);}
  \draww{(0,7) .. controls  +(-3,0) and +(1,0) .. (-2.5,4.5);}
 \end{scope}
 %tour&disk
  \draw (-4,0.5) rectangle (4,9.5);
  \draw (0,5) node {$\bullet$} -- (4,5);
  \draw[->] (2.5,7.5) -- (2.4,7.6) node[right] {$L_1$};
  \draw[->] (1.2,4.2) -- (1.2,4.3); \draw (1.05,4.3) node[right] {$L_2$};
 \end{scope}
 \end{tikzpicture}
 \end{center}
 \caption{Stacking admissible diagrams.} \label{figstack}
\end{figure}

\section{Universality}
\label{secuniversality}

We want to describe the behaviour of the invariant $\tZ$ under null LP--surgeries. For this, we fix an abstract Blanchfield module $(\Al,\bl)$ and we restrict to $\Q\SK$--pairs whose Blanchfield module is isomorphic to $(\Al,\bl)$. For such a $\Q\SK$--pair $(M,K)$, in order to see $\tZ(M,K)$ in the diagram space $\A(\Al,\bl)$, we need to fix an isomorphism from the Blanchfield module of $(M,K)$ to~$(\Al,\bl)$. However, the relation Aut implies that the value of $\tZ(M,K)\in\A(\Al,\bl)$ does not depend on the chosen isomorphism, so that we will ignore it in the sequel.

\subsection{Preliminaries: the LMO invariant}
\label{subsecLMO}

We recall here some properties of the LMO invariant we will need below. This invariant $Z^\LMO$ of $\Q$--spheres, constructed by Le--Murakami--Ohtsuki in \cite{LMO}, is valued in the graded space $\A(\emptyset)$ of trivalent diagrams with oriented trivalent vertices, quotiented out by the relations AS and IHX. The degree of a diagram is the number of its vertices; in particular, $\A_n=0$ when $n$ is odd. Finiteness properties for this invariant were established by Le \cite{Le} with respect to borromean surgeries and generalized by Massuyeau \cite{Mas} to LP--surgeries. It follows that the LMO invariant induces a map on the graded space $\G$ associated to finite type invariants of $\Q$--spheres with respect to LP--surgeries. Further, the map $\varphi:\A(\emptyset)\to\G$ constructed in \cite{GGP} is be defined as in Subsection~\ref{subsecvarphi} (without univalent vertices to deal with).

\begin{theorem}[\cite{Le,Mas}] \label{thLMO}
 The LMO invariant induces a map $Z^\LMO:\G\to\A(\emptyset)$ and the composition $Z^\LMO\circ\varphi$ is the identity on $\A(\emptyset)$.
\end{theorem}

\subsection{Elementary surgeries}

To understand the behaviour of our invariant under null LP--surgeries, we will work on a restricted set of surgeries which generate all of them. 

Given a positive integer $d$, we define a \emph{$d$--torus} as a $\Q$--torus $T_d$ satisfying, for some simple closed curves $\alpha$ and $\beta$ on $\partial T$: 
\begin{itemize}\itemsep=0cm
 \item $H_1(\partial T_d;\mathbb{Z})=\mathbb{Z} \alpha \oplus \mathbb{Z}\beta$, with $\langle\alpha,\beta\rangle=1$,
 \item $d\alpha=0$ in $H_1(T_d;\mathbb{Z})$,
 \item $\beta=d\gamma$ in $H_1(T_d;\mathbb{Z})$, where $\gamma$ is a curve in $T_d$,
 \item $H_1(T_d;\mathbb{Z})=\fract{\Z}{d\Z} \alpha \oplus \mathbb{Z} \gamma$.
\end{itemize}
We define a {\em $d$--surgery} as an LP--replacement of a solid torus by a $d$--torus. Finally, we define an \emph{elementary surgery} as an LP--surgery among the following ones:
\begin{itemize}\itemsep=0cm
 \item connected sum (genus 0),
 \item $d$--surgery (genus 1),
 \item borromean surgery (genus 3).
\end{itemize}

\begin{remark}
In terms of diagrams, the borromean surgeries will correspond to Jacobi diagrams, while connected sums will correspond to isolated vertices. Although it does not seem possible to remove the genus--$1$ elementary surgeries, they don't have a diagrammatic counterpart. 
\end{remark}

\begin{theorem}[\cite{M2} Theorem 1.15] \label{thelsur}
 If $A$ and $B$ are two $\mathbb{Q}$--handlebodies with LP--identified boundaries, then $B$ can be obtained from $A$ by a finite sequence of elementary surgeries and their inverses in the interior of the $\mathbb{Q}$--handlebodies.
\end{theorem}

In the proof of this theorem, arbitrary $d$--tori are used, so that we can reduce the genus--$1$ elementary surgeries to that defined by a fixed $d$--torus for each positive integer $d$. Here we will use the $d$--torus obtained from a standard solid torus by Dehn surgery on the link $J_1\cup J_2$ described in Figure~\ref{figdtorus}; we denote it by $T_d$ in the sequel. Note that $T_1$ is the standard solid torus.

\begin{figure}[htb] 
\begin{center}
\begin{tikzpicture}
 \foreach \s in {-1,1} 
 \draw (0,0) .. controls +(0,\s) and +(-1,0) .. (3,1.5*\s) .. controls +(1,0) and +(0,\s) .. (6,0);
 \draw[yshift=0.3cm] (2,0) ..controls +(0.5,-0.25) and +(-0.5,-0.25) .. (4,0);
 \draw[yshift=0.3cm] (2.3,-0.1) ..controls +(0.6,0.2) and +(-0.6,0.2) .. (3.7,-0.1);
 \foreach \x in {0,1,...,4} 
 \draw[rounded corners=5pt,thick] (1.6+0.6*\x,-0.6) -- (1.8+0.6*\x,-0.4) -- (2.1+0.6*\x,-0.7) (1.5+0.6*\x,-0.7) -- (1.8+0.6*\x,-1) -- (2+0.6*\x,-0.8);
 \draw[thick] (1.5,-0.7) .. controls +(-0.4,0.4) and +(-2,0) .. (3,-0.1) .. controls +(2,0) and +(0.4,0.4) .. (4.6,-0.6);
 \draw[thick] (1.4,-0.8) .. controls +(-0.6,-0.6) and +(-3,0) .. (3,1) .. controls +(3,0) and +(0.6,-0.6) .. (4.5,-0.7);
 \draw (3,-1.2) node {\small{$2d$ crossings}};
 \draw (4.5,0.1) node {$J_2$} (5.2,0.5) node {$J_1$};
\end{tikzpicture}
\caption{A $d$--torus constructed by Dehn surgery} \label{figdtorus}
\end{center}
\end{figure}

\subsection{Behaviour of $\tZ$ with respect to elementary surgeries}

A first step is to describe the behaviour of $\tZ$ under connected sum. The idea is the same as in Lemma~\ref{lemma:connectedsum}, but we now connect-sum with a $\Q$--sphere instead of a $\Q\SK$--pair. 

\begin{lemma} \label{lemma:elsurgenus0}
 Let $(M,K)$ be a $\Q\SK$--pair with Blanchfield module $(\Al,\bl)$. Let $N$ be a $\Q$--sphere. The invariant $\tZ$ satisfies:  
 $$\tZ(M\sharp N,K)=\tZ(M,K)\sqcup Z^\LMO(N)\quad\in\A(\Al,\bl).$$
\end{lemma}
\begin{proof}
 If $L$ and $J$ are surgery links for $(M,K)$ and $N$ respectively, a surgery link for $(M\sharp N,K)$ is obtained by stacking $L$ and $J$, see Figure~\ref{figstackbis}. At each step of the construction of the invariant~$\tZ$, we have a disjoint union of two series of diagrams associated to $L$ and $J$ respectively. Moreover, the winding matrix of $L\sqcup J$ is a bloc diagonal matrix with blocs $W_L$ and $W_J$. Since the components of $J$ do not meet the disk bounded by the unknot of the surgery presentation, $W_J$ has all its coefficients in $\Z$. This implies that the Blanchfield module of $(M\sharp N,K)$ is again $(\Al,\bl)$. Further, the diagrams in $\tZ(M\sharp N,K)$ coming from $J$ have all their univalent vertices labeled by zero, so that these univalent vertices can be removed using the relation LV. By construction, this part of $\tZ(M\sharp N,K)$ coming from $J$ is the LMO invariant of $N$ computed with the Aarhus method.
\end{proof}
 
\begin{figure}[htb] 
 \begin{center}
 \begin{tikzpicture} [scale=0.5]
 \begin{scope} [scale=0.9] %Whitehead
  \draww{(0,3) arc (-90:90:2);}
  \draww{(0,2) arc (-90:90:3);}
  \draww{(0,8) .. controls  +(-4,0) and +(-1,0) .. (-2.5,4.5);}
  \draww{(0,3) .. controls  +(-3,0) and +(1,0) .. (-2.5,5.5);}
  \draww{(0,2) .. controls  +(-4,0) and +(-1,0) .. (-2.5,5.5);}
  \draww{(0,7) .. controls  +(-3,0) and +(1,0) .. (-2.5,4.5);}
 \end{scope}
 \begin{scope} [xshift=-6cm,yshift=2cm,scale=0.7] %huit
  \draw (0,7.5) .. controls +(1,0) and +(0,1) .. (1,6) .. controls +(0,-2) and +(0,1) .. (-1,3);
  \draw (0,6) .. controls +(-1,0) and +(0,1) .. (-2,4) .. controls +(0,-4) and +(0,-3) .. (1,3);
  \draww{(0,7.5) .. controls +(-1,0) and +(0,1) .. (-1,6) .. controls +(0,-2) and +(0,1) .. (1,3);}
  \draww{(0,6) .. controls +(1,0) and +(0,1) .. (2,4) .. controls +(0,-4) and +(0,-3) .. (-1,3);}
 \end{scope}
 %tour&disk
  \draw (-8,1) rectangle (4,9);
  \draw[dashed] (-4,1) -- (-4,9);
  \draw (0,5) node {$\bullet$} -- (4,5);
  \draw[->] (2.24,6) -- (2.14,6.1) node[right] {$L$};
  \draw[->] (-6.1,7.25) -- (-6,7.25) node[above] {$J$};
 \end{tikzpicture}
 \end{center}
 \caption{Stacking diagrams for a $\Q$--sphere and a $\Q\SK$--pair} \label{figstackbis}
\end{figure}

Our second step is to describe the behaviour of $\tZ$ under $d$--surgeries. 

\begin{proposition} \label{propelsurgenus1}
 Let $(M,K)$ be a $\Q\SK$--pair. Fix a positive integer $d$. Consider a $d$--surgery $\left(\frac{T_d}{T_1}\right)$ on $(M,K)$. Denote $J_1\sqcup J_2^d$ the surgery link defined on Figure~\ref{figdtorus}. Let $L$ be an admissible surgery presentation of $(M,K)$. Then $L_1=L\sqcup J_1\sqcup J_2^1$ and $L_d=L\sqcup J_1\sqcup J_2^d$ are admissible surgery presentations of $(M,K)$ and $(M,K)\left(\frac{T_d}{T_1}\right)$ respectively, and $\tZ\left((M,K)\left(\frac{T_d}{T_1}\right)\right)-\tZ(M,K)$ is a series of diagrams of degree at least $1$ containing a univalent vertex associated to $J_2$.
\end{proposition}
\begin{proof}
 The winding matrix of $L_d$ is of the form
 $W_d=\begin{pmatrix} W_L & \zeta & 0 \\ ^t\bar\zeta & \lambda & d \\ 0 & d & 0 \end{pmatrix}$, where $W_L$ is the winding matrix of $L$ and $\zeta,\lambda$ do not depend on $d$. Hence the Blanchfield form of the pair $(M,K)_d=(M,K)\left(\frac{T_d}{T_1}\right)$ is given by the matrix 
 $W_d^{-1}=\begin{pmatrix} W_L^{-1} & 0 & \eta_d \\ 0 & 0 & \frac1d \\ ^t\bar\eta_d & \frac1d & \mu_d \end{pmatrix}$, where $\eta_d=-\frac1dW_L^{-1}\zeta$ and $\mu_d=\frac1{d^2}\left(^t\bar\zeta W_L^{-1}\zeta-\lambda\right)$. If $x_1^d,\dots,x_n^d,y_1^d,y_2^d$ are the associated generators of the Alexander module, then an automorphism $f_d$ between the Blanchfield modules of $(M,K)$ and $(M,K)_d$ is given by \mbox{$x_i^1\mapsto x_i^d$,} $y_1^1\mapsto y_1^d$ and $y_2^1\mapsto dy_2^d$; we may note that $y_1^d=0$ and $y_2^d$ is a $\Qtt$--linear combination of the~$x_i^d$.
 
 \begin{figure}[htb] 
\begin{center}
\begin{tikzpicture} 
\begin{scope} [rotate=-90]
 \foreach \x in {0,1,...,4} 
 \draw[rounded corners=5pt,thick] (1.6+0.6*\x,-0.6) -- (1.8+0.6*\x,-0.4) -- (2.1+0.6*\x,-0.7) (1.5+0.6*\x,-0.7) -- (1.8+0.6*\x,-1) -- (2+0.6*\x,-0.8);
 \draw[thick,->] (1.5,-0.7) .. controls +(-0.4,0.4) and +(-2,0) .. (3,0) node[right] {$J_2$};
 \draw[thick] (3,0) .. controls +(2,0) and +(0.4,0.4) .. (4.6,-0.6);
 \draw[thick,->] (1.4,-0.8) .. controls +(-0.2,-0.2) and +(0.2,0) .. (1,-1) (4.5,-0.7) .. controls +(0.2,-0.2) and +(-0.2,0) .. (5,-1);
 \draw[thick,dashed] (1,-1) -- (0.5,-1) (5,-1) -- (5.5,-1);
 \draw (3,-2) node {\small{$2d$ crossings}};
 \draw (5,-1.3) node {$J_1$};
\end{scope}
\begin{scope} [xshift=3cm,yshift=-3.5cm,scale=0.6]
 \draw[gray!50,line width=3pt] (0,0) rectangle (3.5,2);
 \draw[->,very thick] (2,3) -- (2,-1);
 \draw[->,very thick] (3,3) -- (3,-1);
 \draw (2,1) -- (3,1);
 \draw (1,1) node {$\expd d$};
\end{scope}
\end{tikzpicture}
\caption{The local difference on the surgery presentation and on $\Zp$} \label{figsurgerypres}
\end{center}
\end{figure}

 In the computation of $\Zp(L_d)$, the only part depending on $d$ occurs at the level of the crossings between $J_1$ and $J_2$, which can be presented as in the left hand side of Figure~\ref{figsurgerypres}; the right hand side represents its contribution to $\Zp(L_d)$. When applying $\chi^{-1}$ to the right hand side of Figure~\ref{figsurgerypres}, we get some struts which contribute to the winding matrix $W_d$ and ``correction terms'' with at least one univalent vertex labeled by $J_2$ and joined to a trivalent vertex. The latter contribute to $H_d$ in $\overline{\Zc(L_d)}=\expd\left(\frac12 W_d\right)\sqcup H_d$. We get then $\tZ\big((M,K)_d\big)=\omega_{W_d}(H_d)$. To compare $\tZ\big((M,K)_d\big)$ to $\tZ(M,K)$, we apply the above automorphism $f_d$ to $\omega_{W_1}(H_1)$. The diagrams in the difference $\tZ\big((M,K)_d\big)-\tZ(M,K)$ come from the above correction terms; they have degree at least one and one univalent vertex labeled by $y_2^d$. 
\end{proof}

The last step is to describe the behaviour of $\tZ$ under borromean surgeries. Once again, we have to deal with a local modification of the surgery presentation. To a Y--graph is associated a six-component surgery link, which can be turned to a trivial surgery link by separating the central three components, see Figure~\ref{figborromeandiff}. The effect of such a difference on the invariant $Z$ has been computed by Le in \cite{Le}.

\begin{figure}[htb] 
\begin{center}
\begin{tikzpicture} [scale=0.15]
\begin{scope}[xshift=1000]
 \newcommand{\bras}[1]{
 \draw[rotate=#1] (3,-8) arc (0:180:3);
 \draw[rotate=#1,dashed] (3,-8) arc (0:-180:3);
 \draw[rotate=#1,white,line width=6pt] (-1,-5) -- (-1,-5.7);
 \draw[rotate=#1] (-1,-4.6) -- (-1,-5.7) arc (-180:0:1) (-1,-4.6) arc (180:0:1);}
 \bras{0}
 \bras{120}
 \bras{-120}
\end{scope}
\begin{scope}
 \newcommand{\bras}[1]{
 \draw[rotate=#1] (0,-1.5) circle (2.5);
 \draw [rotate=#1,white,line width=8pt] (-0.95,-4) -- (0.95,-4);
 \draw[rotate=#1] (1,-3.9) -- (1,- 4.6);
 \draw[rotate=#1] (3,-8) arc (0:180:3);
 \draw[rotate=#1,dashed] (3,-8) arc (0:-180:3);
 \draw[rotate=#1,white,line width=6pt] (-1,-5) -- (-1,-5.7);
 \draw[rotate=#1] {(-1,-3.9) -- (-1,-5.7) arc (-180:0:1)};}
 \bras{0}
 \draw [white,line width=6pt,rotate=120] (0,-1.5) circle (2.5);
 \bras{120}
 \draw [rotate=-120,white,line width=6pt] (-1.77,0.27) arc (135:190:2.5);
 \draw [rotate=-120,white,line width=6pt] (1.77,0.27) arc (45:90:2.5);
 \bras{-120}
 \draw [white,line width=6pt] (-1.77,0.27) arc (135:190:2.5);
 \draw [white,line width=6pt] (1.77,0.27) arc (45:90:2.5);
 \draw (-1.77,0.27) arc (135:190:2.5);
 \draw (1.77,0.27) arc (45:90:2.5);
\end{scope}
\end{tikzpicture}
\end{center}
\caption{Surgery link associated to a Y--graph and corresponding trivial surgery}\label{figborromeandiff}
\end{figure}

\begin{theorem}[Le] \label{thLe}
$$Z\left(\raisebox{-1cm}{
\begin{tikzpicture} [scale=0.7]
 \foreach \t in {0,120,240} 
 \draw[rotate=\t] (-1,-1.5) .. controls +(0,0.5) and +(0,-0.3) .. (-0.5,-0.4) .. controls +(0,0.1) and +(-0.2,-0.1) .. (-0.35,0.1) (-0.1,0.2) .. controls +(0.3,0) and +(0,0.3) .. (0.5,-0.2) (0.5,-0.4) .. controls +(0,-0.3) and +(0,0.5) .. (1,-1.5);
\end{tikzpicture}}
\right)-Z\left(\raisebox{-1cm}{
\begin{tikzpicture} [scale=0.7]
 \foreach \t in {0,120,240} {
 \begin{scope} [rotate=\t]
  \draw (-1,-1.5) .. controls +(0,0.5) and +(-0.8,0) .. (0,-1) .. controls +(0.8,0) and +(0,0.5) .. (1,-1.5);
 \end{scope}}
\end{tikzpicture}}
\right)=\raisebox{-1cm}{
\begin{tikzpicture} [scale=0.7]
 \foreach \t in {0,120,240} {
 \begin{scope} [rotate=\t]
  \draw (0,0) -- (0,-1);
  \draw[very thick] (-1,-1.5) .. controls +(0,0.5) and +(-0.8,0) .. (0,-1) .. controls +(0.8,0) and +(0,0.5) .. (1,-1.5);
 \end{scope}}
\end{tikzpicture}}
+\text{ higher degree terms}
$$
\end{theorem}

\subsection{Finiteness properties of $\tZ$}

The above computations provide us with the finiteness properties we seeked for the invariant~$\tZ$.

\begin{theorem} \label{thtZFTI}
 The degree--$n$ part $\tZ_n$ of the invariant $\tZ$ is a degree--$n$ finite type invariant of $\Q\SK$--pairs with respect to null LP--surgeries. Moreover, $\tZ_n$ vanishes on any order--$n$ bracket containing a genus--$0$ elementary surgery.
\end{theorem}
\begin{proof}
 We proceed by induction on $n$. The image of $\tZ_0$ on any $\Q\SK$--pair is the empty diagram with coefficient $1$, so that $\tZ_0(\F_1)=0$. Fix $n>0$. 
 It is easily deduced from Theorem~\ref{thelsur} that $\F_{n+1}$ is generated by the brackets defined by elementary surgeries (see \cite[Corollary 5.5]{M7}). 
 Consider a bracket $[(M,K);c_1,\dots,c_n]$ where the surgery $c_n$ is the connected sum with some $\Q$--sphere $N$. We have
 $$[(M,K);c_1,\dots,c_n]=[(M,K);c_1,\dots,c_{n-1}]-[(M\sharp N,K);c_1,\dots,c_{n-1}],$$
 so that, by Lemma~\ref{lemma:elsurgenus0},
 $$\tZ\big([(M,K);c_1,\dots,c_n]\big)=\tZ\big([(M,K);c_1,\dots,c_{n-1}]\big)\sqcup\big(\emptyset-Z^\LMO(N)\big).$$
 In the right hand side of the latter equality, the first term contains only diagrams of degree at least $n-1$, by induction, and the second term contains only diagrams of degree at least $2$ (see Theorem~\ref{thLMO}). Hence $\tZ\big([(M,K);c_1,\dots,c_n]\big)$ is made of diagrams of degree at least $n+1$. 
 Further, if an order--$n$ bracket is defined by elementary surgeries of genus $1$ and $3$, the local contributions explicited in Proposition~\ref{propelsurgenus1} and Theorem~\ref{thLe} combine, so that the image of this bracket by $\tZ$ contains only diagrams of degree at least $n$. 
\end{proof}

This result implies that $\tZ$ defines a map $\tZ:\G(\Al,\bl)\to\A(\Al,\bl)$.

We shall prove that $\tZ_n$ also vanishes on order--$n$ brackets containing a genus--$1$ elementary surgery. Given a $\Q$--torus $T$, a {\em meridian} of $T$ is a simple closed curve on $\partial T$ that generates the Lagrangian of $T$; it is well-defined up to isotopy. A {\em longitude} of $T$ is a simple closed curve on $\partial T$ that intersects the meridian exactly once. 
A {\em framed $\Q$--torus} is a $\Q$--torus with a fixed oriented longitude. Note that any two framed $\Q$--tori have a canonical LP--identification of their boundaries, which identifies the fixed longitudes. We define finite type invariant of framed $\Q$--tori with respect to LP--surgeries as we defined finite type invariants of $\Q\SK$--pairs with respect to null LP--surgeries. 

\begin{proposition}[\textup{\cite[Corollary 5.10]{M2}}] \label{cortori}
 For each prime integer $p$, let $M_p$ be a $\Q$--sphere such that $|H_1(M_p;\Z)|=p$. If $\mu$ is a degree 1 invariant of framed $\Q$--tori, such that $\mu(T_0)=0$ and $\mu(T_0\sharp M_p)=0$ for any prime $p$, then $\mu=0$.
\end{proposition}

\begin{corollary} \label{corZgenus1surgery}
 For all $n>0$, the invariant $\tZ_n$ vanishes on any order--$n$ bracket containing a genus--$1$ elementary surgery.
\end{corollary}
\begin{proof}
 Consider a bracket $[(M,K);c_1,\dots,c_n]$ where $c_n=\left(\frac{T_d}{T_0}\right)$ is a genus--$1$ elementary surgery. Fix an oriented longitude of $T_0$. For any framed $\Q$--torus $T$, set 
 $$\lambda(T)=\tZ_n\left(\left[(M,K);c_1,\dots,c_{n-1},\frac{T}{T_0}\right]\right).$$ Then $\lambda$ is a degree 1 invariant of framed $\Q$--tori: if $s_1,s_2$ are two disjoint LP--surgeries on $T$, 
 $$\lambda([T;s_1,s_2])=\tZ_n\left(-\left[(M,K)\left(\frac{T}{T_0}\right);c_1,\dots,c_{n-1},s_1,s_2\right]\right)= 0.$$
 We have $\lambda(T_0)=0$. Moreover, if $M_p$ is a $\Q$--sphere such that $|H_1(M_p;\Z)|=p$ and $B_p$ is the $\Q$--ball obtained from $M_p$ by removing an open $3$--ball, then 
 $$\lambda(T_0\sharp M_p) = \tZ_n\left( \left[(M,K);c_1,\dots,c_{n-1},\frac{B_p}{B^3}\right]\right) =0,$$
 thanks to Theorem~\ref{thtZFTI}. Finally, by Proposition~\ref{cortori}, $\lambda=0$.
\end{proof}

We finally explicit the behaviour of $\tZ_n$ on brackets defined by borromean surgeries. We start with a standard lemma.

\begin{lemma} \label{lemmaTrivialSurgery}
 Let $J_1$ and $J_2$ be disjoint knots in a $3$--manifold $M$ such that $J_2$ is a meridian of $J_1$ with $0$ framing. Then the surgery on $J_1\sqcup J_2$ preserves the manifold $M$ and a meridian of $J_1$, and changes a meridian of $J_2$ into a longitude of $J_1$.
\end{lemma}
\begin{proof}
 Let $\nu(J_1)$ and $\nu(J_2)$ be tubular neighborhoods of $J_1$ and $J_2$ such that $\nu(J_1)\cap\nu(J_2)$ is an annulus whose core is simultaneously a meridian of $J_1$ and a longitude of $J_2$. The union $T=\nu(J_1)\cup\nu(J_2)$ is a solid torus which shares a meridian with $J_1$. For $i=1,2$, the surgery replaces $\nu(J_i)$ by a torus $T_i$ whose meridian is a former longitude of $J_i$. Hence the core of $\nu(J_1)\cap\nu(J_2)$ is a longitude of $T_1$ and a meridian of $T_2$, so that $T'=T_1\cup T_2$ is again a solid torus with the same meridian as~$T$. Further, a former meridian of $J_2$ can be slid over a meridian disk of~$T_1$, so that it is isotopic to a former longitude of $J_1$.
\end{proof}

\begin{proposition} \label{proptZcircphi}
 For all $n\geq0$, the composition $\tZ_n\circ\varphi_n$ is the identity of $\A_n(\Al,\bl)$.
\end{proposition}
\begin{proof}
 It is enough to prove it on the elementary diagrams introduced in Subsection~\ref{subsecvarphi} to define the map $\varphi$. Let $D$ be such an elementary diagram, of degree $n$. The trivalent edges can be cut open using the relation LD to insert two univalent vertices labeled by zero (with the notations of Figure~\ref{figrelLVEVLD}, one can set $f_{v_1v_2}^{D'}=0$, so that the diagram $D'$ is trivial), thus we can assume that $D$ is a union of graphs \raisebox{-3ex}{
 \begin{tikzpicture} 
 \foreach \t/\i/\p in {0/3k-2/right,120/3k-1/right,240/3k/above} 
 \draw[rotate=\t] (0,0) -- (0,-0.5) node[\p] {$\gamma_{\scriptscriptstyle{\i}}$};
 \end{tikzpicture}} for $k=1,\dots,n$, with fixed linkings $f_{ij}$ between the vertices labeled $\gamma_i$ and $\gamma_j$ respectively. 
 Let $\Gamma$ be an associated null Y--link in $M\setminus K$ following the rules of Subsection~\ref{subsecvarphi}. The $6n$ components  associated to $\Gamma$ in a surgery link are the $3n$ leaves $\ell_i$ corresponding to the labels $\gamma_i$ and $3n$ components $k_i$ corresponding to the edges of $\Gamma$, with coherent numbering. 
 Thanks to Theorem~\ref{thLe}, in $\tZ_n\circ\varphi_n(D)$, the only degree--$n$ term is the union of the $n$ graphs \raisebox{-3ex}{
 \begin{tikzpicture} 
 \foreach \t/\i/\p in {0/3k-2/right,120/3k-1/right,240/3k/above} 
 \draw[rotate=\t] (0,0) -- (0,-0.5) node[\p] {$x_{\scriptscriptstyle{\i}}$};
 \end{tikzpicture}}, 
 where $x_i$ is the generator of $\Al$ associated with $k_i$. 
 By definition of the operation $\omega$, $x_i$ equals in~$\Al$ the homology class of a meridian $m(\tilde k_i)$ of a lift $\tilde k_i$ in the infinite cyclic covering of $M\setminus K$, and the fixed linking associated to the vertices labeled $x_i$ and $x_j$ is the equivariant linking $\lk_e\big(m(\tilde k_i),m(\tilde k_j)\big)$. Here, we work in $(M,K)$, so that the $k_i$ are as represented on the right hand side of Figure~\ref{figborromeandiff}. Hence we can apply Lemma~\ref{lemmaTrivialSurgery} to conclude that $m(\tilde k_i)$ is isotopic to $\tilde\ell_i$, so that $x_i=\gamma_i$ and $\lk_e\big(m(\tilde k_i),m(\tilde k_j)\big)=f_{ij}$. 
 Eventually, we get $\tZ_n\circ\varphi_n(D)=D$.
\end{proof}

\subsection{Augmented invariant and universality}

The degree--$1$ invariants $\rho_p$ can be merged with $\tZ$ into a universal finite type invariant of $\Q\SK$--pairs by setting, for a $\Q\SK$--pair $(M,K)$: $$\tZ^\aug(M,K)=\tZ(M,K)\sqcup\expd\left(\sum_{p\text{ prime}}\rho_p(M)\right).$$

The following formula is classical and holds for any objects and any invariants with values in some ring (see for instance \cite[Lemma 6.2]{M2}).
$$\left(\prod_{j=1}^n\lambda_j\right)\big(\left[\left(M,K\right);(c_i)_{i\in I}\right]\big)
=\sum_{\emptyset=J_0\subset\dots\subset J_n=I}\prod_{j=1}^n\lambda_j\left(\left[(M,K)\left(\left(c_i\right)_{i\in J_{j-1}}\right);\left(c_i\right)_{i\in J_j\setminus J_{j-1}}\right]\right)$$
It follows that the degree of finite type invariants is subadditive under multiplication. Moreover, the degree--$n$ part of $\tZ^\aug$ is given by 
$$\tZ_n^\aug=\sum_{k=0}^n\sum_{\substack{p_1<\dots<p_s \\ \textrm{ prime integers }}}
  \sum_{\substack{t_1+\dots+t_s=n-k\\t_i>0}}\tZ_k \sqcup\left(\coprod_{i=1}^s\frac{1}{t_i!}(\rho_{p_i})^{t_i}\right).$$
We deduce that $\tZ_n^\aug$ is a finite type invariant of degree $n$, so that the invariant $\tZ^\aug$ induces a map $\tZ^\aug:\G(\Al,\bl)\to\A^\aug(\Al,\bl)$.

\begin{theorem} \label{thtZauguniv}
 The invariant $\tZ^\aug$ induces the inverse of the map $\varphi:\A^\aug(\Al,\bl)\to\G(\Al,\bl)$.
\end{theorem}

\begin{proof}
 We know from \cite[Theorem 2.7]{M7} that $\varphi$ is surjective, so that it is enough to prove that $\tZ^\aug\circ\varphi$ is the identity. Let $D$ be an $(\Al,\bl)$--augmented diagram of degree~$n$. Write $D$ as the disjoint union of its Jacobi part $D_J$ 
 and its $0$--valent part $D_{\bullet}$. 
 Apply the above formula, noting that for a term in the right hand side of the obtained equality to be non trivial:
 \begin{itemize}
  \item the order of each bracket must be exactly the degree of the corresponding invariant, 
  \item each invariant $\rho_p$ must be evaluated on a bracket associated to the diagram $\bullet_p$,
  \item the invariant $\tZ_k$ must be evaluated on a bracket associated to a diagram without isolated vertices.
 \end{itemize}
 It follows that $\tZ_n^\aug\circ\varphi_n(D)=\big(\tZ_k\circ\varphi_k(D_J)\big)\sqcup D_{\bullet}=D_J\sqcup D_{\bullet}=D$, where the third equality is due to Proposition~\ref{proptZcircphi}.  
\end{proof}

\subsection{Kricker invariant {\em versus} Lescop invariant}

The description of the graded space $\G(\Al,\bl)$ allows us to prove that the Kricker invariant and the Lescop invariant induce the same map on~$\G(\Al,\bl)$. 

\begin{proposition} \label{prop:eqKriLes}
 For a given Blanchfield module $(\Al,\bl)$ with annihilator $\delta$, the map induced by both the Kricker invariant and the Lescop invariant on the graded space $\G(\Al,\bl)$ is the map $\psi\circ\tZ:\G(\Al,\bl)\to\A(\delta)$.
\end{proposition}
\begin{proof}
 If $Z$ stands for $Z^\Kri$ or $Z^\Les$, then $Z\circ\varphi=\psi:\A(\Al,\bl)\to\A(\delta)$ (see Theorem~\ref{thinvariantZ}). Further, $Z$ is multiplicative under connected sum, so that $Z_n$ vanishes on order--$n$ brackets that contain a genus--$0$ elementary surgery (by the same argument as in Theorem~\ref{thtZFTI}). Thus $Z\circ\varphi=\widehat\psi:\A^\aug(\Al,\bl)\to\A(\delta)$, where we define $\widehat\psi:\A^\aug(\Al,\bl)\to\A(\delta)$ as follows: for an augmented $(\Al,\bl)$--colored diagram $D$, $\widehat\psi(D)=0$ if $D$ contains an isolated vertex and $\widehat\psi(D)=\psi(D)$ otherwise. To conclude, we note that $\widehat\psi\circ\tZ^\aug=\psi\circ\tZ$. 
\end{proof}

Here, we are able to explicit and compare $Z^\Kri$ and $Z^\Les$ on the graded space $\G(\Al,\bl)$. In the case of $Z^\Kri$, our refinement $\tZ$ gives an expression of $Z^\Kri$ on the whole $\F_0(\Al,\bl)$ as the composition of the universal invariant $\tZ$ and the explicit map $\psi$ (Proposition~\ref{proptZKri}). A similar refinement of $Z^\Les$ would give an alternative construction of a universal invariant and would enable us to compare more deeply the Kricker and Lescop invariants.

\section{Knots in $\Z$--spheres}
\label{secZpairs}

In this section, we describe the differences that appear when we restrict our study to $\Z\SK$--pairs. Let $\F_0^\z$ be the $\Q$--vector space generated by all $\Z\SK$--pairs up to orientation-preserving homeomorphism. In the work of Garoufalidis--Kricker--Rozansky, the considered surgery move is the null borromean surgery (called null-move in \cite{GR}). One can also consider null $\Z$LP--surgeries, defined as the null LP--surgeries using $\Z$--handlebody and the $\Z$--homology. It turns out that these two moves define the same filtration of $\F_0^\z$: Auclair and Lescop proved that any null $\Z$LP--surgery can be realized by a borromean surgery on a null Y--link \cite[Lemma~4.11]{AL}. We denote $\G^\z$ the associated graded space. 

To a $\Z\SK$--pair $(M,K)$, we associate the {\em integral Alexander module} $\Al_\z(M,K)$ defined as the $\Ztt$ module $H_1(\tilde X;\Z)$, with the notations of Section~\ref{secbackground}. The Blanchfield form is defined on $\Al_\z$ similarly. We may note that $\Al(M,K)=\Al_\z(M,K)\otimes_\Z\Q$ and $\Al_\z(M,K)$ has no $\Z$--torsion (see \cite[Lemma~5.5]{M3}). Once again, the classes of $\Z\SK$--pairs up to null $\Z$LP--surgeries are characterized by the isomorphism classes of integral Blanchfield modules. Hence our filtration splits along the isomorphism classes of integral Blanchfield modules. We shall fix an abstract integral Blanchfield module $(\Al_\z,\bl)$ and consider the associated graded space $\G^\z(\Al_\z,\bl)$.

At the level of diagram spaces, we again use rational coefficients and univalent vertices labeled in the rational Alexander module; the only difference occurs in the definition of the relation Aut. On $(\Al,\bl)$--colored diagrams, the relation $Aut^\z$ is reduced to automorphisms of the Blanchfield module that are induced by automorphisms of the integral Blanchfield module. This is restrictive in general (see \cite[remark after Proposition~7.9]{M7}), so that the associated diagram space $\A^\z(\Al_\z,\bl)$ might be richer. 
At the level of invariants, we consider the same invariants $Z^\Kri$ and $Z^\Les$, and the invariant $\tZ^\z$ is defined using the same construction, but with values in $\A^\z(\Al_\z,\bl)$. Note that there is no ``augmented'' diagrams or invariants in this setting.

In \cite[Theorem~2.17]{M7}, a canonical surjective $\Q$--linear map $\varphi^\z:\A^\z(\Al_\z,\bl)\to\G^\z(\Al_\z,\bl)$ is constructed.

\begin{theorem} \label{thZZuniv}
 The invariant $\tZ^\z$ of $\Z\SK$--pairs induces the inverse of the map $\varphi^\z:\A^\z(\Al_\z,\bl)\to\G(\Al_\z,\bl)$.
\end{theorem}

Once again, the Kricker invariant can be recovered from $\tZ^\z$: $Z^\Kri=\psi\circ\tZ^\z$. Further, both the Kricker and the Lescop invariants induce the same map on $\G^\z(\Al_\z,\bl)$, namely the map $\psi\circ\tZ^\z$.

\appendix

\section{Corollaries of the multinomial formula}

For integers $k>0$ and $k_1,\dots,k_p\geq0$, a multinomial coefficient is defined as $\binom{k}{k_1,\dots,k_p}=\prod_{i=1}^p\binom{\sum_{j=1}^ik_j}{k_i}=\frac{k!}{\prod_{i=1}^pk_i}$. It fits into the multinomial formula: 
$$\left(\sum_{i=1}^px_i\right)^k=\sum_{k_1+\dots+k_p=k}\binom{k}{k_1,\dots,k_p}\prod_{i=1}^px_i^{k_i},$$
which, for $x_i=1$, gives $\displaystyle\sum_{k_1+\dots+k_p=k}\binom{k}{k_1,\dots,k_p}=p$.
From this, one can deduce the following generalization for all $a,k>0$ and $b\geq0$ (proceed by induction on $p>0$ and $k\in\{0,\dots,p\}$).
\begin{equation} \label{formuleMult}
 \sum_{\substack{k_1+\dots+k_p=k\\k_1,\dots,k_{p-q}\geq0\\k_{p-q+1},\dots,k_p>0}}\binom{k}{k_1,\dots,k_p}=\sum_{j=0}^q(-1)^j\binom qj(p-j)^k
\end{equation}

We will also need the following formula.
\begin{lemma} \label{lemmabinom}
 For $a>b>0$, $\displaystyle\sum_{j=1}^a(-1)^j\binom ajj^b=0$.
\end{lemma}
\begin{proof}
 We claim that for $a>0$ and $b\in\{1,\dots,a\}$, there exist polynomial functions $f_{a,b}$ of degree~$b$ such that, for all $x\in\R$, $\sum_{j=1}^a(-1)^j\binom ajj^bx^j=f_{a,b}(x)\,(1-x)^{a-b}$. To see this, start with the binomial formula $\sum_{j=0}^a(-1)^j\binom ajx^j=(1-x)^{a}$, differentiate it and then multiply both sides by $x$. This gives the case $b=1$, and the induction step is obtained in the same way. Now evaluating at $x=1$ gives the required equality.
\end{proof}

\begin{lemma} \label{lemma:calcul}
 For all $k>0$, 
 $\displaystyle \sum_{a=0}^k\sum_{\substack{k_0+\cdots+k_a=k\\k_0\geq0,\ k_1,\ldots,k_a>0}}\frac{(-1)^a}{\prod_{j=0}^ak_j!}=0$.
\end{lemma}
\begin{proof}
 We have
 \begin{align*}
  k!\sum_{a=0}^k\sum_{\substack{k_0+\cdots+k_a=k\\k_0\geq0,\ k_1,\ldots,k_a>0}}\frac{(-1)^a}{\prod_{j=0}^ak_j!}&=\sum_{a=0}^k(-1)^a\sum_{\substack{k_0+\cdots+k_a=k\\k_0\geq0,\ k_1,\ldots,k_a>0}}\binom k{k_0,\dots,k_a}\\
  &=\sum_{a=0}^k(-1)^a\sum_{j=0}^a(-1)^j\binom aj(a-j+1)^k\\
  &=\sum_{a=0}^k\sum_{j=0}^a(-1)^j\binom aj(j+1)^k\\
  &=\sum_{j=0}^k(-1)^j(j+1)^k\sum_{a=j}^k\binom aj\\
  &=\sum_{j=0}^k(-1)^j(j+1)^k\binom{k+1}{j+1}\\
  &=\sum_{j=1}^{k+1}(-1)^{j+1}j^k\binom{k+1}{j}
 \end{align*}
 where the first equality comes from the definition of the binomial coefficient, and the second one from \ref{formuleMult}. This vanishes thanks to Lemma~\ref{lemmabinom}.
\end{proof}

\begin{lemma} \label{lemma:calculbis}
 For all $\ell>1$,
 $\displaystyle \sum_{c=1}^\ell\sum_{\substack{\ell_1+\cdots+\ell_c=\ell\\ \ell_1,\ldots,\ell_c>0}}\frac{(-1)^c}{c\prod_{j=1}^c\ell_j!}=0$.
\end{lemma}
\begin{proof}
 The computation is similar to that of the previous lemma.
 \begin{align*}
  \ell!\sum_{c=1}^\ell\sum_{\substack{\ell_1+\cdots+\ell_c=\ell\\\ell_1,\ldots,\ell_c>0}}\frac{(-1)^c}{\prod_{j=1}^c\ell_j!}&=\sum_{c=1}^\ell\frac{(-1)^c}c\sum_{\substack{\ell_1+\cdots+\ell_c=\ell\\\ell_1,\ldots,\ell_c>0}}\binom \ell{\ell_1,\dots,\ell_c}\\
  &=\sum_{c=1}^\ell\frac{(-1)^c}c\sum_{j=0}^c(-1)^j\binom cj(c-j)^\ell\\
  &=\sum_{c=1}^\ell\sum_{j=1}^c\frac{(-1)^j}c\binom cjj^\ell\\
  &=\sum_{c=1}^\ell\sum_{j=1}^c(-1)^j\binom{c-1}{j-1}j^{\ell-1}\\
  &=\sum_{c=0}^{\ell-1}\sum_{j=0}^c(-1)^{j+1}\binom{c}{j}(j+1)^{\ell-1}\\
  &=\sum_{j=0}^{\ell-1}(-1)^{j+1}(j+1)^{\ell-1}\sum_{c=j}^{\ell-1}\binom cj\\
  &=\sum_{j=0}^{\ell-1}(-1)^{j+1}(j+1)^{\ell-1}\binom{\ell}{j+1}\\
  &=\sum_{j=1}^\ell(-1)^jj^{\ell-1}\binom{\ell}{j}
 \end{align*}
 Once again, this vanishes thanks to Lemma~\ref{lemmabinom}.
\end{proof}

\def\cprime{$'$}
\providecommand{\bysame}{\leavevmode ---\ }
\providecommand{\og}{``}
\providecommand{\fg}{''}
\providecommand{\smfandname}{\&}
\providecommand{\smfedsname}{\'eds.}
\providecommand{\smfedname}{\'ed.}
\providecommand{\smfmastersthesisname}{M\'emoire}
\providecommand{\smfphdthesisname}{Th\`ese}

\end{document}